\documentclass[dvipsnames,11pt,reqno,twoside,final]{amsart}
\usepackage[a4paper,left=3cm,right=3cm,top=3cm,bottom=3cm]{geometry}
\usepackage[T1]{fontenc}
\usepackage[utf8]{inputenc}
\usepackage{xcolor}
\usepackage{amssymb,mathtools}  
\usepackage{eufrak}
\usepackage{yfonts}
\usepackage{dsfont}
\usepackage{enumitem}
\usepackage{subcaption}
\usepackage{centernot}
\usepackage{booktabs}

\usepackage{preamble/mhequ}

\usepackage{amsthm}

\newcounter{counterEnvMain}
\newcounter{counterEnvDefault}
\numberwithin{counterEnvDefault}{section}

\theoremstyle{plain}

\newtheorem{lemma}[counterEnvDefault]{Lemma}
\newtheorem*{lemma*}{Lemma}

\newtheorem{theorem}[counterEnvDefault]{Theorem}
\newtheorem*{theorem*}{Theorem}

\newtheorem{proposition}[counterEnvDefault]{Proposition}
\newtheorem*{proposition*}{Proposition}

\newtheorem{corollary}[counterEnvDefault]{Corollary}
\newtheorem*{corollary*}{Corollary}

\newtheorem*{assumption*}{Assumption}

\theoremstyle{definition}

\newtheorem*{exercise*}{Exercise}

\newtheorem{definition}[counterEnvDefault]{Definition}
\newtheorem*{definition*}{Definition}

\newtheorem*{notation*}{Definition}

\newtheorem{remark}[counterEnvDefault]{Remark}
\newtheorem*{remark*}{Remark}

\newtheorem*{claim*}{Claim}

\newtheorem*{assertion*}{Assertion}

\usepackage{microtype}
\renewcommand\epsilon\varepsilon

\usepackage{hyperref}
\definecolor{colorlinks}{RGB}{0, 24, 168}
\definecolor{colorcites}{RGB}{0, 138, 118}
\hypersetup{
    colorlinks=true,
    linkcolor=colorlinks,
    citecolor=colorcites,
    urlcolor=colorlinks,
    pdfborder={0 0 0}
}

\usepackage{xargs}
\usepackage[colorinlistoftodos,prependcaption,textsize=tiny]{todonotes}

\newcommandx\work[2][1=]{\todo[linecolor=RoyalBlue,backgroundcolor=RoyalBlue!25,bordercolor=RoyalBlue,#1]{\textsc{todo} #2}}
\newcommandx\comment[2][1=]{\todo[linecolor=OliveGreen,backgroundcolor=OliveGreen!25,bordercolor=OliveGreen,#1]{\textsc{comment} #2}}
\newcommandx\mistake[2][1=]{\todo[linecolor=red,backgroundcolor=red!25,bordercolor=red,#1]{\textsc{mistake} #2}}
\newcommandx\improve[2][1=]{\todo[linecolor=orange,backgroundcolor=orange!25,bordercolor=orange,#1]{\textsc{improve} #2}}
\newcommandx\change[2][1=]{\todo[linecolor=yellow,backgroundcolor=yellow!25,bordercolor=yellow,#1]{\textsc{change} #2}}
\newcommandx\mem[2][1=]{\todo[linecolor=orange,backgroundcolor=orange!25,bordercolor=orange,#1]{\textsc{mem} #2}}
\newcommandx\status[2][1=]{\todo[linecolor=Blue,backgroundcolor=Blue!25,bordercolor=Blue,#1]{\textsc{Status} #2}}

\newcommand\lra{\longleftrightarrow}
\newcommand{\nlra}{\centernot\longleftrightarrow}

\newcommand{\eps}{\epsilon}

\newcommand\Var{\operatorname{Var}}

\newcommand{\sfw}{\mathsf{w}}
\newcommand{\e}{\mathrm{e}}
\newcommand{\om}{\omega}

\newcommand\ind[1]{\mathds{1}_{\{#1\}}}

\newcommand\bfi{\mathbf{i}}

\newcommand\C{\mathbb C}
\newcommand\D{\mathbb D}

\newcommand\R{\mathbb R}
\renewcommand\S{\mathbb S}

\newcommand\Z{\mathbb Z}

\newcommand\bbC{\mathbb C}
\newcommand\bbD{\mathbb D}

\newcommand\bbP{\mathbb P}

\newcommand\bbR{\mathbb R}
\newcommand\bbS{\mathbb S}

\newcommand\bbZ{\mathbb Z}

\newcommand{\sfQ}{\mathsf{Q}}

\newcommand\calB{\mathcal B}
\newcommand\calC{\mathcal C}
\newcommand\calD{\mathcal D}

\newcommand\calF{\mathcal F}

\newcommand\calK{\mathcal K}

\newcommand\calU{\mathcal U}

\newcommand\calW{\mathcal W}

\newcommand\calY{\mathcal Y}

\newcommand\frakC{\mathfrak C}

\newcommand\frakL{\mathfrak L}

\newcommand\frakc{\mathfrak c}
\newcommand\frakd{\mathfrak d}

\newcommand\frakf{\mathfrak f}
\newcommand\frakg{\mathfrak g}

\newcommand\frakm{\mathfrak m}

\newcommand\fraks{\mathfrak s}

\newcommand\fraky{\mathfrak y}

\newcommand{\sfP}{\mathsf{P}}

\mathtoolsset{showonlyrefs}

\makeatletter
\@namedef{subjclassname@2020}{\textup{2020} Mathematics Subject Classification}
\makeatother
\microtypesetup{expansion=false}

\title[Analyticity of the FK correlation length]{Analyticity of the correlation length in subcritical FK-percolation and Potts models}

\author{Lucas D'Alimonte}
\address{Institut Camille Jordan, Université Claude Bernard Lyon 1}
\email{dalimonte@math.univ-lyon1.fr}

\author{Loïc Gassmann}
\address{Université de Fribourg}
\email{loic.gassmann@unifr.ch}

\date{\today}

\usepackage{tikz}

\begin{document}

\maketitle

\begin{abstract}

We prove that the correlation length of FK-percolation and Potts models on $\Z^d$ is analytic in the whole subcritical regime of parameters.
Our proof makes use of the identification of the inverse of the polar dual of the mass with the solution of an equation involving the leading eigenvalue of the \emph{FK-Ruelle operator}, introduced in the context of Ornstein--Zernike theory in Campanino, Ioffe and Velenik (Probab. Theory Relat. Fields, 2003), and crucially relies on Kato's theory of analytic perturbation of linear operators. 
As a consequence of this result, we prove that the correlation length of FK-percolation cannot be invariant by rotation except at countably many values of the temperature parameter in the subcritical phase.  
\end{abstract}

\setcounter{tocdepth}{2}

\makeatletter
\def\l@section{\@tocline{1}{0pt}{1em}{}{\bfseries}}
\def\l@subsection{\@tocline{2}{0pt}{4em}{}{}}
\makeatother

\makeatletter
\newcommand{\sectionstarwithouttoc}[1]{%
  \begingroup
  \let\addcontentsline\@gobblethree
  \section*{#1}%
  \endgroup
}
\makeatother

\tableofcontents

\section{Introduction}

\subsection{Preface}

This article is primarily concerned with the analyticity properties of the correlation length of the \emph{subcritical} FK-percolation measure. 
Background on the analyticity problem for off-critical thermodynamic quantities of statistical mechanics lattice models is already presented in the introduction of the companion paper~\cite{dalimonte_analyticity_2026}, thus we shall be briefer here. 
This article focuses on the FK-percolation model on the lattice $\Z^d$, for any dimension $d \geq 2$, in its subcritical phase.
Since the very influential work of~\cite{duminil-copin_sharp_2019}, it is known that the phase transition for this model is \emph{sharp}, that is to say, that the two-point function of the model decays exponentially fast in the whole subcritical regime.
The inverse of the exponential rate of decay of the two-point function is an important quantity in percolation theory, and is usually called the \emph{correlation length}.
The rate of decay itself is referred to as \emph{the mass} of the model.

\medskip
\noindent
As discussed in the introduction of~\cite{dalimonte_analyticity_2026}, understanding the regularity properties of thermodynamic quantities as functions of the temperature parameter is a problem of some significance in statistical mechanics, and some attention has been given to the behaviour of the mass of some classical models.
Some results in that direction have been obtained in the last fifty years; we try to summarize them in what follows.  
The first rigorous analysis of the analyticity properties of the mass in a context of a random surface model in a perturbative regime appeared in~\cite{abraham_random_1984}. 
This work also established the link between the analyticity problem and the celebrated \emph{Ornstein--Zernike theory}, which is a crucial ingredient in almost all the proofs that will be discussed below. Prior to that,~\cite{paes-leme_ornstein-zernike_1978} studied the analyticity properties of a different quantity related to the rate of decay of the Fourier transform of the two-point function, the \emph{complex mass shell}. 
Slightly later on, the mass of a related statistical mechanics model, the \emph{self-avoiding walk} model, has been proved to be analytic in the activity parameter in~\cite{chayes_random_1984,chayes_ornstein-zernike_1986} (see also the reference book~\cite{madras_self-avoiding_1993}). 
More related to this work, the case of Bernoulli percolation was treated in~\cite{campanino_gaussian_1991}, in which analyticity of the mass in the whole subcritical regime was established. 

\medskip
\noindent
In the case of correlated \emph{spin models}, however, the picture is less complete. 
One of the archetypical test grounds for such models is the classical \emph{Ising model}. 
In two dimensions, the Ising model is integrable, as proved in the groundbreaking work of Onsager~\cite{Onsager}. 
Its mass was exactly computed in~\cite{2D_Ising_mass}, and indeed defines an analytic function of the inverse temperature parameter in the subcritical regime. 
Related to two-dimensional integrability is the case of the free measure of the critical FK-percolation model, in two dimensions, with $q > 4$. 
In that case, the mass was shown to be non-zero, and more strikingly was also exactly computed using a rigorous implementation of the \emph{Bethe ansatz} ~\cite{duminil-copin_discontinuity_2021} and shown to define an analytic function of $q$ on the critical line for $q \in (4, +\infty)$. 
In dimensions $d \geq 3$, these integrability properties disappear, and the best results to this date are convergent power series expansions of the mass around $\beta = 0$~\cite{ocaroll_analyticity_1984-1,bricmont_frohlich_1985} and around $\beta = +\infty$~\cite{ocaroll_analyticity_1984-2,bricmont_frohlich_1985} (in the second regime, one needs to consider the exponential decay rate of~\emph{truncated correlations}). It is also worth mentioning that the latter perturbative expansions hold in more complicated settings than the Ising one, \emph{e.g.} in the context of lattice gauge theories~\cite{ocaroll_gluebal_1984}. The general picture for the Ising mass in dimensions greater than two is thus open, and the problem of proving its analyticity is classical, and for instance mentioned in the open problems section of the review article~\cite{ott_asymptotics_2023}. 

\medskip
\noindent
In this work, we prove that the correlation length of FK-percolation in $\Z^d$ for any $d\geq 2$ in a fixed direction is analytic as a function of the percolation parameter $p$ in the whole subcritical interval.
This result implies as a special case the analyticity of the correlation length of the Potts model for any $q \geq 2$, in any dimension $d \geq 2$. 
This is a consequence of the celebrated \emph{Edwards--Sokal coupling}, that will be recalled below.
Specializing to $q =2$ treats the case of the Ising model, and thus solves the conjecture of~\cite{ott_asymptotics_2023} mentioned above. 
We conclude this introduction by mentioning that in fact, we prove the stronger result that the correlation length is \emph{jointly} analytic in both the inverse temperature parameter and the direction in which it is computed.

\subsection{Definition of the model}

\subsubsection{FK-percolation}
This work is primarily concerned with a classical percolation model called \emph{FK-percolation} (sometimes referred to as \emph{random-cluster model}), that we now define. 
Let $d \geq 2$ be the dimension of the underlying lattice, and $p \in (0,1), q\geq 1$ be two parameters. 
We abuse notation by writing $\Z^d$ for the hypercubic lattice with set of vertices given by $\Z^d$ and nearest-neighbour edges for the $L^1$-norm on this lattice. 
Let $G = (V(G), E(G))$ be a subgraph of $\Z^d$. 
A \emph{percolation configuration} on $G$ is an element of $\{0,1\}^{E(G)}$, typically denoted by $\om$. 
For an edge $e \in E(G)$, we say that $e$ is \emph{open} if $\om_e=1$ and \emph{closed} otherwise. 
The number of open edges of a configuration $\om$ will be referred to as $o(\om)$. 

We now assume that $G$ is a \emph{finite} subgraph of $\Z^d$.
If $\om$ is a percolation configuration on $G$ and $\xi $ a percolation configuration on $\Z^d$,  write $k^\xi_G(\om)$ for the number of connected components of vertices of $V(G)$ in the subgraph with edges $\om|_{E(G)} \cup \xi|_{E(\Z^d)\setminus E(G)}$. Such connected components shall be called \emph{clusters}. 
The FK-percolation measure on $G$ with boundary conditions $\xi$ is now defined on the set of percolation configurations on $G$ as follows
\begin{equation}
\phi^\xi_{G,p,q}[\om] \coloneqq \frac{1}{\mathbf{Z}^\xi_{G,p,q}}\left(\frac{p}{1-p} \right)^{o(\om)}q^{k^\xi_G(\om)}.
\end{equation}
In the latter equation, $\mathbf{Z}^\xi_{G,p,q}$ is the constant normalizing the measure at 1, and is referred to as the \emph{partition function}. 
For $n \in \Z_{> 0}$, define $\Lambda_n \coloneqq \{-n, \dots, n\}^d$ seen as a subgraph of~$\Z^d$. For every vertex $x$ of $\bbZ^d$, we further define the box centered at $x$ by $\Lambda_n(x) \coloneqq x + \Lambda_n$, also viewed as a subgraph of $\bbZ^d$.
It is classical (see~\cite{grimmett_RCM}) that when the boundary conditions~$*$ are either \emph{free} (corresponding to $\xi \equiv 0$) or \emph{wired} (corresponding to $\xi \equiv 1$) then the sequence of measures $\phi^*_{\Lambda_n, p,q}$ converges to a percolation measure on $\Z^d$, called the free (resp. wired) measure in infinite volume, and denoted by $\phi^{0}_{p,q}$ (resp. $\phi^1_{p,q}$).

Denote by $\{0 \lra \infty\}$ the event that the vertex cluster of $0 \in \Z^d$ is infinite. 
The critical point $p_c$ is defined as follows :
\begin{equation}
p_c \coloneqq p_c(q,d) \coloneqq \sup \{p \in (0,1), \phi^1_{p,q}[0 \lra \infty] = 0 \}.
\end{equation}
It is notorious, and is the basic result of percolation theory, that the model enjoys a non-trivial \emph{phase transition}, that is to say, that $p_c \in (0,1)$ as long as $d \geq 2$ and $q\geq 1$.

Motivated by the quantitative decay of the two-point function of the FK-percolation in the subcritical regime, we introduce the following important quantity. 

\begin{definition*}[Mass]
Let $p \in (0,1), q\geq 1$ and $x \in \R^d$. 
Then the following\footnote{$\lfloor nx \rfloor$ denotes the point of the lattice which is the closest to $nx \in \R^d$. In case of equality, we break ties by arbitrarily choosing the first vertex in the lexicographical order.} limit exists in $\R_{\geq 0}$ and is called the \emph{mass} (or \emph{inverse correlation length}) in the direction given by $x$:
\begin{equation}
\frakm(p, x) \coloneqq \lim_{n \rightarrow +\infty} -\frac 1 n \log \phi^0_{p,q}[0 \lra \lfloor nx \rfloor]. 
\end{equation} 
\end{definition*}

One of the most celebrated results of the last decades of percolation theory is the proof of the \emph{sharpness} of the phase transition of FK-percolation~\cite{duminil-copin_sharp_2019}. 
In this article, the authors prove that for any $d\geq 2, q \geq 1, x \in \R^d\setminus \{0\}$, and any $p <p_c$, it is the case that  $\frakm(p,x) > 0$.
That is to say, the two-point function of FK-percolation decays exponentially fast at any subcritical value of the parameter $p$. 

\subsubsection{Potts model}

FK-percolation is intimately tied to a lattice spin model, called the Potts model, that we briefly introduce now. 
First set an integer $q \geq 2$ (the number of colours).
The spin space will be the set $\{1, \dots, q\}$. 
Given an inverse temperature $\beta > 0$,  a finite set $\Lambda \subset \Z^d$, and a \emph{boundary condition} $\tau \in \{0,1, \dots, q\}^{\bbZ^d}$, define the probability measure:
\begin{equation}
\mu^\tau_{\Lambda, \beta}[\sigma] \coloneqq \frac{1}{\mathbf{Z}^{\tau}_{\Lambda, \beta}} \exp\big( \beta \sum_{x\sim y \atop x,y \in\Lambda} \ind{\sigma_x = \sigma_y} + \beta \sum_{x\sim y \atop x \in \Lambda, y \notin\Lambda} \ind{\sigma_x = \tau_y}\big),
\end{equation}
where $x \sim y$ indicates that $x$ and $y$ are neighbouring vertices. The constant $\mathbf{Z}^{\tau}_{\Lambda, \beta}$ is the unique number normalizing the measure to 1, and is called the \emph{partition function}. 
We will mostly be interested in \emph{monochromatic} measures. 
For $j\in \{0, 1, \dots, q\}$, abbreviate $\mathbf{j}$ for the boundary condition constant equal to $j$ on $\Z^d$. 
The measure $\mu^{\mathbf{0}}_{\Lambda, \beta}$ will be referred to as \emph{free} in the volume $\Lambda$, while $\mu^{\mathbf{j}}_{\Lambda, \beta}$ with $j \neq 0$ will be referred to as \emph{monochromatic} in volume $\Lambda$. 
It is classical (see ~\cite{HDC_lecture_notes}) that all the measures $\mu^{\mathbf{j}}_{\Lambda, \beta}$ admit limits when $\Lambda \rightarrow \Z^d$, provided a suitable growth condition on the volumes $\Lambda$ (the sequence of boxes $(\Lambda_n)_{n \geq 1}$ satisfies this growth condition). 
Those limits will be referred to as $\mu^{\mathbf{j}}_{\beta}$.

The Potts measure also features a phase transition; there exists a non-trivial parameter $\beta_c \in (0,+\infty)$ such that when $\beta < \beta_c$, all the monochromatic measures in infinite volume coincide, while when $\beta > \beta_c$, they become distinct and mutually singular.  

Define the \emph{mass} of the model at temperature $\beta > 0$ associated to $x \in \R^d$ by
\begin{equation}
\frakm^{\mathsf{Potts}}(\beta, x) \coloneqq -\lim_{n\rightarrow \infty} \frac 1 n  \log\big(\mu^{\mathbf{0}}_{\beta}[\sigma_0 = \sigma_{\lfloor nx \rfloor}] -\tfrac 1 q\big).
\end{equation}

Again, the deep result of~\cite{duminil-copin_sharp_2019} implies that the phase transition of the Potts model is \emph{sharp}, namely that $\frakm^{\mathsf{Potts}}(\beta, x) > 0$ for any $\beta < \beta_c, x \in \R^d\setminus \{0\}$.  

\begin{figure}
	\begin{center}
	\includegraphics[width = 0.7\textwidth]{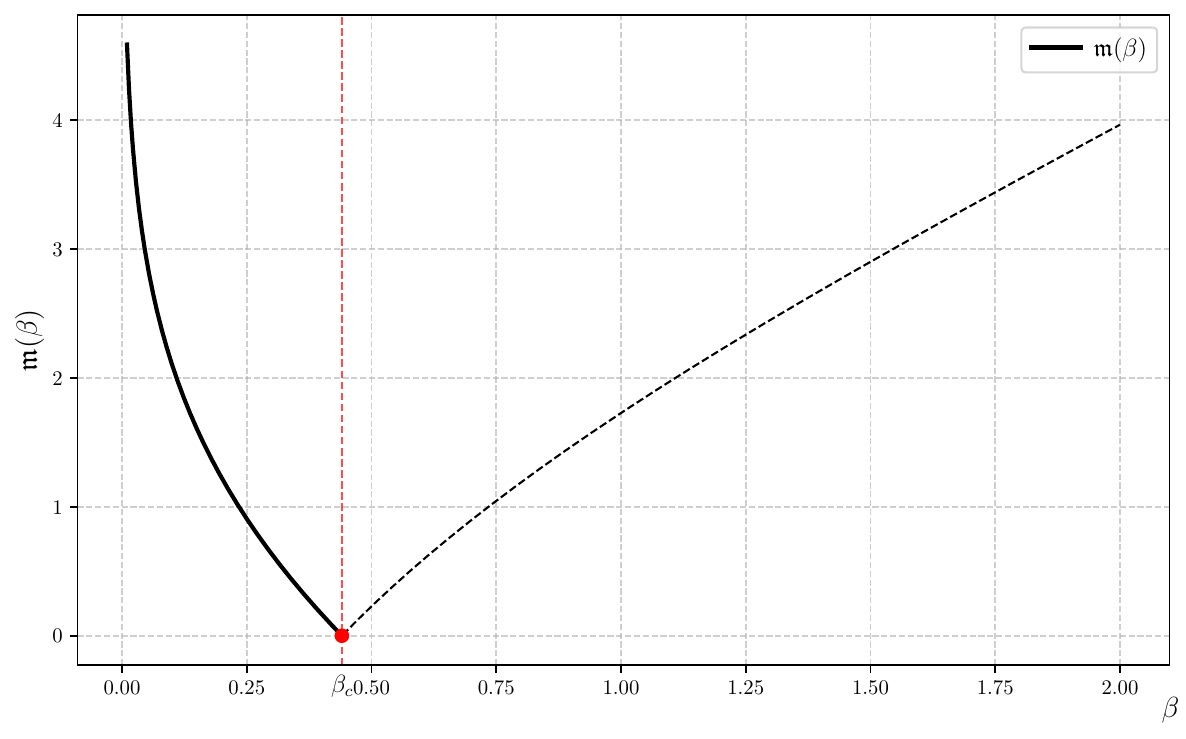}
	\caption{Graphical representation of the mass of the 2D Ising model in an axis direction, which is explicitly computable with the results of~\cite{2D_Ising_mass}. 
	The plain part of the curve corresponds to the high-temperature regime; we prove in this article that it is indeed analytic in the case of the Potts model, for any values of $q \geq 1$ and $d\geq 1$. 
	The dotted part corresponds to the rate of decay of truncated correlation functions, which is not covered by the results of this article. 
	Finally, the red dot corresponds to $\beta_c$, the unique point at which $\frakm(\beta_c) = 0$.
	 }
	\label{fig:mass_2D}	
	\end{center}
	\end{figure}

\subsubsection{Edwards--Sokal coupling}

It is well-known that when $q$ is an integer, the FK-percolation model with parameters $p \in (0,1)$, $q \in \Z_{\geq 2}$ and the Potts model with $q$ colours at inverse temperature $\beta(p) \coloneqq -\log(1-p) > 0$ can be coupled through the so-called \emph{Edwards--Sokal coupling} (again we refer to~\cite{HDC_lecture_notes} for a more extensive background). 
In particular, we will use the following very classical input:

\begin{lemma}[Edwards--Sokal coupling for the masses]
\label{lem:ES_coupling}
For any $p \in [0,1)$, $q \in \bbZ_{\geq 2}$, $x \in \R^d\setminus \{0\}$,
\begin{equation}\frakm(p,x) = \frakm^{\mathsf{Potts}}(\beta(p),x).\end{equation}
\end{lemma}

In particular, as $p \mapsto \beta(p)$ is analytic, analyticity properties of $\frakm^{\mathsf{Potts}}$ will immediately follow from the analyticity properties of $\frakm$ by composition. 

\subsection{Exposition of the results}
\subsubsection{Analyticity of the correlation length}

We are now ready to state our main results. 
Recall that a function $f : U \mapsto \R$ defined on an open set $U \subset \R^m$ is said to be \emph{analytic} on $U$ if it can be expanded as a convergent multivariate power series at any $u \in U$. 
That is to say, for any $u = (u_1, \dots, u_m) \in U$, there exist coefficients $a_{n_1, \dots, n_m} \in \R$ and an open neighbourhood $V$ of $u$ such that for any $x = (x_1, \cdots, x_m) \in V$,
\[
f(x) = \sum_{n_1, \dots, n_m \geq 0}a_{n_1, \dots, n_m} (x_1 - u_1)^{n_1}\cdots (x_m - u_m)^{n_m}. 
\]
We will sometimes use the terminology \emph{jointly analytic} instead of \emph{analytic} to emphasize on the multivariate setting. 

\begin{theorem}\label{thm:analyticity_corr_length}
For any dimension $d \geq 2$ and any cluster weight $q\geq 1$, the function
\[
(p, x) \mapsto \frakm(p, x)
\]
is jointly analytic on $(0, p_c) \times \R^d\setminus \{0\}$.
\end{theorem}

\begin{remark}
Note that the analyticity in the $x$ variable, at a fixed value of $p < p_c$ was already established \cite{campanino_fluctuation_2008}. 
\end{remark}

Lemma~\ref{lem:ES_coupling} then implies the following.

\begin{corollary}
For any dimension $d \geq 2$ and any color number $q \in \Z_{\geq 2}$, the function
\[
(\beta,x)   \mapsto \frakm^{\mathsf{Potts}}(\beta, x)
\]
is jointly analytic on $(0, \beta_c) \times \R^d\setminus\{0\}$.
\end{corollary}

\subsubsection{Duality in two dimensions}

In two dimensions, we can also draw some conclusions regarding the \emph{truncated mass} of FK-percolation, defined as follows, for $x \in \R^d\setminus \{0\}$:
\begin{equation}
\frakm^{\mathsf{f}}(p,x) \coloneqq -\lim_{n\rightarrow \infty} \frac 1 n \log \phi^{1}_{p,q}[ 0 \lra \lfloor nx \rfloor, 0 \nlra \infty] \in \R_{\geq 0}.
\end{equation}
While it is not currently known\footnote{Except in the special case of the FK-Ising measure obtained when $q=2$, see~\cite{duminil-copin_exponential_2020}.} whether $\frakm^{\mathsf{f}}(p,x) > 0$ for any $p \neq p_c$ in general dimension $d \geq 2$, the case of dimension 2 is special, due to planar duality. 
Indeed, denote by $p^*$ the dual parameter of $p$, that is, the unique solution of the equation
\begin{equation}pp^* = q(1-p)(1-p^*). \end{equation}
Then, it was proved in~\cite{borgs_covariance_1996} (see also~\cite{dalimonte_entropic_2024}) that for any $q \geq 1$, when $d=2$, it is the case that 
\begin{equation} \frakm^\mathsf{f}(p,x) = 
\begin{cases}
\frakm(p,x) &\text{ if } p<p_c \\
2\frakm(p^*,x) &\text{ if }p>p_c.
\end{cases} 
\end{equation}
This proves the following corollary in dimension 2:

\begin{corollary}\label{cor:truncated_corr_lenght_RCM}
In dimension $d=2$, the truncated mass
\begin{equation} (p,x) \mapsto \frakm^{\mathsf{f}}(p,x)\end{equation}
is jointly analytic and strictly positive on $(0,1)\setminus\{p_c\} \times \R^d \setminus \{0\}$.
\end{corollary}

In the case of two-dimensional Potts model, such a duality principle maps the mass at a subcritical parameter towards the so-called \emph{surface tension} of the model. 
This quantity is defined as follows.
For any unit vector $\mathbf{n} \in \bbS^1$ satisfying $\mathbf{n}_1 \geq 1/\sqrt 2$, define the following \emph{Dobrushin boundary condition}: 
\begin{equation} \tau^{\mathbf{n}}(y) = 
\begin{cases}
1 &\text{if }\langle {\mathbf{n}}^\perp, y \rangle \geq 0, \\
2  &\text{if }\langle {\mathbf{n}}^\perp, y \rangle < 0,
\end{cases}
 \end{equation}
 where $\langle \cdot,\cdot \rangle$ is the Euclidean scalar product on $\R^2$ and ${\mathbf{n}}^\perp = (-\mathbf{n}_2, \mathbf{n}_1)$.  
The surface tension of the model is then defined as 
\begin{equation} 
\sigma^{\mathbf{n}}(\beta) \coloneqq 
\lim_{n\rightarrow \infty}
- \frac{1}{n} 
\log 
\frac{
\mathbf{Z}_{\Lambda_n, \beta, q}^{\tau^{\mathbf{n}}}
}{
\mathbf{Z}_{\Lambda_n, \beta, q}^{\mathbf{1}}
},.
\end{equation}
It is well known~\cite{wu_duality_1997} that in two dimensions, the surface tension at any parameter $\beta > \beta_c$ corresponds to the mass at the dual parameter $\beta^* < \beta_c$.
As a consequence, we obtain the following:

\begin{corollary}\label{cor:surface_tension}
When $d=2$, for any $q \in \Z_{\geq 2}$, the function
\begin{equation} 
(\beta, \mathbf{n}) \mapsto \sigma^{\mathbf{n}}(\beta) > 0  
\end{equation}
is jointly analytic on $(\beta_c, +\infty) \times \bbS^1$.
\end{corollary}

\subsection{Discussion about the saturation phenomenon}

It is interesting to put our main result, Theorem~\ref{thm:analyticity_corr_length} in regard with the results of~\cite{aoun_non-analyticity_2021,aoun_two-point_2022,aoun_ornstein-zernike_2024}, which are particularly related to this work.

Indeed, in this series of works, the authors considered a class of long-range models, including (but not restricted to) the FK-percolation model with coupling constants of the form $J_x = \psi(x)\e^{-\rho(x)}$, with $\rho$ a norm on $\Z^d$ and $\psi$ a sub-exponential correction. 
In this framework, some conditions on the functions $\psi$ and $\rho$ were identified in~\cite{aoun_non-analyticity_2021,aoun_two-point_2022} to ensure the existence of a parameter $\beta_\mathsf{sat} \in (0, \beta_c)$ such that the mass of the model behaves differently in the regimes $(0, \beta_\mathsf{sat})$ and $(\beta_\mathsf{sat}, \beta_c)$. 
Indeed in the former regime, the rate of decay of the two-point function is governed by the norm $\rho$, while in the latter, the rate of decay is $\beta$-dependent (and actually strictly decreasing with $\beta$). 
Furthermore, more refined estimates on the behaviour of the two-point function are available.
On the one hand, it is proved in~\cite{aoun_two-point_2022} that for a restricted class of prefactor functions $\psi$, the classical Ornstein--Zernike asymptotical behaviour for the two-point function fails to occur in the whole saturated regime $\beta \in (0, \beta_{\mathsf{sat}})$. 
On the other hand, in~\cite{aoun_ornstein-zernike_2024}, the authors prove that under a \emph{non-saturation} assumption (called $\mathsf{NSA}$), then the usual Ornstein--Zernike decay of the two-points function is recovered. Furthermore it is conjectured that this $\mathsf{NSA}$ assumption is actually sharp, that is to say, that it holds in the whole regime $(\beta_{\mathsf{sat}}, \beta_c)$. 

In link with this work, we note that for models displaying $\beta_\mathsf{sat} \in (0, \beta_c)$ --- many fairly simple example of such models are discussed in~\cite{aoun_two-point_2022}\footnote{A concrete instance is the FK-percolation model on $\Z^d$ with the choice of coupling constants $J_x = \exp(- \rho(x)^\eta)\e^{-\rho(x)}$, where $\rho$ is a norm on $\Z^d$ respecting the symmetries of the lattice, and $0 < \eta < 1$.}--- the mass fails to be analytic at the parameter $\beta_\mathsf{sat}$. 
This is an interesting and previously unexpected phenomenon, as failure of analyticity of some thermodynamical quantities occurs \emph{within the disordered phase} of the model, revealing a more intricate change of behaviour than the mere phase transition of the model. 

In light of this discussion, our result proves the absence of such a phenomenon in the nearest-neighbour FK-percolation and Potts models, by showing analyticity of the correlation length in the whole subcritical phase. 
Furthermore, it was conjectured in~\cite[Open Problem 1.16]{aoun_non-analyticity_2021} that the correlation length should always be analytic in the regime $(\beta_\mathsf{sat}, \beta_c)$ (observe that in the regime $(0,\beta_\mathsf{sat})$, it is constant equal to $\rho$ and thus trivially analytic). 
We strongly believe that our work provides evidence for this conjecture to be true. 
Indeed, it establishes it in the arguably degenerate case of a model in which $\beta_\mathsf{sat} = 0$, 
but most importantly uses mostly three main ingredients (exponential mixing, Ruelle operator representation, mass gap) that are shown to be valid in the regime $\mathsf{NSA}$ discussed above. 
We thus think that our method supports the validity of this conjecture, and could possibly provide a good approach to tackle it.

\subsection{A rigidity principle and an application to non-rotational invariance of subcritical FK-percolation}

Analyticity of the correlation length provides a ``rigidity principle'', that can be summarized as follows. 

\begin{center}
``\emph{A property of the model that can be reformulated as an identity for a function of the correlation length is either true in the whole subcritical regime, or true only on an isolated set of parameters}.''
\end{center}

This principle can be useful in the sense that it is often easy to check the (non-)validity of an identity implying the correlation length in a \emph{perturbative} regime, that is for very low values of $p$. 

We provide an example of this principle concerning the (non-)rotational invariance of off-critical FK-percolation. 
For any $p < p_c$, define the sets (see~\cite{campanino_fluctuation_2008,dalimonte_near-critical_2025} for a discussion):
\begin{equation} 
\calU_p \coloneqq \{ x \in \R^d, \frakm(p,x) \leq 1 \} 
\quad 
\text{and}
\quad 
\calW_p \coloneqq \{ t \in \R^d, \forall \mathbf{n} \in \bbS^{d-1}, \langle \mathbf{n}, t\rangle \leq \frakm(p,\mathbf{n}) \}.
 \end{equation}
The first set is called the \emph{equi-decay set} and the second one the \emph{Wulff shape}.  
 
It is expected that the renormalized Wulff shape $\calW_p$ becomes round when $p$ approaches the critical point $p_c$. 
This fact has been established in a certain number of two-dimensional cases~\cite{duminil-copin_wulff_2013,manolescu_wulff_2026}, and in high dimension (higher than 15 in the case of Bernoulli percolation) in~\cite{liu_crossover_2026}.
Furthermore, this phenomenon should be a \emph{feature} of the critical point: a general statistical mechanics model on a lattice outside of its critical regime should ``feel'' the lattice and not exhibit any rotational symmetry. 
Using our main result Theorem~\ref{thm:analyticity_corr_length}, we are able to provide a slightly weaker --- but almost optimal --- version of this second statement. 
Define the set of ``exceptional rotationally invariant points'' as follows
\begin{equation}  \mathsf{Rot}_\mathsf{FK} \coloneqq \{ p \in (0,p_c), \exists \lambda \in \R, \partial\calW_p = \lambda\bbS^{d-1}  \}. \end{equation}
It corresponds to the set of ``exceptional'' parameters at which the Wulff shape is a dilation of the unit sphere\footnote{In the definition of this exceptional set, we might have replaced $\calW_p$ by $\calU_p$ without any issue as those two convex bodies are convex dual to each other, see~\cite{campanino_fluctuation_2008,dalimonte_near-critical_2025}.}. Then, our result is the following. 
\begin{theorem}\label{thm:non_rot}
For any $d \geq 2$, any $q \geq 1$ and any $\eps >0$, 
\begin{equation} \mathsf{Rot}_\mathsf{FK} \cap (\eps, p_c -\eps)\quad  \emph{\text{is finite}}.  \end{equation}
\end{theorem}

It is natural to conjecture that $\mathsf{Rot}_\mathsf{FK}$ is empty (observe that this is true in the case of the two-dimensional Ising model by the computation of~\cite{2D_Ising_mass}, see also~\cite{messikh_surface_2006}) but we are currently unable to derive this stronger result. 

\begin{proof}
That the Wulff shape is a sphere implies that the mass is constant on $\S^{d-1}$. 
We prove that this can occur at most in an isolated subset of $(0, p_c)$. 
Consider the vectors $\mathbf{e_1} = (1,0,\dots,0)$ and $\mathbf{e_{\mathsf{diag}}} = \tfrac 1 {\sqrt d} (1, 1, \dots, 1) \in \S^{d-1}$. 
We now consider the function $f : p\in (0, p_c) \mapsto  \frac{\frakm(p,\mathbf{e_{\mathsf{diag}}})}{\frakm( p,\mathbf{e_1})} - 1$.
By Theorem~\ref{thm:analyticity_corr_length}, $f$ is real analytic in its domain: thus its zero set is either $(0, p_c)$, either consists in only isolated points and the result follows.
We rule out the first alternative by the easy observation that 
\[ f(p) \underset{p \rightarrow 0^+}{\longrightarrow} \sqrt d - 1, \]
since the rescaled equi-decay set converges towards the $L^1$ unit ball as $p \rightarrow 0^+$. 
\end{proof}

\begin{remark}
Other interesting identities involving surface tensions or masses might be considered in the framework of the rigidity principle applied to the mass and/or the surface tension of the two-dimensional Potts models, such as 
\begin{itemize}
\item \emph{Antonov's rule} $2\sigma^{\mathbf{0}/\mathbf{1}}(\beta) = \sigma^{\mathbf{1}/\mathbf{2}}(\beta)$, the validity of which implies the \emph{perfect wetting} of the system~\cite{messager_antonov_1991}. 
\item \emph{Young's relation} for the contact angle of a sessile drop on a wall~\cite{de-coninck_young_1989}. 
\item Strict triangle inequality for the mass: for any pair of non-collinear vectors $(u, v)$, we have the strict inequality $\frakm(p,u+v) < \frakm(p,u) + \frakm(p,v)$.
Equality in this relation would imply the creation of \emph{facets} in the Wulff shape.  
The validity of this strict inequality in the case of the Potts model is a consequence of~\cite{campanino_fluctuation_2008}; however our method gives an alternative proof of the weaker statement that once $u$ and $v$ are fixed this strict inequality can only fail for an isolated set of parameters. 
\end{itemize}
\end{remark}

\subsection{Polar duality}

Recall the definitions of the \emph{equi-decay set} and of the \emph{Wulff shape}.
\begin{equation} 
\calU_p = \{ x \in \R^d, \frakm(p,x) \leq 1 \} 
\quad 
\text{and}
\quad 
\calW_p = \{ t \in \R^d, \forall \mathbf{n} \in \bbS^{d-1}, \langle \mathbf{n}, t\rangle \leq \frakm(p,\mathbf{n}) \}.
 \end{equation}
 
 Those two sets are polar dual to each other (see~\cite{convexity} for all the relevant background on convex duality).
 Moreover, $\frakm(p, \cdot)$ is the gauge function (resp. the support function) of $\calU_p$ (resp. of $\calW_p$). Introduce $\frakg(p, \cdot)$ to be the gauge function (resp. the support function) of $\calW_p$ (resp. of $\calU_p$), that is the function such that: 
 \begin{equation} 
\calW_p = \{ x \in \R^d, \frakg(p,x) \leq 1 \} 
\quad 
\text{and}
\quad 
\calU_p = \{ t \in \R^d, \forall \mathbf{n} \in \bbS^{d-1}, \langle \mathbf{n}, t\rangle \leq \frakg(p,\mathbf{n}) \}.
 \end{equation}
From this equation, one immediately sees that $\frakg$ is strictly positive in $(0,p_c) \times \R^d\setminus\{0\}$.
For reasons that will appear clearer later, it is actually more convenient to prove that the \emph{inverse} of $\frakg$ is analytic.
\textbf{ Thus, in what follows, we will systematically adopt the notation $\fraky(p, t) := \frac{1}{\frakg(p, t)}$}.
This new\footnote{For any $\mathbf{n} \in \bbS^{d-1}$, the quantity $\fraky(p,\mathbf{n})$ is nothing but the rate of decay of point-to-hyperplanes connection probabilities in the direction $\mathbf{n}$. This is also the radius of the Wulff shape in the direction given by the unit vector $\mathbf{n}$. We will not use this geometric interpretation, but refer to~\cite{ott_decay_2022,dalimonte_near-critical_2025} for further details.} quantity $\fraky$ is well-defined for $(p,t) \in (0,p_c) \times \R^d\setminus\{0\}$.
We also note that $\fraky(p, \cdot)$ is homogeneous of degree $-1$. 
The main focus of this paper is the proof of the following result.

\begin{theorem}[Analyticity of $\fraky$]
\label{thm:analyticity_leg_m}
For any dimension $d \geq 2$ and $q \geq 1$, the function 
\[
(p,t) \mapsto \fraky(p, t)
\]
is jointly analytic on $(0, p_c) \times \R^d \setminus \{0\}$.
\end{theorem}  

Transmitting this joint analyticity property to the original function $\frakm$ is not difficult. 
We use an \emph{ad hoc} result, Lemma~\ref{lem:dual_analytic}, formulated and proved in the Appendix.  
Thus, one can already provide a proof of Theorem~\ref{thm:analyticity_corr_length}

\begin{proof}[Proof of Theorem~\ref{thm:analyticity_corr_length}]
By definition of $\fraky$, Theorem \ref{thm:analyticity_leg_m} readily implies that the function $(p,t) \mapsto \frakg(p,t)$ is jointly analytic on $(0,p_c) \times \R^d \setminus \{0\}$. 

It remains to check the additional assumptions of Lemma~\ref{lem:dual_analytic}, which are all provided by~\cite[Theorem B]{campanino_fluctuation_2008}.
Indeed, it is the content of that result that for any fixed $p < p_c$, the surfaces $\partial \calU_p$ and $\partial\calW_p$ are analytic, strictly convex and have strictly positive curvature. 
Thus, by Lemma~\ref{lem:dual_analytic}, $(p,x) \mapsto \frakm(p,x)$ is jointly analytic on $(0,p_c) \times \R^d \setminus \{0\}$.
\end{proof}

\subsection{Discussion and overview of the proof}

\subsubsection{Limitations of the ``direct approach''}

Recall that 
\begin{equation}\label{equ:def_corr_length}
\frakm(p, x) \coloneqq  \lim_{n \rightarrow +\infty} -\frac 1 n \log \phi^0_{p,q}[0 \lra \lfloor n x \rfloor].
\end{equation}
A natural first attempt to extend~\eqref{equ:def_corr_length} to the complex plane is to use the uniform analytic extension of the quantity $\phi^0_{p,q}[0 \lra \lfloor n x \rfloor]$ to a complex neighbourhood of $p$ provided by~\cite[Theorem 1.7]{dalimonte_analyticity_2026}, 
and to study the logarithm of these functions; thus directly constructing a mass for the complexified model $\phi^0_{p+z,q}$ defined in~\cite{dalimonte_analyticity_2026}. 
However, the presence of the logarithm in~\eqref{equ:def_corr_length} is a serious nuisance to this strategy, as one would need to show the existence of a complex neighbourhood of $p$ on which all the complex extensions of $ \phi^0_{p,q}[0 \lra \lfloor n x \rfloor]$ do not vanish.
We are not aware whether this property holds true or not; in particular, it might very well be the case that the quantity $\phi^0_{p+z,q}[0 \lra \lfloor n x \rfloor]$ vanishes for $z$ arbitrarily close to $0$, because of the cancellations that may occur when $z$ is not real\footnote{An example illustrating this phenomenon is provided in Appendix~\ref{sec:counter_example}}. 

The implementation of a weaker version of this strategy is well-known in the case of a \emph{magnetic} perturbation of the model. 
Indeed, in~\cite{frohlich_rodriguez_2012,frohlich_rodriguez_2017}, the authors build upon Lee--Yang theory and a beautiful argument due to Lebowitz and Penrose~\cite{lebowitz_penrose} to argue that for any $h\in \C$ with $\Re(h)>0$ and any $\beta > 0$, in the case of the Ising model with temperature $\beta$ and external magnetic field $h$, the quantity
\begin{equation}\label{def:mass_gap_modulus}
- \limsup_{n \rightarrow  \infty} \frac 1 n \log | \langle \sigma_0 \sigma_{\lfloor nx \rfloor} \rangle^c_{\beta, h} |
\end{equation}
is strictly positive. 
The superscript $\langle \cdot \rangle^c$ means that one is looking at the \emph{truncated} correlation functions of the Ising model.
This method may be successfully replicated in our context thanks to~\cite[Theorem 1.7]{dalimonte_analyticity_2026}; 
however its conclusion is significantly weaker than the one of the present paper, 
as it only establishes an upper bound on the exponential decay of the correlations.
Indeed, it does not rule out the fact that replacing the $\limsup$ in~\eqref{def:mass_gap_modulus} by a $\liminf$ may result in an infinite mass. 
As explained above, this is a serious issue that is hard to circumvent.

\subsubsection{Outline of the method}

We thus chose to follow a different strategy, inspired by~\cite{campanino_ornstein-zernike_2003} and prove only that the the real function defined in~\eqref{equ:def_corr_length} can be extended to an analytic function, instead of constructing directly a mass for the complexified model $\phi^0_{p+z,q}$. As explained previously, our method is to first analytically extend the function $\fraky$ instead of~$\frakm$. We explain this choice below.

\medskip
\noindent
The idea of~\cite{campanino_ornstein-zernike_2003} allows to interpret the inverse of the polar dual of the mass $\fraky(p,t)$ as the solution of an equation involving the spectral radius of an operator, which is called the \emph{FK-Ruelle} operator $\frakL_{\psi^t_p}$, acting on an infinite-dimensional functional Banach space. 
The identification of the spectral radius with the dominant eigenvalue of the FK--Ruelle operator, and the important fact that it is isolated from the rest of the spectrum and of algebraic multiplicity one, constitute one of the core results of~\cite{campanino_ornstein-zernike_2003} and fall in the classical framework of \emph{Ruelle--Perron--Frobenius} type theorems. 
Our main contribution is to construct a complexified version of the FK-Ruelle operator $\frakL_{\psi^s_{p+z}}$ and to show that this complexified operator is an holomorphic perturbation of the real FK--Ruelle operator. 
We then crucially rely on Kato's theory of analytic perturbations of linear operators developed in~\cite{kato_perturbation_1995} to conclude that the mapping from $p$ to the leading eigenvalue of $\frakL_{\psi_{p}^t}$ is analytic.
Analyticity of $\fraky$ then follows from an easy application of the holomorphic implicit function theorem, and analyticity of $\frakm$ follows, as explained above.

\medskip
\noindent
The main innovation of this paper is thus the proof of the regularity of the FK-Ruelle operator, when the parameter $p$ is perturbed by a small complex number. 
To this end, we rely on a rigorous cluster expansion method \emph{à la}~\cite{ott_weak_2020-1,ott_ising}.
We want to emphasize that the fact that this method succeeds in proving regularity of the FK-Ruelle potential $\psi^s_{p+z}$ is a feature of the specific form of the FK-Ruelle potential.
Indeed, the reader familiar with Ornstein--Zernike theory may be surprised by the fact that we do not use the ``modern'' renewal picture provided by~\cite{ott_potts_2018,dalimonte_near-critical_2025}. 
Unfortunately, the form of the decoupled weight functions appearing in those works is too intricate to verify the regularity properties needed to implement Kato's perturbation theory. 
This explains the need to perform the cluster expansion technique directly at the level of the FK-Ruelle potential. 

\subsubsection{Organization of the article}
The remainder of the article is organized as follows: 
\begin{itemize}
\item Section~\ref{sec:kato} is a review of one of the fundamental results of Kato's analytic perturbation theory, that shall be used to prove the main result. 
\item Section~\ref{sec:perturbed_ruelle} introduces the FK--Ruelle operator, and states Theorem~\ref{thm:complex_reg_ruelle}, which is the crucial tool for this works. 
We use it to derive the crucial identity~\eqref{equ:corr_length_eigenvalue_Ruelle}, which is the starting point of the analysis. 
\item Section~\ref{sec:conditional_proof} introduces the Ruelle--Perron--Frobenius theorem of~\cite{campanino_ornstein-zernike_2003}, which relates the spectral radius of the FK--Ruelle operator with its leading eigenvalue. 
It contains a proof of Theorem~\ref{thm:analyticity_corr_length}, conditionally on the crucial complex regularity Theorem~\ref{thm:complex_reg_ruelle}.
\item Section~\ref{sec:exp_bounds} is dedicated to the proof of Theorem~\ref{thm:complex_reg_ruelle}, using a cluster expansion method based on~\cite{ott_weak_2020-1}.
\item Finally, Appendix~\ref{sec:appendix} gather technical tools that are needed in the body of the article. It contains a separate new proof of the continuity of the mass in its real subcritical $p$-parameter, a re-proof of a mixing estimate in a slighty different context than the classical formulation in FK-percolation, Lemma~\ref{lem:dual_analytic} about analyticity of polar dual transforms of regular functions, and a counterexample to the generalization of Theorem~\ref{thm:clust_exp_UDE} to general local events.
\end{itemize}

The main novelty of this article is thus Theorem~\ref{thm:complex_reg_ruelle} and its proof; 
the rest of the arguments are mainly adaptation or improvements of the arguments of~\cite{campanino_ornstein-zernike_2003,campanino_fluctuation_2008}.

\section{Preparation: Kato's perturbation theory in several variables}
\label{sec:kato}

As previously explained, we import a key input for the remainder of this work, extracted from the celebrated \emph{analytic perturbation theory} developed by Kato.
The best exposition to that theory is presented in the book~\cite{kato_perturbation_1995}. 
This is a review section; the reader familiar with this theory may safely skip it. 

We will work in a framework that is not \emph{stricto sensu} covered by the latter reference: namely, our perturbation will be analytic \emph{in several variables}. 
This extension has however been studied in~\cite{baumgartel_article} (see the supplement in ~\cite{baumgartel_book} for an english version). 
The result below is presented in~\cite[S2, Corollary 1]{baumgartel_book} and is essentially a direct adaptation of the argument of \cite[chapter VII-§ 1.3]{kato_perturbation_1995}.

In the following, we fix a Banach space $\calF$, and some $\delta > 0$.
If $\mathfrak{L}$ is a linear operator on $\calF$, we denote by $\Sigma(\mathfrak{L})$ the spectrum of $\mathfrak{L}$ (see for example \cite[III-§6]{kato_perturbation_1995}). 
An eigenvalue $\lambda$ of $\mathfrak{L}$ is said to be \emph{non-degenerate} if:
\begin{enumerate}
\item $\lambda$ has algebraic multiplicity 1.
\item $\lambda$ is isolated, that is, there exists a rectifiable simple closed curve $\Gamma$ that separates\footnote{That is, that $\Gamma$ does not intersect $\Sigma(\mathfrak{L})$ and the only point of $\Sigma(\mathfrak{L})$ in the interior of $\Gamma$ is $\lambda$.} $\lambda$ from the rest of the spectrum $\Sigma(\frakL)\setminus \{\lambda\}$.
\end{enumerate} 

\begin{theorem}[\cite{baumgartel_book}]
\label{thm:kato}

Let $m \in \Z_{> 0}$. 
Let $(\mathfrak{L}_{z})_{z \in U}$ be a uniformly bounded family of linear operators on the Banach space $\calF$, where $U$ is a neighbourhood of 0 in $\C^m$. 
Assume that:

\begin{enumerate}
\item The map $z \in U \mapsto \mathfrak{L}_{z}$ is holomorphic.\footnote{An operator-valued map is said to be holomorphic in several variables if it is differentiable in all its variables in norm in a complex domain. Note that the argument of \cite[S2, Def. 1]{baumgartel_book} requires only the weaker notion of ``holomorphy of type A'', as it deals with the general case of unbounded operators. 
}
\item The operator $\mathfrak{L}_{0}$ admits a non-degenerate eigenvalue $\lambda \in \bbC$ separated from the rest of $\Sigma(\mathfrak{L}_{0})$ by a curve $\Gamma$.
\end{enumerate} 
Then, there exists a neighbourhood $V$ of $0$ and a \emph{holomorphic} map $z \in V \mapsto \lambda(z) \in \C$ such that $\lambda(0) = \lambda$, and for all $z \in V$, $\lambda(z)$ is a non-degenerate eigenvalue of $\frakL_{z}$ separated from the rest of $\Sigma(\mathfrak{L}_{z})$ by the same curve $\Gamma$.
Furthermore, if $\lambda$ satisfies the following maximality assumption:
\begin{equation}
\sup_{\mu \in \Sigma(\mathfrak{L}_{0}) \setminus \{\lambda\}} |\mu| < |\lambda|,
\end{equation}
then $\lambda(z)$ remains a dominant eigenvalue for $\frakL_{z}$, that is
\begin{align}\label{equ:max_kato}
\sup_{\mu \in \Sigma(\mathfrak{L}_{z}) \setminus \{ \lambda(z)\}} |\mu| < |\lambda(z)|.
\end{align}
\end{theorem}

\section{Holomorphic perturbation of the FK--Ruelle operator}
\label{sec:perturbed_ruelle}

The goal of this section is to derive the fundamental identity~\eqref{equ:corr_length_eigenvalue_Ruelle}, which will be the starting point of our analysis. 
Informally, it expresses $\fraky(p,x)$ for $p \in (0,p_c)$ and $x \in \Z^d\setminus \{0\}$ as the solution of an equation involving the spectral radius of a certain operator, called the \emph{FK-Ruelle operator}. 
This operator falls into the more general setting of Ruelle operators, which have been a very active topic of investigation; we refer for example to the reference works~\cite{parry_zeta_1990,sarig_thermodynamic_1999,baladi_positive_2000,ruelle_thermodynamic_2004}.
The operator we study is specific to FK-percolation, and was introduced in~\cite{campanino_fluctuation_2008}. 

\medskip

{\bf In the remainder of the paper, we fix $d \geq 2$, $q \geq 1$ and a subcritical parameter $p \in (0, p_c)$, which we shall often omit from the notation. We also choose a reference direction by fixing $t \in \partial \calW_p$. Our goal is to prove that $\fraky$ is analytic around $(p,t)$.}

\subsection{The classical FK--Ruelle operator}

\subsubsection{Irreducible cluster decomposition \emph{à la} Campanino--Ioffe--Velenik}

In this section we recall the definition of irreducible clusters given in \cite{campanino_fluctuation_2008}. 
We adopt the very elegant construction of the authors, based on the notion of \emph{surcharge}. 

We shall perform the construction in the fixed direction $t \in \partial \calW_p$, this means that we will work on cones constituted of all the points of $\bbZ^d$ whose \emph{surcharge} with respect to the fixed direction $t$ is smaller than a certain constant $\kappa$. More precisely, for any $v \in \Z^d$ and $\kappa \in (0,1)$, define the following sets: 
\begin{itemize}
\item The forward $t$-cone of surcharge $\kappa$ and of apex $v$:
\[ \calY_\kappa^{t,+}(v) \coloneqq v + \{ y \in \Z^d, \langle t, y \rangle \geq (1-\kappa)\frakm(p, y)\}.  \]
\item The backward $t$-cone of surcharge $\kappa$ and of apex $v$:
 \[ \calY_\kappa^{t,-}(v) \coloneqq v - \{ y \in \Z^d, \langle t, y \rangle \geq (1-\kappa)\frakm(p, y)\}. \]
 \item For $v,w \in \Z^d$, the $\kappa$-surcharged $t$-diamond of apexes $v$ and $w$:
\begin{equation}\calD^t_\kappa(v,w) \coloneqq  \calY^{t, +}_\kappa(v) \cap \calY^{t, -}_\kappa(w).\end{equation}
\end{itemize}

As in~\cite{campanino_fluctuation_2008}, if $\kappa$ is chosen too small, it may be the case that $\calY_\kappa^{t,+}(v)$ is degenerated, in the sense that the cone contains no edge of $\Z^d$ with endpoint $v$. 
We observe that one may always perform a \emph{local} surgery, so that the cone contains exactly one edge with endpoint $v$. By the symmetries of the lattice $\bbZ^d$, we may assume that this edge is $\{ v, v +(1, 0, \dots, 0) \}$. One may perform the same \emph{local} surgery on $\calY_\kappa^{t,-}(v)$, symmetrically. The only edge of $\calY_\kappa^{t,-}(v)$ with endpoint $v$ is therefore the edge $\{ v, v - (1, 0, \dots, 0) \}$. This local modification does not affect the asymptotic analysis of the associated Ruelle operator. We shall therefore always work implicitly with the modified $t$-cones and the resulting $t$-diamonds.
The careful reader might keep in mind that this local modification is necessary when the direction associated with the smallest surcharge relative to the chosen direction $t$ is not aligned with a principal axis of $\Z^d$, as illustrated in Figure~\ref{fig:clusters}, \textsc{top left}.

The set defined above will be used to define the notion of $t$-confined and $t$-irreducible clusters.
Let $\frakc = (V(\frakc), E(\frakc))$ be a finite connected subgraph of $\Z^d$, and let $\kappa \in (0,1)$. 
We say that $\frakc$ is \emph{$t$-diamond confined with surcharge} $\kappa$ if there exist $v,w \in V(\frakc)$ such that $V(\frakc) \subset \calD^t_\kappa(v,w)$.
In that case, $v$ (resp. $w$) is unique and is called $\mathsf{f}(\frakc)$ (resp. $\mathsf{b}(\frakc)$).
We refer to Figure~\ref{fig:clusters} for illustration.
The \emph{displacement} of a $t$-diamond confined cluster is defined by:
\begin{equation}
X(\frakc) \coloneqq  \mathsf{b}(\frakc) - \mathsf{f}(\frakc).
\end{equation}

Let $\frakc$ be $t$-diamond confined with surcharge $\kappa$. 
We say that $\frakc$ is $(\kappa, t)$-irreducible if it cannot be written as a union of two $t$-diamond confined subgraphs $\frakc_d$ and $\frakc_d'$, with surcharge~$\kappa$, such that $\mathsf{b}(\frakc_d) = \mathsf{f}(\frakc_d')$. We denote by $\frakC^t_\kappa$ the set of $(\kappa, t)$-irreducible subgraphs of $\Z^d$. 

\begin{figure}
\begin{center}
    \includegraphics[width = .40\textwidth]{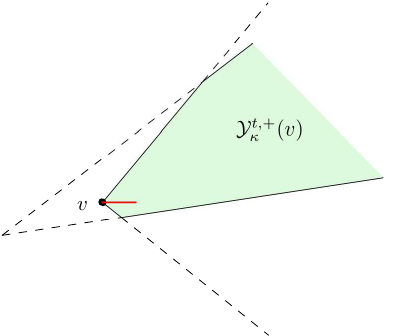} \hspace{.02\textwidth} 
    \includegraphics[width = .40\textwidth]{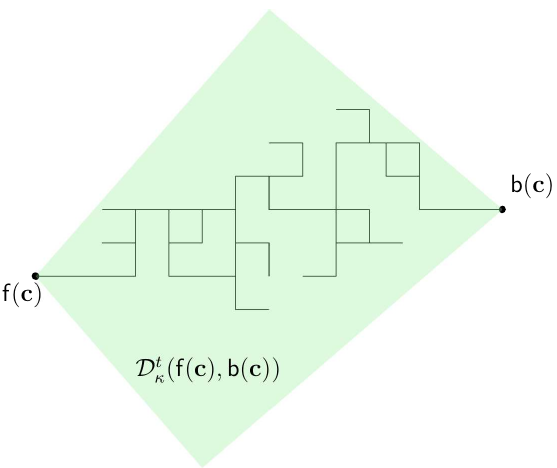}
        \includegraphics[width = .80\textwidth]{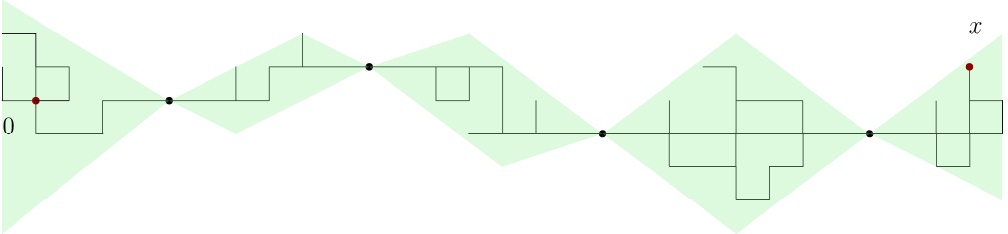}
    \caption{
    \textsc{top left}: Example of the local surgery needed for the non-triviality of the construction. The cone $\calY_\kappa^{t,+}(v)$ is enlarged so as to contain the red horizontal edge. 
    The actual asymptotic cone corresponds to the smaller opening. For simplicity, in remainder of the figure, we \textbf{choose $t$ in the direction of the first coordinate vector} since no local surgery is needed in this case. \textsc{top right}: A $t$-irreducible cluster. The green region denotes the $(t,\kappa)$-confining diamond region: note its particular shape.
    \textsc{bottom}: The diamond decomposition \emph{à la} Campanino--Ioffe--Velenik. 
    The vertices 0 and $x$ are represented in dark red. 
    Note that the first and the last pieces have a slightly different shape: 
    they are only backward and forward confined respectively. 
    The argument of \cite{campanino_ornstein-zernike_2003, campanino_fluctuation_2008} essentially shows that when considering the rate of decay of connection probabilities from $0$ to $x$ one can neglect these extra pieces and focus on the decompositions in which these two boundary pieces are trivial clusters.}
    \label{fig:clusters} 
\end{center}
\end{figure} 

Fix $\frakc \in \frakC^t_\kappa$. 
Define the \emph{outer edge boundary} of $\frakc$ as the set
\begin{equation} 
\partial E(\frakc) \coloneqq \{e = \{x,y\} \in \Z^d, x \in V(\frakc), y\notin V(\frakc)\}. 
\end{equation}
Observe that due to the $t$-diamond confinement and by the local surgery, the two edges 
$e_- \coloneqq \{ \mathsf{f}(\frakc), \mathsf{f}(\frakc) - (1, 0, \dots, 0) \}$ 
and 
$e_+ \coloneqq \{ \mathsf{b}(\frakc), \mathsf{b}(\frakc) +(1, 0, \dots, 0) \}$ 
automatically belong to $\partial E(\frakc)$ ; they will play a special role in what follows. 

The main interest of $\frakC_{\kappa}^t$ is that, the quantity
$
\e^{\langle t, x\rangle}\phi[ 0 \lra \lfloor nx \rfloor ]
$
is exponentially dominant when $x$ lies in a $t$-cone of small surcharge $\kappa$. Moreover, it is proven in \cite{campanino_ornstein-zernike_2003, campanino_fluctuation_2008} that for such values of $x$, this quantity admits a decomposition as a concatenation of irreducible elements $\frakc$ of $\frakC^t_{\kappa}$, as shown in Figure \ref{fig:clusters}. 
This suggests that the study of the polar dual of the mass $\fraky(p,x)$ for $x$ close to $t$ reduces to that of a stochastic process generating chains $\frakc_1,\frakc_2,\frakc_3,\dots$ of clusters of $\frakC^t_{\kappa}$.

\subsubsection{$\theta$-exponential Hölder functions on strings of irreducible clusters}

We adopt the framework of \cite{campanino_ornstein-zernike_2003, campanino_fluctuation_2008} to deal with chains of element of $\frakC^t_{\kappa}$. Let us briefly recall the main definitions of this framework.

Note that the trivial cluster $\frakc = \{0\}$ is also an element of $\frakC^t_{\kappa}$. Since this element plays a major role in the following, we shall denote it $\emptyset$. We work with the following state space of chains of clusters 
\begin{equation}
\Omega^t_{\kappa} \coloneqq \{ \bar{\frakc} \in (\frakC^t_\kappa)^{\Z_{>0}}:\forall k \geq 1, \mathsf{f}(\bar{\frakc}_k) = 0~\mbox{and}~  \forall k, \ell > 0,  \bar{\frakc}_k = \emptyset \Rightarrow \bar{\frakc}_{k+\ell} = \emptyset \},
\end{equation}
where $\bar\frakc_k$ denotes the $k$-th element of the chain $\bar\frakc$. The element $\emptyset$ is merely a formal device for completing finite chains of clusters into infinite ones. For two strings $\bar{\frakc}, \bar{\frakc}' \in \Omega_{\kappa}^t$, define their proximity index by 
\begin{equation} \bfi(\bar{\frakc}, \bar{\frakc}') \coloneqq \min\{k \geq 1, \bar{\frakc}_k \neq \bar{\frakc}'_k\}. \end{equation} 
For $f : \Omega_{\kappa}^t \rightarrow \C$, and $k \in \Z_{>0}$,
define the $k$-\emph{variation} of $f$ as
\begin{equation}
\Var_k(f) \coloneqq \sup_{\{\bar{\frakc}, \bar{\frakc}'~:~\mathbf{i}(\bar\frakc, \bar\frakc') \geq k\}} |f(\bar{\frakc}) - f(\bar{\frakc}')|.
\end{equation}
For any $\theta \in (0,1)$, define the following $\theta$-Hölder measurement:
\begin{equation}
|f|_\theta \coloneqq \sup_{k> 1} \frac{\Var_k(f)}{\theta^{k}} \in [0, +\infty]. 
\end{equation}
A function $f$ satisfying $|f|_\theta < \infty$ is called \emph{$\theta$-exponential Hölder}. 
Define the functional~set
\begin{equation}
\calF^t_{\kappa, \theta} \coloneqq \{ f  : \Omega_{\kappa}^t  \rightarrow \C ,~ \Vert f \Vert_\infty + |f|_\theta < \infty \}. 
\end{equation}
Note that $\calF^t_{\kappa, \theta}$ is a Banach space when equipped with the norm $\Vert f \Vert_\theta = |f|_\theta +~\Vert f \Vert_\infty$.

\subsubsection{The FK-Ruelle potential}

For $\frakc \in \frakC_{\kappa}^t \setminus \{\emptyset\}$, define the following percolation events:

 \begin{equation}\label{equ:evts_ins_out} \mathsf{Ins}(\frakc) \coloneqq \bigcap_{e\in E(\frakc)} \{ \om_e = 1\}
\quad\text{ and }\quad
 \mathsf{Out}(\frakc) \coloneqq \bigcap_{e\in \partial E(\frakc) \setminus\{e_-, e_+\} }  \{ \om_e = 0  \}
. \end{equation}
 
When $\frakc = \emptyset$, both events are defined to be the total event, the probability of which is~$1$. 
We are now ready to the define the main protagonist of this article. Recall that the parameters $p \in (0,p_c)$ and $t \in \partial\calW_p$ are fixed throughout the paper.

\begin{definition}[FK-Ruelle potential]
For $\kappa \in (0,1)$ and $s \in \R^d$, the \emph{FK-Ruelle} potential is defined as 
\begin{equation}\label{equ:ruelle_potential}
\psi^s_{ \Omega^t_\kappa, p} : \bar{\frakc}\in \Omega_{\kappa}^t \mapsto \log \phi^0_{\bar{\frakc}_1,p}[\mathsf{Ins}(\bar{\frakc}_1)] + \log \phi^0_p [ \mathsf{Out}(\bar{\frakc}_1) \; \vert \bigcap_{k=2}^\infty \mathsf{Out}(X_k + \bar{\frakc}_k)] + \langle s, X(\bar\frakc_1) \rangle  \in \R,
\end{equation}
where $X_k$ is defined by $X_k \coloneqq \sum_{i=1}^{k-1} X(\bar{\frakc}_i)$ and $X_k + \cdot$ denotes the translation by $X_k$.
\end{definition}

The finite-energy property of the model directly implies that the FK-Ruelle potential is a well defined real-valued function. 
Due to the translation by $X_k$ and the condition $\mathsf{f}(\bar{\frakc}_k) = 0$ in the definition of $\Omega_{\kappa}^t$, the clusters $(X_k + \bar{\frakc}_k)_{k \in \Z_{>0}}$ form a chain, where successive clusters share exactly one vertex (see Figure~\ref{fig:clusters}). 
Note that $\exp(\psi^t_{ \Omega^t_\kappa, p}(\bar\frakc) - \langle t, X(\bar\frakc_1) \rangle)$ is exactly the probabilistic cost of the first cluster of the chain $\bar{\frakc}_1$ conditionally on the rest of the chain. 
Thus, the FK--Ruelle potential is essentially a transition kernel for chains of irreducible clusters, with a bias on the displacement of $\bar\frakc_1$ in the direction given by $t$. 

We will prove the analyticity of $\fraky$ around the point $(p,t)$, by studying the FK-Ruelle potential around the same  point.
 
From the FK--Ruelle potential  $\psi^s_{ \Omega^t_\kappa, p}$, we move on to the following definition:

\begin{definition}[FK--Ruelle operator]
For $\kappa \in (0,1)$ and $s \in \R^d$, define an operator $\frakL_{\psi^s_{ \Omega^t_\kappa, p}}$ by setting, for any $f \in \calF^t_{\kappa, \theta}$ and $\bar\frakc \in \Omega^t_\kappa$, 
\begin{equation}\label{equ:def_FK_ruelle_operator}
[\frakL_{\psi^s_{ \Omega^t_\kappa, p}} f](\bar{\frakc})\coloneqq \sum_{\bar{\frakc}'\neq \emptyset:~\sigma(\bar{\frakc}') = \bar{\frakc}}\e^{\psi^s_{ \Omega^t_\kappa, p}(\bar{\frakc}')}f(\bar{\frakc}'), 
\end{equation}
where the \emph{Bernoulli shift} $\sigma$ is defined on $\Omega_{\kappa}^t$ by $\sigma(\bar{\frakc}')=(\bar{\frakc}'_{k+1})_{k\geq 1}$ and $\bar{\frakc}'\neq \emptyset$ means that the chain $\bar{\frakc}' \in \Omega_{\kappa}^t$ is not the trivial chain consisting only of trivial clusters $\emptyset$.
\end{definition}

\begin{remark}
The notation is unfortunately quite heavy, as we will need to perturb this operator in both $p$ and $s$, while the operator acts in the fixed functional set $\calF^t_{\kappa, \theta}$. 
As will be shown later, the operator is of main interest for values of the tilting parameter $s$ that are close to $t$. 
\end{remark}

{ \bf In the remainder of the paper, the base space $ \Omega^t_\kappa$ will be fixed, and thus dropped from the notation for readability. 
The reader should however keep in mind that $ \psi^s_{p}$ and $\frakL_{\psi^s_{ p}}$ are heavily $t$-dependent.}

\subsection{An holomorphic perturbation of the FK--Ruelle operator with complex mass-gap}

The novelty of our approach lies in the next result, which is the core of our method. 
Informally, we construct extensions of the operator $\frakL_{\psi^s_{ p}}$, in which both $p$ and $s$ are allowed to take \emph{complex} values in small neighbourhoods of real parameters.
Furthermore, those complex operators are shown to satisfy a \emph{complex mass-gap} condition, which will be important in what follows.

For $z \in \bbC$ and $\delta > 0$, we denote by $\bbD_\delta(z)$ the open disk of radius $\delta$ centered at $z$. 
When $z$ is the origin $0$, we shall just write $\bbD_\delta$. 
If $w$ is a vector of $\R^d$, then $w$ can be seen as an element of $\C^d$ and we still denote by $\D_{\delta}(w)$ the disc of radius $\delta$ centered at $w$ in $\C^d$. 
As its use will be recurrent in what follows, we introduce the following polydisk:
 \begin{equation}
 \calU_{\delta} \coloneqq \calU_{t,\delta} \coloneqq \bbD_{\delta} \times \bbD_{\delta}(t) \subset \C^{d+1}.
 \end{equation}
We will consider small perturbations $(p+z,s)$, with $(z,s) \in \calU_{\delta}$, around the point $(p,t)$. 
\newpage
 \begin{theorem}
 \label{thm:complex_reg_ruelle}
There exist $\kappa > 0$ and $\delta>0$ sufficiently small such that:
\begin{enumerate}
\item For any $\bar\frakc \in \Omega^t_\kappa$, 
the real function 
\begin{equation}
(z,s) \in  \calU_{\delta} \cap \R^{d+1} \mapsto \psi^{s}_{p+z}(\bar\frakc) \in \R
\end{equation}
is analytically extendable to the polydisk $\calU_{\delta}$.
\item The complex extension satisfies the following regularity properties:
\begin{enumerate}
\item There exists $\theta \in (0,1)$ such that  
\begin{equation}\label{equ:complex_theta_bound_potential}
\tag{$\bbC$-Hölder}
\sup_{(z,s) \in \calU_{\delta}} | \psi^{s}_{p+z} |_{\theta}  < +\infty.
\end{equation}
\item The following uniform bound is satisfied: 
\begin{equation}\label{equ:complex_summability_potential}
\tag{$\bbC$-MG}
\sup_{(z,s) \in \calU_{\delta}}  \sup_{\bar{\frakc} \in \Omega_\kappa^t} \sum_{\bar{\frakc}' \neq \emptyset:~\sigma(\bar{\frakc}') = \bar{\frakc}} \vert \e^{\psi^{s}_{p+z}(\bar{\frakc}')} \vert < +\infty.
\end{equation}
\end{enumerate}
 \end{enumerate}
 \end{theorem}

As suggested by the naming,~\eqref{equ:complex_theta_bound_potential} is a uniformly complex exponential Hölderness condition.
We call~\eqref{equ:complex_summability_potential} a \emph{complex mass-gap condition}.

\begin{remark}\label{remq:real_mass_gap}
Theorem~\ref{thm:complex_reg_ruelle} is an improvement of the verification of the \emph{mass-gap condition} in~\cite{campanino_fluctuation_2008}, which was crucially used to derive the Ornstein--Zernike asymptotic behaviour of the two-point function of subcritical FK-percolation. In particular, it was proved in this article that with the notation of Theorem~\ref{thm:complex_reg_ruelle}, there exists $\kappa, \delta > 0$ such that 
\begin{equation}\label{equ:real_theta_bound_potential}
\tag{$p$-Hölder}
\sup_{s \in \bbD_\delta(t)} | \psi^{s}_{p} |_{\theta}  < +\infty.
\end{equation}
and that
\begin{equation}\label{equ:velenik_mass_gap}
\tag{$p$-MG}
\sup_{s \in \bbD_\delta(t)} \sup_{\bar{\frakc} \in \Omega_\kappa^t} \sum_{\bar{\frakc}' \neq \emptyset:~\sigma(\bar{\frakc}') = \bar{\frakc}} \vert \e^{\psi^{s}_{p}(\bar{\frakc}')} \vert < +\infty.
\end{equation}

Thus, observe that Theorem~\ref{thm:complex_reg_ruelle} would be novel even if $z$ was chosen to lie in a real neighbourhood of 0.
\end{remark}

The proof of Theorem \ref{thm:complex_reg_ruelle} is the main innovation of this paper, and Section~\ref{sec:exp_bounds} is dedicated to it. 
We use it as our main tool for the next two subsections. 
We already note that Theorem~\ref{thm:complex_reg_ruelle} implies that the complexified operators $\frakL_{\psi_{p+z}^s}$ are all well-defined and bounded on the space $\calF^t_{\kappa, \theta}$. 
Indeed, this follows from the classical \emph{Lasota--Yorke inequality} that we recall below.

\begin{lemma}[Lasota--Yorke inequality] \label{lem:LY_inequality}
Call
\begin{equation} 
M \coloneqq \sup_{(z,s) \in \calU_{\delta}}  \sup_{\bar{\frakc} \in \Omega_\kappa^t} \sum_{\bar{\frakc}' \neq \emptyset: \sigma(\bar{\frakc}') = \bar{\frakc}} |\e^{\psi^{s}_{p+z}(\bar{\frakc}')}| < +\infty
\end{equation} 
Then, for any $f \in \calF^t_{\kappa, \theta}$, and $(z,s) \in \calU_{\delta}$,
\begin{equation}\label{equ:lasota-yorke}
\Vert \frakL_{\psi_{p+z}^s}  f \Vert_\theta \leq M\theta|f|_\theta + M(1+\theta\vert \psi_{p+z}^s \vert_\theta\e^{\theta\vert \psi_{p+z}^s\vert_\theta}) \Vert f \Vert_\infty.
\end{equation}
In particular, if $\delta>0$ is fixed according to Theorem \ref{thm:complex_reg_ruelle}, then the family $(\frakL_{\psi_{p+z}^s})_{(z,s) \in\calU_{\delta}}$ is uniformly bounded as a family of linear operators on $\calF^t_{\kappa,\theta}$. 
\end{lemma}

\begin{proof}
Let $f \in \calF^t_{\kappa,\theta}$, and fix $(z,s) \in \calU_{\delta}$. 
The uniform norm is immediately bounded as follows:
\begin{equation} \label{equ:bound_infinite_norm}
\Vert \frakL_{\psi_{p+z}^s} f \Vert_\infty \leq M\Vert f \Vert_\infty. 
\end{equation}

Let $k \in \Z_{\geq 1}$, and $\bar{\frakc}, \bar{\frakc}' \in \Omega$ satisfying $\bfi(\bar{\frakc}, \bar{\frakc}') = k$. 
We couple the preimages $\bar{\frakd}$ and $\bar{\frakd}'$ having the same first symbol in the next two sums, so that $\bfi(\bar{\frakd},\bar{\frakd}') = k+1$, to write
\begin{multline}\label{equ:1_LY}
 |\frakL_{\psi_{p+z}^s} f(\bar{\frakc}) - \frakL_{\psi_{p+z}^s} f(\bar{\frakc}')| \\ \leq  \sum_{\bar{\frakd} \neq \emptyset :\sigma(\bar{\frakd}) = \bar{\frakc} } | \e^{\psi_{p+z}^s(\bar{\frakd})}(f(\bar{\frakd})- f(\bar{\frakd}'))| + \sum_{\bar{\frakd}' \neq \emptyset:\sigma(\bar{\frakd}')=\bar{\frakc}'} |f(\bar{\frakd}')(\e^{\psi_{p+z}^s(\bar{\frakd})}-\e^{\psi_{p+z}^s(\bar{\frakd}')})|.
\end{multline}
The first sum in~\eqref{equ:1_LY} is bounded by $M|f|_\theta \theta^{k+1}$. For bounding the second sum, the mean value inequality yields
\begin{equation} |\e^{\psi_{p+z}^s(\bar{\frakd})} -\e^{\psi_{p+z}^s(\bar{\frakd}')}| \leq  \max \big\{\vert \e^{\psi_{p+z}^s(\bar{\frakd})} \vert,\, \vert e^{\psi_{p+z}^s(\bar{\frakd}')} \vert \big\} |\psi_{p+z}^s|_\theta \theta^{k+1}. \end{equation}
Now, observe that, as $\bar{\frakd}$ and $\bar{\frakd}'$ coincide on at least one symbol, then 
\begin{equation} 
\vert \e^{\psi_{p+z}^s(\bar{\frakd})} \vert \leq \e^{\theta|\psi_{p+z}^s|_\theta} \vert \e^{\psi^s_{p+z}(\bar{\frakd}')}\vert . 
\end{equation}
Hence we get the following upper bound for the second sum in~\eqref{equ:1_LY}:
\begin{equation}
|\psi_{p+z}^s|_\theta \theta^{k+1} \sum_{ \bar{\frakd}' \neq \emptyset:\sigma(\bar{\frakd}')=\bar{\frakc}'} |f(\bar{\frakd}')|\max \big\{\vert \e^{\psi_{p+z}^s(\bar{\frakd})} \vert,\, \vert e^{\psi_{p+z}^s(\bar{\frakd}')} \vert \big\}  \leq |\psi_{p+z}^s|_\theta \theta^{k+1} M \e^{\theta|\psi_{p+z}^s|_\theta} \Vert f \Vert_\infty.
\end{equation}
Combining those two bounds, and dividing by $\theta^{k}$ yields:
\begin{equation}\label{equ:final_LY}
|\frakL_{\psi_{p+z}^s} f|_\theta \leq M\theta(|f|_\theta + |\psi^s_{p+z}|_\theta\e^{\theta|\psi_{p+z}^s|_\theta}\Vert f \Vert_\infty). 
 \end{equation}
 
Summing~\eqref{equ:bound_infinite_norm} and~\eqref{equ:final_LY} concludes the proof of~\eqref{equ:lasota-yorke}.
\end{proof}

\subsection{An identity between the spectral radius of the FK--Ruelle operator and the mass}

In this subsection, we do not use the full power of the last subsection, but rather only study a \emph{real} perturbation of the operator $\frakL_{\psi_p^s}$. 
The aim is to prove~\eqref{equ:corr_length_eigenvalue_Ruelle}, which is an identity relating the inverse of the polar dual of the mass $\fraky(p,\cdot)$ with the spectral radius of the FK--Ruelle operator. 
In what follows, we will not make any further reference to the parameters $\kappa \in (0,1)$ and $\theta \in (0,1)$ , and will let them be implicitly fixed by Theorem~\ref{thm:complex_reg_ruelle}. 
We also fix $\delta > 0$ according to Theorem~\ref{thm:complex_reg_ruelle}.

Let us take a \emph{real} pair $(z,s) \in \calU_{\delta} \cap \bbR^{d+1}$. 
Recall that Lemma~\ref{lem:LY_inequality} implies that the operator $\frakL_{\psi_{p+z}^s}$ is bounded: we will be interested in its spectral radius, that we call $\rho_{t}(p+z,s)$ in what follows.
Recall that $t$ is a fixed vector that denotes the direction in which the irreducible decomposition is performed. 
The spectral radius $\rho_t$ is related to the mass of the model via the fundamental identity below. 

\begin{theorem} \label{thm:corr_length_eigenvalue_Ruelle}
For any $(z,s) \in \calU_{\delta} \cap\R^{d+1}$, the following relation holds: 
\begin{equation}\label{equ:corr_length_eigenvalue_Ruelle}
\tag{$\star$}
\rho_{t}(p+z,\fraky(p+z,s)\cdot s) = 1.
\end{equation}
\end{theorem}

\begin{remark}
Observe that, by definition of $\fraky$, the vector $\fraky(p+z,s)\cdot s$ is always on the boundary of the Wulff shape. 
This identity generalizes a result formally proved by Campanino, Ioffe and Velenik,
which is the content of Section 4.3 of~\cite{campanino_fluctuation_2008} (see also Section~3.3 of~\cite{campanino_ornstein-zernike_2003} for a detailed computation in a similar setup).
The only difference here is that we also perturb the operator in the $p$-direction. 
Given the improved mass-gap estimate given by Theorem~\ref{thm:complex_reg_ruelle}, the proof is extremely similar to theirs.
We give a formal proof below, based on a claim that follows from~\cite{campanino_ornstein-zernike_2003} for which we only give a heuristic explanation.
\end{remark}

\begin{proof}[Proof of Theorem~\ref{thm:corr_length_eigenvalue_Ruelle}]

For every $(z,s) \in \calU_{\delta} \cap\R^{d+1}$ define the function 
\begin{equation}\label{equ:tilted_susceptibility_power_series}
S_{z,s}: \alpha \in \R_{> 0} \mapsto \sum_{x \in \Z^d} \e^{\alpha \langle s, x \rangle} \phi^0_{p+z}[0 \longleftrightarrow x]. 
\end{equation}
By definition of the Wulff shape, it is clear that the series above is convergent (resp. divergent) if $\alpha< \fraky(p+z,s)$ (resp. $\alpha> \fraky(p+z,s)$). 

Recall the definition of the forward $t$-cone of surcharge $\kappa$ with apex $0$ at parameter $p$,
\[
\calY_\kappa^{t,+}(0) = \{ y \in \Z^d, \langle t, y \rangle \geq (1-\kappa)\frakm(p, y)\}.
\]
When $z = 0$ and $s = t$, it is clear from the definition of $\calY_\kappa^{t,+}(0)$ that the vertices $x$ outside of~$\calY_\kappa^{t,+}(0)$ can be neglected when studying the radius of convergence of \eqref{equ:tilted_susceptibility_power_series}. 
We now use the continuity of the mass $\frakm$ (it is the object of Lemma~\ref{lem:mass_continuous}, proved in the appendix), to argue that 
by reducing $\kappa$ if necessary, this remains true for all $(z,s) \in \calU_{\delta} \cap\R^{d+1}$. Thus, for all $(z,s) \in \calU_{\delta} \cap\R^{d+1}$, the convergence of the series $S_{z,s}$ is equivalent to the convergence of the series
\begin{equation}\label{equ:restrict_tilted_susceptibility_power_series}
\hat{S}_{z,s}: \alpha \in \R_{> 0} \mapsto \sum_{x \in \calY_\kappa^{t,+}(0)} \e^{\alpha \langle s, x \rangle} \phi^0_{p+z}[0 \longleftrightarrow x]. 
\end{equation}

We will conclude by using the following claim

\medskip
\noindent{\bf Claim.}  
For every $(z,s) \in \calU_{\delta} \cap\R^{d+1}$ and $\alpha > 0$, the series defining $\hat{S}_{z,s}(\alpha)$ is convergent if and only if the series 
\begin{equation}\label{equ:sum_operators_1}
\sum_{n\geq 1} [\frakL_{\psi_{p+z}^{\alpha s}}]^n \mathds{1}(\emptyset) 
\end{equation}
is convergent, where $\mathds{1}$ denotes the function that is constant equal to 1 (which is clearly an element of $\calF_{\kappa, \theta}^t$).\\

\medskip

A proof of this claim in the case where $(z,s) = (0,t)$ can be found in \cite[Section 3.3]{campanino_ornstein-zernike_2003}. 
The full claim follows by observing that the bounds on $\psi_{p+z}^s$ used in this proof remain valid for all $(z,s) \in \calU_{\delta} \cap\R^{d+1}$. 
We will omit the details of this verification and just claim that this follows from the uniform bounds given in Theorem \ref{thm:complex_reg_ruelle}. For convenience of the reader, a heuristic explanation of the claim is provided in the remark right below the end of the proof. 

It follows from the claim and the previous observations that the series in \eqref{equ:sum_operators_1} is convergent (resp. divergent) if $\alpha< \fraky(p+z,s)$ (resp. $\alpha> \fraky(p+z,s)$). We now remark that
\begin{itemize}
\item If $\rho_{t}(p+z,\alpha s) < 1$, the convergence of \eqref{equ:sum_operators_1} follows from the the upper bound of Gelfand's formula for bounded linear operators (see, \emph{e.g.},~\cite[Section 17.1]{lax}).
\item  Conversely, if $\rho_{t}(p+z,\alpha s) > 1$, the divergence of \eqref{equ:sum_operators_1} follows from the strict positivity of the eigenfunction in the Ruelle-Perron-Frobenius theorem presented in the next section (see Theorem \ref{thm:RPF}) and the fact that for any strictly positive $h$ in the eigenspace associated to the leading eigenvalue, we have
\begin{align*}
\Vert h\Vert_{\infty} [\frakL_{\psi_{p+z}^{\alpha s}}]^n \mathds{1}(\emptyset)  \geq [\frakL_{\psi_{p+z}^{\alpha s}}]^n h(\emptyset) \geq \rho_t(p+z,\alpha s)^n h(\emptyset).
\end{align*}
\end{itemize}
Conditionally on the continuity of the application $\alpha \mapsto \rho_{t}(p+z,\alpha s)$ at the point $\fraky(p+z,s)$ we directly obtain:
\begin{equation}
\rho_{t}(p+z,\fraky(p+z,s)\cdot s) = 1.
\end{equation}
The claimed continuity property is established in the next section. In fact, we will prove that the function $(z,s) \mapsto \rho_t(p+z, s)$ is analytic, and hence continuous, in a neighbourhood of $(0,t)$. Since $\fraky(p,t)=1$, by our choice of $t \in \partial \calW_p$, the continuity of~$\fraky$ (see Lemma \ref{lem:mass_continuous}) implies that, after possibly reducing $\delta$, the vector $\fraky(p+z,s)\cdot s$ is contained in this neighbourhood, for all $(z,s) \in \calU_{\delta}$.
\end{proof}

\begin{remark}[\bf Heuristic explanation of the claim]
Let us now give a heuristic explanation for the claim. We fix $(z,s) \in \calU_{\delta} \cap\R^{d+1}$ and $x \in \calY_\kappa^{t,+}(0)$. For $x$ far from $0$, the clusters contributing to the event $\{0 \longleftrightarrow x\}$ will typically admit a decomposition of the type
\begin{equation}
\calC = \frakc_\mathsf{b} \circ \frakc_1 \circ \frakc_2 \circ \cdots \circ \frakc_n\circ\frakc_\mathsf{f},
\end{equation}
where $\frakc_1,\dots,\frakc_n \in \frakC^t_\kappa$ and $\frakc_\mathsf{b},\frakc_\mathsf{f}$ are backward and forward confined clusters respectively. See Figure~\ref{fig:clusters}, \textsc{bottom}, for an illustration of a typical contribution to the series~\eqref{equ:restrict_tilted_susceptibility_power_series}. The notation $\frakc \circ \frakc'$ means that we glue the two clusters $\frakc$ and $\frakc'$ by identifying $\mathsf{b}(\frakc)$ with $\mathsf{f}(\frakc')$.
By summing over the possible clusters $\calC$ connecting $0$ to $x$, we get the following bound on the terms appearing in \eqref{equ:restrict_tilted_susceptibility_power_series}
\begin{equation}\label{equ:dec_ir}
\e^{\alpha \langle s, x \rangle}\phi_{p+z}[0 \longleftrightarrow x] \approx \sum\limits_{n \geq 1}\sum\limits_{\substack{ \frakc_1,\dots,\frakc_n \in \frakC^t_\kappa \\ 0 \leftrightarrow x}}\exp( \psi_{p+z}^{\alpha s }(\frakc_1) +  \psi_{p+z}^{\alpha s }(\frakc_1, \frakc_2) + \cdots + \psi_{p+z}^{\alpha s }(\frakc_1,\dots,\frakc_n ) ),
\end{equation}
where the last sum runs over the families of $n$ irreducible pieces $\frakc_1,\dots,\frakc_n$ connecting $0$ to~$x$, i.e. $\sum_{k=1}^n X(\frakc_k) = x$.
The sign $\approx$ comes from the fact that we neglect the contributions of $\frakc_\mathsf{b}, \frakc_\mathsf{f}$ as they do not affect the rate of decay in $x$; again we refer to~\cite{campanino_ornstein-zernike_2003} for the complete argument. Also note that we abused notation by writing $\frakc_1,\dots,\frakc_k$ instead of writing $\frakc_1,\dots,\frakc_k,\emptyset,\emptyset,\dots$.

Observing that for all $n \geq 1$
\begin{equation}
[\frakL_{\psi^{\alpha s}_{p+z}}]^n\mathds{1}(\emptyset) = \sum\limits_{\frakc_1,\dots,\frakc_n \in \frakC^t_\kappa}\exp( \psi_{p+z}^{\alpha s }(\frakc_1) +  \psi_{p+z}^{\alpha s }(\frakc_1, \frakc_2) + \cdots + \psi_{p+z}^{\alpha s }(\frakc_1,\dots,\frakc_n ) )
\end{equation}
and summing \eqref{equ:dec_ir} over the $x$ lying in $\calY_\kappa^{t,+}(0)$, we get
\begin{equation}
\hat{S}_{z,s}(\alpha) = \sum_{x \in \calY_\kappa^{t,+}(0)} \e^{\alpha \langle s, x \rangle}\phi_{p+z}[0 \longleftrightarrow x] \approx \sum\limits_{n \geq 1}[\frakL_{\psi^{\alpha s}_{p+z}}]^n \mathds{1}(\emptyset).
\end{equation}
This is exactly the statement of the claim.

\end{remark}

\section{Analyticity of $\rho_t$ and proof of Theorem~\ref{thm:analyticity_leg_m}}
\label{sec:conditional_proof}

We are now ready to provide a proof of Theorem~\ref{thm:analyticity_leg_m} conditionally on Theorem~\ref{thm:complex_reg_ruelle}.
The strategy of proof is the following: the interest of the relation given by~\eqref{equ:corr_length_eigenvalue_Ruelle} is that it identifies the quantity $\fraky(p+z, s)$ as the solution of an equation implying the spectral radius of a bounded operator. Our aim is now to apply the analytic version of the implicit function theorem, that will conclude on the analyticity of $\fraky$ around the point $(p,t)$ fixed at the beginning of the proof.
Thus, it remains to show that the spectral radius~$\rho_t(p+\cdot, \cdot)$ is a jointly analytic function in a neighbourhood of~$(0,t)$ (together with a condition on the non-vanishing of its derivative in its second coordinate). 

To this end, we first use a deep result, once again provided by~\cite{campanino_ornstein-zernike_2003} and called a \emph{Ruelle--Perron--Frobenius theorem}, asserting that when $s$ is a vector of $\R^d$ close to $t$, the spectral radius~$\rho_t(p, s)$ of the unperturbed operator $\frakL_{\psi^{p}_{s}}$ is actually a non-degenerate eigenvalue. That fact established, we may use Kato's theory to provide an analytic extension of this eigenvalue when considering the complex perturbation of the operator constructed in the last section. This shows that the spectral radius is analytic, and we conclude using equation~\eqref{equ:corr_length_eigenvalue_Ruelle}.

\subsection{The Ruelle--Perron--Frobenius theorem for the FK--Ruelle operator}
The following result is proved in~\cite{campanino_ornstein-zernike_2003}:

\begin{theorem}[Ruelle--Perron--Frobenius theorem]
\label{thm:RPF}
For any $(z,s) \in \calU_{\delta}\cap\R^{d+1}$, the spectral radius $\rho_t(p+z,s)$ is a non-degenerate (in the sense of Section~\ref{sec:kato}) eigenvalue of the operator~$\frakL_{\psi_{p+z}^s}$.
Moreover, $\rho_t(p+z,s)$ is dominant, in the sense that
\begin{equation}\label{eq:dominant_RPF_theorem}
\sup_{\lambda\in \Sigma(\frakL_{\psi^s_{p+z}})\setminus\{\rho_t(p+z,s)\} } |\lambda| < \rho_t(p+z,s), 
\end{equation}
and the corresponding eigenspace is spanned by a strictly positive function $h_{z,s} \in \calF^t_{\kappa,\theta}$. 
\end{theorem}

Such a result is usually called a \emph{Ruelle--Perron--Frobenius theorem}, and is classical in the literature~\cite{sarig_thermodynamic_1999,baladi_positive_2000}.
However, the specific Theorem~\ref{thm:RPF} had to be reproved in~\cite{campanino_ornstein-zernike_2003}, due to contextual differences with the latter references (the presence of the transient element $\emptyset$ makes the shift reducible, as opposed to the classical settings). 
\begin{remark}
The only place in which the uniformity in $z$ and in $s$ stated in Theorem~\ref{thm:RPF} is used is in the proof of Theorem~\ref{thm:corr_length_eigenvalue_Ruelle}. 
However, the careful reader will note that the principal use of this result is to derive the non-degeneracy of the principal eigenvalue at the point $(p,t)$.  
\end{remark}
\begin{proof}
This theorem was proved on abstract Ruelle operators satisfying a set of assumptions in~\cite{campanino_ornstein-zernike_2003}.
These assumptions are exactly the restriction of the conclusions of Theorem \ref{thm:complex_reg_ruelle} to the real parameters $(z,s)$. 
\end{proof}

\subsection{Proof of Theorem~\ref{thm:analyticity_leg_m}}

We can now conclude using Kato's theory from Section~\ref{sec:kato}.
Recall that we work at a fixed parameter $p<p_c$ and with a fixed $t \in \partial\calW_p$ and that the parameters $\kappa, \theta \in (0,1)$ and  $\delta > 0$ are fixed according to Theorem~\ref{thm:complex_reg_ruelle}. 

\begin{proof}[Proof of Theorem~\ref{thm:analyticity_leg_m}]

Lemma~\ref{lem:LY_inequality} ensures that the family $(\frakL_{\psi_{p+z}^s})_{(z,s) \in\calU_{\delta}}$ is uniformly bounded as a family of linear operators on $\calF^t_{\kappa,\theta}$.

\medskip
\noindent{ \bf Step 1: Holomorphicity of $\frakL_{\psi^s_{p+z}}$.}

The holomorphicity of this family is immediate. 
We use here the notion of ``holomorphicity of type A'' which is the object of Definition 1 in~\cite[S2]{baumgartel_book}.
Indeed, the operators $\frakL_{\psi^s_{p+z}}$ have the same domain.
Moreover, for any $f \in \calF^t_{\kappa,\theta}$, the map
\begin{equation} (z,s) \in  \calU_{\delta} \mapsto \frakL_{\psi_{p+z}^{s}} f \in  \calF^t_{\kappa,\theta} \end{equation} 
is holomorphic, since Theorem~\ref{thm:complex_reg_ruelle} implies that it is the sum of uniformly bounded holomorphic functions.

\medskip
\noindent{ \bf Step 2: Holomorphicity of the leading eigenvalue $\lambda(z,s)$.}

By step 1, the map $(z,s) \in \calU_{\delta} \mapsto \frakL_{\psi^s_{p+z}}$ is holomorphic.
Furthermore, by Theorem~\ref{thm:RPF}, the non-perturbed operator $\psi^{t}_p$ has a non-degenerate and isolated eigenvalue $\lambda(0, t)$ equal to the spectral radius $\rho_t(p,s)$ that is dominant in the sense given by \eqref{eq:dominant_RPF_theorem}.
Thus, we may apply Kato's theorem~\ref{thm:kato} to conclude that there exists a map $(z,s) \in \calU_{\delta} \mapsto \lambda(z,s)$, defined in a possibly smaller polydisk (we adapt the value of $\delta $ accordingly) such that
\begin{itemize}
\item $\lambda$ is holomorphic in $\calU_{\delta}$.
\item For any $(z,s) \in \calU_{\delta}$, $\lambda(z,s) \in \C$ is a non-degenerate, dominant eigenvalue of the operator $\frakL_{\psi_{p+z}^s}$.
\end{itemize}
In particular, for all $(z,s) \in \calU_{\delta} \cap \R^{d+1}$,
\begin{align}\label{equ:identification_rho_lambda}
\lambda(z,s) = \rho_t(p+z,s),
\end{align}
and $(z,s) \mapsto \rho_t(p+z,s)$ is jointly analytic on $\calU_{\delta} \cap \R^{d+1}$. 

Recall that $t$ was fixed on the boundary of the Wulff shape $\calW_p$, which means that $\fraky(p,t) = 1$. 
In order to apply the holomorphic version of the implicit function theorem, we introduce the complex function $\Lambda : \D_{\delta} \times \D_{\delta}(t) \times \D_{\delta}(1) \to \bbC$ defined by
\begin{align*}
\Lambda(z,s,r) = \lambda(z,rs).
\end{align*}
Clearly, the joint holomorphicity of $\lambda$ around $(0,t)$ implies the joint holomorphicity of $\Lambda$ around $(0, t, 1)$.

\medskip
\noindent{\bf Step 3: Non-vanishing of the partial derivative of $\Lambda$ in $r$ at $(0,t,1)$.}

We now justify that 
\begin{equation}
\frac{\partial \Lambda}{\partial r}\Big|_{(0, t , 1)} \neq 0. 
\end{equation}
Recall from Theorem~\ref{thm:RPF} that the operator~$\frakL_{\psi_p^{t}}$ possesses a uniformly positive eigenvector $h \coloneqq h_{0, t} \in\calF^t_{\kappa,\theta}$ associated to the eigenvalue $\lambda(0,t)$. 
Due to the specific form of the Ruelle potential and by the aforementioned positivity property of $h$, there exists a constant $\varsigma>0$ such that for any $\eta > 0$,
\begin{equation}
\frakL_{\psi_p^{(1+\eta)t}} h  \geq e^{\varsigma \eta}  \frakL_{\psi_p^{t}} h = e^{\varsigma \eta } \lambda(0,t) h.
\end{equation}
Since $\Vert\cdot\Vert_\theta \geq \Vert\cdot\Vert_\infty$, it follows that
\begin{align*}
\lim_{k \to \infty} \Big(\Vert (\frakL_{\psi^{(1+\eta)t}_p}^k h \Vert_\theta \Big)^{1/k} \geq e^{\varsigma \eta} \lambda(0,t) .
\end{align*}
By the easy upper bound of Gelfand's formula for bounded linear operators (see for instance~\cite[Section 17.1]{lax}), we get that\footnote{One may also look at~\cite[equation (4.9)]{campanino_ornstein-zernike_2003} for a much stronger inequality.}:
\begin{equation}
\lim_{k \to \infty} \Big(\Vert (\frakL_{\psi^{(1+\eta)t}_p}^k h \Vert_\theta \Big)^{1/k} \leq \lambda(0,(1+\eta)t).
\end{equation}
Combining the two inequalities above, we finally get that $\Lambda(0,t ,1+\eta) \geq e^{\varsigma \eta} \Lambda(0,t,1)$, which in turn implies that 
\begin{equation}
\frac{\partial \Lambda}{\partial r}\Big|_{(0, t, 1)} \geq \varsigma \Lambda(0, t, 1) > 0,
\end{equation}
which is the condition needed to apply the implicit function theorem to $\Lambda$ around $(0,t,1)$.
\medskip
\noindent{\bf Step 4: Conclusion.}
By equation \eqref{equ:corr_length_eigenvalue_Ruelle} from Theorem \ref{thm:corr_length_eigenvalue_Ruelle} and equation \eqref{equ:identification_rho_lambda}, the following relation holds for all $(z,s) \in \calU_{\delta} \cap \R^{d+1}$
\begin{align*}
\Lambda(z,s,\fraky(p+z,s)) = 1.
\end{align*}
Step 2 and 3 allows us to apply the holomorphic version of the implicit function theorem to $\Lambda$ at the point $(0,t, 1)$. 
This directly shows that the function $\fraky$ is jointly analytic around the point $(p,t)$. 
Recall that $t$ was chosen to lie on $\partial \calW_p$. 
However, since $\fraky(p,\cdot)$ is homogeneous of degree $-1$, the result extends to every $(p,t) \in (0,p_c) \times \R \setminus \{0\}$. 
So, the function~$\fraky$ possesses an analytic extension defined on a complex neighbourhood of its domain of definition $(0,p_c) \times \bbR^{d} \setminus \{0\}$. 
This implies that $\fraky$ is jointly analytic on $(0,p_c) \times \R^d \setminus \{0\}$, which completes the proof.
\end{proof}

\section{Regularity of the complexified Ruelle potential via cluster expansion}
\label{sec:exp_bounds}

To conclude, it remains to prove the main result about the regularity of the complex Ruelle potential, namely Theorem~\ref{thm:complex_reg_ruelle}.
To this end, we will use a cluster expansion method, very similar to~\cite{ott_weak_2020-1,dalimonte_analyticity_2026}. We refer to the monograph \cite{book_velenik} for a gentle introduction to cluster expansion.

\subsection{Setup}
Let us state some important convention used in the rest of the paper.

\begin{itemize}
\item We recall that \textbf{the parameter $p < p_c$ is a fixed subcritical parameter}.
As such, the dependency of the quantities in the parameter $p$ is implicit, and we do not keep track of it.

\item A subgraph of $\bbZ^d$ is a graph $G$ with vertex set $V(G)$ included in $V(\Z^d)$ and edge set $E(G)$ included in $E(\Z^d)$. Note that a subgraph of $\bbZ^d$ is \textbf{not necessarily finite}.

\item When $G$ is a \textbf{finite} subgraph of $\Z^d$. 
The rational function $p \mapsto \mathbf{Z}^0_{G,p}$ is well defined and can be naturally extended to an analytic function on the complex domain~$\C \setminus \{1\}$. 
For $z \in \C$, we shall denote by $\mathbf{Z}^0_{G,p+z}$ the value of this extension at the point $p+z$. 
In a similar fashion, we denote by $\phi^0_{G,p+z}$ the analytic extension of the measure $\phi^0_{G,p}$ to the complex plane. 
This extension is well-defined for all $z$ such that $\mathbf{Z}^0_{G,p+z}$ is non zero.
\end{itemize}

\subsection{Graph restriction and cluster expansion}

We start by importing the cluster expansion formalism developed in~\cite{ott_ising,ott_weak_2020-1} and used in~\cite{dalimonte_analyticity_2026} to provide analytic extensions of probabilities of local events. 
In these articles, the cluster expansion method is performed on $d$-dimensional tori. 
In what follows, we will need to be slightly more general, and to perform the cluster expansion on general subgraphs of $\Z^d$. 

\subsubsection{Graph restriction and mixing}

A key ingredient for performing the cluster expansion is the exponential weak-mixing property of subcritical FK-measure.
In order to extend the cluster expansion method to arbitrary subgraphs of $\bbZ^d$, we shall rely on the following proposition that shows that weak mixing of subcritical percolation remains valid after restricting the underlying graph.

\begin{lemma}
\label{lem:unif_total_variation_distance}
Let $a > 1$. There exist two constants $C, c>0$ such that, for any subgraph $G$ of $\bbZ^d$, any $N \geq 0$ and any $x \in \bbZ^d$, if $\xi$ is a boundary condition on $\bbZ^d$ \emph{whose edges outside of~$E(G)$ are closed}, then
\begin{equation}
\mathrm{d}_{\mathsf{TV}} \big( \phi^{\xi}_{G\cap \Lambda_{a N}(x), p} ( \cdot_{|G\cap \Lambda_N(x)} ), \phi^{0}_{G\cap \Lambda_{a N}(x), p}( \cdot_{|G\cap \Lambda_N(x)} ) \big) < C\exp(-cN), 
\end{equation}
where $\mathrm{d}_{\mathsf{TV}}$ denotes the distance in total variation.
\end{lemma}

This lemma follows directly from the stronger ratio weak mixing stated in Lemma \ref{lem:ratio_mixing_uniform_G} and exponential decay of connection probabilities in FK-percolation. 

Note that it is necessary that the edges of $\xi$ outside of $E(G)$ are closed; otherwise, one may have a discrepancy of boundary conditions at distance $0$ from $G\cap \Lambda_N(x)$, which would result in a large total variation distance between the two measures. 
This condition on $\xi$ and the intersection with $G$ taken in the above FK-percolation measure have to be interpreted as replacing the underlying graph $\bbZ^d$ by the subgraph~$G$.
 
The fact that the bound is uniform over the subgraphs $G$ of $\bbZ^d$ actually follows from the fact that removing edges from $\bbZ^d$ by restricting the model to a subgraph $G$ can only decrease the probabilities of connection.

\subsubsection{Compatibility of the cluster expansion with the graph restriction operation}

We now slightly extend the cluster expansion representation of~\cite{ott_weak_2020-1} by showing that the weights are compatible with the graph restriction operation. Fix $L \in \Z_{>0}$. Following the renormalisation scheme performed as a preliminary to cluster expansion in~\cite{ott_weak_2020-1}, we define the~$L$-renormalised lattice~$\bbZ^{d}_L$ as the graph with vertex set $((2L)\bbZ)^d$ and edges between sites at distance $2L$ (for the graph distance in~$\bbZ^d$). Since $\bbZ^{d}_L$ can be naturally embedded in $\bbZ^d$, vertices of $\bbZ^{d}_L$ will sometimes be viewed as vertices of $\bbZ^d$. For a given subgraph $G$ of $\Z^d$, we define the distance between~$v \in \Z^d_L$ and $G$ as the graph distance in~$\Z^d$ between~$v$ and~$G$, viewing $v$ as a vertex of $\Z^d$. The polymer set of $G$ is defined by
\begin{align*}
\sfP^L_{G} \coloneqq \Big\{ \gamma,~ 
\begin{array}{ll}
\gamma \mbox{ is a finite, connected and non empty subgraph of } \bbZ^{d}_L \\
\mbox{and every vertex of } \gamma \mbox{ is at distance at most } L \mbox{ from } G
\end{array}
\Big\}.
\end{align*}
For $\gamma$ in $\sfP^L_{G}$, we denote by $|\gamma|$ the number of vertices of $\gamma$. Finally, we denote by $\mathsf{Span}(\gamma)$ the subgraph of $\bbZ^d$ obtained by taking the union of all the boxes of radius $L$ centered at a vertex of $\gamma$, i.e.,
\begin{align*}
\mathsf{Span}(\gamma) = \bigcup\limits_{v \in \gamma} \Lambda_L(v)
\end{align*}
The following result states that the cluster expansion of~\cite{ott_weak_2020-1} can be performed uniformly over the subgraphs of $\bbZ^d$ with compatible cluster weights $\sfw_{G,z}(\gamma)$.

\begin{proposition}\label{prop:clust_exp_comp}
For all $c > 0$, there exists $L \in \Z_{>0}$ and $\delta > 0$ such that for any $0 <\eps < \delta$, there exists a constant $C_\eps>0$ going to 0 when $\eps$ goes to 0 such that the following is true. 
For any finite subgraph $G$ of $\bbZ^d$, for any $z \in \bbD_\eps$, there exists a weight function $\sfw_{G,z} : \sfP^L_{G} \rightarrow \bbC$ satisfying: 
\begin{enumerate}
\item  {\bf Uniform bound.} For any $\gamma \in \sfP^L_{G}$ and any $z \in \bbD_\eps$,
\begin{equation}\label{equ:bound_weight}
 |\sfw_{G,z}(\gamma)| \leq C_\eps e^{-|\gamma|c}. 
 \end{equation}
\item {\bf Subgraph compatibility.} For any finite subgraph $G'$ of $\bbZ^d$, any polymer $\gamma \in \sfP^L_{G} \cap \sfP^L_{G'}$ such that $\mathsf{Span}(\gamma) \cap G = \mathsf{Span}(\gamma) \cap G'$ and any $z \in \bbD_\eps$, 
\begin{equation} \sfw_{G,z}(\gamma) = \sfw_{G',z}(\gamma). \end{equation}
\item {\bf Convergence of the cluster expansion.} The following formula holds true
\begin{equation}\label{equ:clust_exp_subgraph}
\frac{
\mathbf{Z}^0_{G,p+z}
}
{
\mathbf{Z}^0_{G,p}
}
=
\exp\big( \sum_{n\geq 1} \sum_{\gamma_1, \dots, \gamma_n \in \sfP^L_{G}} \varphi_n(\gamma_1, \dots, \gamma_n) \prod_{k=1}^n \sfw_{G,z}(\gamma_k) \big),
\end{equation} 
where $\varphi_n$ is the Ursell function on ordered collections of polymers. In particular the series in the right-hand side converges. 
\end{enumerate}  
\end{proposition}

\begin{proof}[Sketch of proof]

This follows by a close inspection of the construction of~\cite{ott_weak_2020-1} (see also~\cite{ott_ising}). The construction of~\cite{ott_weak_2020-1} shows that the above result holds true when the graphs $G$ and $G'$ are chosen to be $d$-dimensional tori instead of general subgraphs of $\bbZ^d$. So, we shall just explain how the constructions and bounds of~\cite{ott_weak_2020-1} can be extended uniformly to the class of general subgraphs of $\bbZ^d$. 

Let us fix a renormalisation scale $L>0$. The information percolation and the $\pm 1$~expansion described in~\cite{ott_weak_2020-1} readily adapts to the case of subgraphs $G$ by replacing subboxes of the torus by their intersections with $G$ and replacing the polymer set of the torus by the polymer set $\sfP^L_{G}$ of $G$. It follows that the ratio of partition functions at $p+z$ and $p$ can be rewritten in the form of a polymer model as follows
\begin{align}\label{equ:ratio_as_polymer_model}
\frac{
\mathbf{Z}^0_{G,p+z}
}
{
\mathbf{Z}^0_{G,p}
} = \sum\limits_{\bar \gamma \subset \sfP^L_{G}} \prod\limits_{\gamma,\gamma' \in \bar \gamma} \delta(\gamma,\gamma') \prod\limits_{\gamma \in \bar \gamma}  \sfw_{G,z}(\gamma),
\end{align}
where $\delta(\gamma,\gamma')$ is $0$ when $\gamma \cup \gamma'$ is connected and $1$ otherwise. In the above, the weight $\sfw_{G,z}(\gamma)$ of a polymer $\gamma$ only depends on an explicit Glauber dynamic performed on the graph $\mathsf{Span}(\gamma) \cap G$. The {\bf Subgraph compatibility} \emph{(2)} follows from this observation simply by choosing a coupling of the Glauber dynamics on $G$ and $G'$ where each edge of $\bbZ^d$ that is present in both $G$ and $G'$ is attributed the same Poisson point process of resampling times and the same family of uniform variables in the two Glauber dynamics.

Inspecting the construction of~\cite{ott_weak_2020-1} reveals that the {\bf Uniform bound} \emph{(1)} obtained in the case of the torus relies solely on the following two bounds:
\begin{itemize}
\item[(i)] \textbf{Exponential mixing of the Glauber dynamic.}
There exists $c > 0$, $\alpha < \infty$ and $N_0 > 0$ such that for any $N \geq N_0$ and for any edge $e$ incident to the origin
\begin{align}
\bbP(\sigma^{0,0}_{\alpha N,\Lambda_N}(e) \neq \sigma^{1,1}_{\alpha N, \Lambda_N}(e)) \leq \exp(-cN),
\end{align}
where $\sigma^{\omega,\bar\eta}_{t,\Lambda}$ stands for the Glauber dynamic performed in the graph $\Lambda$ with boundary conditions $\omega,\bar\eta$ evaluated at time $t$, as defined in~\cite{ott_weak_2020-1}.
\item[(ii)] \textbf{Exponential decay of connections.}
For any $a>1$, there exists $c>0$ and $N_0 > 0$ such that for any $N \geq N_0$
\begin{align}
\phi^1_{\Lambda_{aN},p}(A_N) \leq \exp(-cN),
\end{align}
where $A_N$ is the event that $\Lambda_N$ contains a cluster of diameter larger than $N/100$.
\end{itemize}

For a given subgraph $G$ of $\bbZ^d$ and a vertex $x$ of $\bbZ^d$, these two bounds remain valid upon replacing every occurrence of $\Lambda_N$ by $G \cap \Lambda_N(x)$, and replacing the boundary conditions $1$ by the boundary conditions that equals $1$ on the edges of $G$ and $0$ elsewhere\footnote{Equivalently, this amount to replacing the underlying torus by the graph $G$}. Moreover, the resulting bounds are uniform over all choices of $G$ and $x$, meaning that the constants $c$ and $N_0$ can be chosen independently of $G$ and $x$.
Indeed, the uniformity of bound (i) is a consequence of the estimate given by Lemma~\ref{lem:unif_total_variation_distance}, which allows to perform the argument\footnote{In our setup, the assumption of quasi-transitivity in \cite{harel_spinka} is replaced by the uniformity in $x$ of the mixing property.} of \cite[Theorem 16]{harel_spinka} uniformly over the subgraphs $G$ of $\bbZ^d$. The uniformity of bound~(ii) is a direct consequence of the monotonic properties of the model and exponential decay of the two-point function. 

This implies that, if $L$ is chosen large enough, then the {\bf uniform bound}~\emph{(1)} obtained for tori remains valid on any subgraph $G$ of $\bbZ^d$. The {\bf convergence of the cluster expansion}~\emph{(3)} for the polymer model \eqref{equ:ratio_as_polymer_model} follows from the {\bf uniform bound}~\emph{(1)} and a classical condition for convergence of cluster expansion~\cite[Theorem 5.4]{book_velenik} (the same condition is used in ~\cite{ott_weak_2020-1}).

\end{proof}


\subsection{Uniform bounds on the logarithmic derivative of probabilities of unique decoupling events}

We next introduce a specific class of percolation events for which we will be able to prove strong analyticity properties of the corresponding probabilities.

\begin{definition}[Unique decoupling event]
\label{def:UDE}
Let $G$ be a subgraph of $\Z^d$. 
A percolation event $A$ is called a \emph{unique decoupling event} (UD-event) for $G$ if the following properties hold:
\begin{itemize}
\item {\bf Locality.} $A$ is measurable with respect to the state of finitely many edges of $G$. 
The support of $A$ shall be called $\mathsf{Supp}_A \subset G$.
\item {\bf Uniqueness.} There exists a unique configuration $\om_A$ on the edges of $\mathsf{Supp}_A$ realising the event $A$. This means that a configuration $\omega$ on $G$ realises the event $A$ if and only if $\om$ coincide with $\omega_A$ on $\mathsf{Supp}_A$. 
\item {\bf Decoupling property.} There exists a number $n_A \in \Z_{\geq 0}$ such that for any finite subgraph $H$ of $G$ containing all the edges in $\mathsf{Supp}_A$, for any $\om$ percolation configuration on $H$ realising $A$, then 
\begin{equation}
	k|_{H}(\om) = k|_{H\setminus \mathsf{Supp}_A}(\om) + n_A. 
\end{equation}
where for any subgraph $H'$ of $H$, $k|_{H'}(\om)$ denotes the number of components of $\om$ when restrained to the subgraph $H'$.
\end{itemize}
Note that we abused notation by writing $\mathsf{Supp}_A \subset G$ instead of $\mathsf{Supp}_A \subset E(G)$ and by writing $H \setminus \mathsf{Supp}_A$ to denote the graph with vertex set $V(H)$ and edge set $E(H) \setminus \mathsf{Supp}_A$. We shall always do so in the rest of the paper to simplify notation.

\end{definition}

\begin{remark}\label{rem:ins_out_UDE}
If $A$ is the event that a prescribed finite pattern $\om_A$ occurs, and if every connected component of the pattern $\om_A$ has at most one intersection point with each connected component of $G \setminus \mathsf{Supp}_A$, then $A$ is a UD-event.
Observe that the events we have just described are exactly the events $A$ corresponding to the realisation of a finite pattern that imposes free boundary conditions on the graph induced by the complementary of $\mathsf{Supp}_A$. 
\end{remark}

\begin{remark}\label{rq:rewriting_psi_log_UDE}
Fix some $s\in \R^d$.
One interesting consequence about the FK-Ruelle potential $\psi_p^s$ defined in Section \ref{sec:perturbed_ruelle} is the following: the spatial Markov property of the FK-percolation measure allows to write that for any $\bar\frakc \in \Omega^t_\kappa$, 
 \begin{equation}\label{equ:markov_for_potential}
 \psi^s_p(\bar\frakc) = \log \phi^0_{\bar\frakc_1,p}[\mathsf{Ins}(\bar\frakc_1)] + \log \phi^0_{G_{\bar\frakc},p}[\mathsf{Out}(\bar\frakc_1)] + \langle s, X(\bar\frakc_1) \rangle , 
 \end{equation}
 where $G_{\bar\frakc}$ denotes the subgraph of $\Z^d$ obtained by removing all the edges that are in the support of the event $\bigcap_{k=2}^{\infty} \mathsf{Out}(X_k + \bar\frakc_k)$.
The rewriting provided by~\eqref{equ:markov_for_potential} has the advantage of expressing the FK-Ruelle potential as the sum of the logarithm of two UD-events for the base graph of the measures, plus the harmless term $\langle s, X(\bar\frakc_1) \rangle$. 
 Indeed, 
 \begin{itemize}
\item $\mathsf{Ins}(\bar\frakc_1)$ is trivially a UD-event for the graph $\bar\frakc_1$ (the decoupling property follows directly from the fact that its support is $\bar\frakc_1$).
\item $\mathsf{Out}(\bar\frakc_1)$ is a UD-event for $G_{\bar\frakc}$ because it is the event that all the edges in a specific finite set are closed. 
 \end{itemize} 
 Representation~\eqref{equ:markov_for_potential} combined with this observation will be particularly useful in regard of Theorem~\ref{thm:clust_exp_UDE}, that is a convergent cluster expansion for logarithms of probabilities of UD-events. 
\end{remark}

The interest of considering UD-events lies in the following elementary property. 

\begin{lemma}\label{lem:dec_UDE}
Fix a subgraph $G$ of $\bbZ^d$. Take a \emph{finite} subgraph $H$ of $G$, and a UD-event~$A$ for $G$ such that $\mathsf{Supp}_A \subset H$. 
Then, provided that $\mathbf{Z}^0_{H,p+z} \neq 0$,
\begin{equation}
\frac{ \phi^0_{H,p+z}[A] }{\phi^0_{H,p}[A] } = \alpha_z^{o(\om_A)}\cdot \frac{\mathbf{Z}^0_{H,p}}{\mathbf{Z}^0_{H,p+z}} \cdot \frac{\mathbf{Z}^0_{H\setminus\mathsf{Supp}_A,p+z}}{\mathbf{Z}^0_{H\setminus\mathsf{Supp}_A,p}},  
\end{equation}
where $\alpha_z = (1 + \frac z p)(1 - \frac{z}{1-p})^{-1}$. 
\end{lemma}
\begin{proof}
This is a straightforward computation, as the second and third properties in the definition of UD-events allow to split partition functions as follows: 
\begin{equation} \mathbf{Z}^0_{H,p+z} \cdot \phi^0_{H, p+z}[A] = q^{n_A} ( \tfrac{p+z}{1-(p+z)})^{o(\om_A)} \cdot \mathbf{Z}^0_{H\setminus\mathsf{Supp}_A,p+z}. \end{equation}
Applying this with $z=0$ and dividing both identities yields the announced result. 
\end{proof}

The interest of this specific class of events is that one is able to perform a cluster expansion of their complex probabilities, \emph{which radius of convergence is uniform in the size of their support}\footnote{At first sight, this might seem to be covered by the results of~\cite{dalimonte_analyticity_2026}. However, these results are not of the same nature. The cited paper performs a ``resummation trick'' to see that one may construct analytic extensions of the probability of any local event with a uniform radius of convergence, but it does not express them in terms of convergent cluster expansions. As such, the following proposition yields a stronger result, but on a more restricted class of events that the one considered in~\cite{dalimonte_analyticity_2026}.}. 
This is the content of the next result. Note that the counterexample of Appendix \ref{sec:counter_example} shows that the UD-event assumption cannot be removed. Recall that $\D_{\delta}$ denotes the open disc in $\bbC$ of radius $\delta$ centred at $0$.

\begin{theorem}\label{thm:clust_exp_UDE}
There exists $\delta > 0$ such that the following is true. 
\begin{enumerate}
\item {\bf Analytic logarithm of UD-event probabilities.} For any subgraph $G$ of $\Z^d$ and any \emph{UD-event} $A$ for $G$, the following limit exists for all $z \in \D_{\delta}$,
\begin{equation}
\log \phi^0_{G, p+z}[A] \coloneqq \lim\limits_{N \to \infty} \log \phi^0_{G\cap\Lambda_N,p+z}[A]
\end{equation}
and the function
\begin{equation} z \in \D_{\delta} \mapsto \log \phi^0_{G,p+z}[A] \end{equation}
is an analytic extension of the real function $u \in (-\delta, \delta) \mapsto \log \phi^0_{G,p+u}[A]$.
\item {\bf Double-sided exponential growth bound on UD-event probabilities.} For any $\eps < \delta$, there exists a constant $c_\eps > 0$ going to 0 as $\eps$ tends to 0 such that for any subgraph $G$, for any \emph{UD-event} $A$ for $G$,  the complex extensions defined above satisfy the following growth bound: for any $|z| < \eps$, 
\begin{equation}\label{equ:double_sided_bound}
\big\vert \log \phi^0_{G,p+z}[A] - \log \phi^0_{G,p}[A] \big\vert \leq c_\eps |\mathsf{Supp}_A|. 
\end{equation}
\end{enumerate}
\end{theorem}

\begin{proof}[Proof of Theorem~\ref{thm:clust_exp_UDE}]
We start with the first item. 
Fix a large value of $c > 0$ and take the constants $L \in \Z_{>0}$ and $\delta > 0$ provided by Proposition \ref{prop:clust_exp_comp}. Fix $G$ a subgraph of $\Z^d$, and $A$ a UD-event for $G$. 

Fix $N \in \Z_{>0}$ so large that $\mathsf{Supp}_A \subset G \cap \Lambda_N$. To shorten notation we set $G_N \coloneqq G \cap \Lambda_N$. 
Note that Proposition~\ref{prop:clust_exp_comp} implies in particular that $\mathbf{Z}^0_{G_N,p+z}$ is non zero so $\phi^0_{G_N,p+z}$ is well defined and we may apply Lemma~\ref{lem:dec_UDE} to write
\begin{equation}\label{equ:expansion_ratio}
\frac{\phi^0_{G_N,p+z}[A]}{\phi^0_{G_N,p}[A] } 
=
\alpha_z^{o(\om_A)}\cdot \frac{\mathbf{Z}^0_{G_N,p}}{\mathbf{Z}^0_{G_N,p+z}} \cdot \frac{\mathbf{Z}^0_{G_N\setminus\mathsf{Supp}_A,p+z}}{\mathbf{Z}^0_{G_N\setminus\mathsf{Supp}_A,p}}.
\end{equation} 
The term $\alpha_z^{o(\om_A)}$ is harmless for defining the logarithm.
Indeed, it does not vanish in any neighbourhood on which $\alpha_z \neq 0$, which can be chosen independently of both $A$, $G$ and $N$. 
Furthermore, the complex determination of the logarithm of this quantity that is real-valued on $\R_{>0}$ is bounded by a quantity that depends on $A$ but not on $N$.

We will now apply our cluster expansion result Proposition~\ref{prop:clust_exp_comp} to the other two terms. 
Indeed, by the choice of $\delta$ and by Proposition~\ref{prop:clust_exp_comp}, we have
\begin{equation}\label{equ:introduction_F}
\frac{\mathbf{Z}^0_{G_N,p}}{\mathbf{Z}^0_{G_N,p+z}} \cdot \frac{\mathbf{Z}^0_{G_N\setminus\mathsf{Supp}_A,p+z}}{\mathbf{Z}^0_{G_N\setminus\mathsf{Supp}_A,p}} = \exp(F_{G_N, A}(z) ),  \end{equation}
where the function $F_{H, A}$ is defined for any $H$ finite subgraph of $\Z^d$ and any UD-event $A$ such that $\mathsf{Supp}_A \subset H$ as:
\begin{multline*}
F_{H, A}(z) \coloneqq 
\sum_{n\geq 1}
\Big\{
\sum_{\gamma_1, \dots, \gamma_n \in \sfP_{H\setminus\mathsf{Supp}_A}^{L} }
 \varphi_n(\gamma_1, \dots, \gamma_n)
 \prod_{k=1}^n 
 \sfw_{H\setminus \mathsf{Supp}_A,z}(\gamma_k) 
  \\
  -
 \sum_{\gamma_1, \dots, \gamma_n \in \sfP_{H}^L } 
 \varphi_n(\gamma_1, \dots, \gamma_n)\prod_{k=1}^n \sfw_{H, z}(\gamma_k) \Big\}.
\end{multline*}

To conclude on the first item of the statement, it remains to provide an upper bound on $F_{G_N, A}(z)$ which is independent of $N$. 
Observe that, since  $H\setminus \mathsf{Supp}_A$ was defined to have the same vertex set as $H$, the two polymer sets $\sfP_{H\setminus\mathsf{Supp}_A}^{L}$ and $\sfP_{H}^{L}$ are actually equal. Moreover, by the subgraph compatibility property of the weights $\sfw_{G,z}$ given by Proposition~\ref{prop:clust_exp_comp}, it is the case that all the terms $\gamma_1, \dots, \gamma_n$ satisfying the condition 
\begin{align*}
\mathsf{Span} \big( \bigcup_{i=1}^n \gamma_i \big) \cap \mathsf{Supp}_A = \emptyset
\end{align*}
are counted with the same weight in the two sums. Hence, after cancellations, we obtain: 
 
 \begin{multline} \label{equ:simplified_F}
 F_{G_N, A}(z) = \sum_{n \geq 1}
  \sum_{\substack{\gamma_1, \dots, \gamma_n \in \sfP_{G_N}^L  \\ \mathsf{Span}(\cup \gamma_i)\cap \mathsf{Supp}_A \neq \emptyset} }
  \varphi_n(\gamma_1, \dots, \gamma_n) \\
  \Big\{
  \prod_{k=1}^n \sfw_{G_N \setminus \mathsf{Supp}_A, z}(\gamma_k)
  -
  \prod_{k=1}^n \sfw_{G_N, z}(\gamma_k)
  \Big\}.
 \end{multline}
 Thus, 
 \begin{multline}
 |F_{G_N, A}(z)| 
 \leq 
 \sum_{n \geq 1}
  \sum_{\substack{\gamma_1, \dots, \gamma_n \in \sfP_{G_N}^L  \\ \mathsf{Span}(\cup \gamma_i)\cap \mathsf{Supp}_A \neq \emptyset} }
  \vert \varphi_n(\gamma_1, \dots, \gamma_n) \vert \\
  \Big\{
  \prod_{k=1}^n |\sfw_{G_N \setminus \mathsf{Supp}_A, z}(\gamma_k)|
  +
  \prod_{k=1}^n |\sfw_{G_N, z}(\gamma_k)|
  \Big\}.
 \end{multline}
If the constant $c$ chosen at the beginning of the proof is large enough, then the uniform bound on the weight functions given by Proposition \ref{prop:clust_exp_comp} allows to apply~\cite[Corollary A.2]{dalimonte_analyticity_2026}, which yields the existence of a constant $c_\eps$ such that for any $|z|<\eps< \delta$, 
 \begin{equation}\label{equ:bound_F}
  |F_{G_N, A}(z)|  \leq c_\eps |\mathsf{Supp}_A|. 
 \end{equation}
 
 We are now able to prove the first items of the theorem. 
 Indeed, observe that we now know thanks to~\eqref{equ:expansion_ratio} together with~\eqref{equ:introduction_F} that the ratio $\frac{\phi_{G_N,p+z}[A]}{\phi_{G_N,p}[A]}$ does not vanish in a disk centered at 0. 
 Consider the representation of the logarithm of this ratio that is real-valued on $\R_{>0}$, and call it $\Psi_{N,A}$.
 
 We have proved that:
 \begin{itemize}
 \item For any $N\geq 1$, the function $\Psi_{N,A}$ is analytic in $\D_\delta$. Indeed, $\Psi_{N,A}$ is the logarithm of a non-vanishing rational function.
 \item For any $N\geq 1$, any $z\in \D_\eps$, $|\Psi_{N,A}(z)| \leq c_\eps |\mathsf{Supp}_A| + o(\om_A) \sup_{z\in\D_\eps} |\log \alpha_z|$. 
 This is obtained combining~\eqref{equ:expansion_ratio},~\eqref{equ:introduction_F} and~\eqref{equ:bound_F}.
 \item For any \textbf{real} perturbation $u \in (-\delta, \delta)$, $\Psi_{N,A}(u) \underset{N\rightarrow\infty}{\longrightarrow} \log \phi^0_{G, p+u}[A] - \log \phi^0_{G, p}[A]$, which is finite due to the fact that $A$ is a UD-Event (and in particular has a non-zero probability under the probability measures $\phi^0_{G,p+u}$ and $\phi^0_{G,p}$).
 \end{itemize} 
We may apply Vitali's convergence theorem (see \cite[Theorem 4, p. 164]{remmert2013classical}) to conclude that $\Psi_{N,A}$ converges uniformly towards a holomorphic function $\Psi_A$ in a neighbourhood of~0. This concludes the proof of the first item.
 
The second item of the theorem follows directly by observing that $\sup_{z\in\D_\eps} |\log \alpha_z|$ tends to 0 with $\eps$ and that $o(\om_A) \leq |\mathsf{Supp}_A|$. 

\end{proof}

A crucial observation is that the representation of the increments of complex logarithm of UD-events probabilities given by \eqref{equ:expansion_ratio} and~\eqref{equ:introduction_F}  allows to prove the following \emph{complex mixing property} for the complex extensions. 

\begin{proposition}[``Complex mixing'' property for UD-events]
\label{prop:complex_mixing_UDE}
There exist $\delta>0$, $C > 0$ and $\varsigma > 0$ such that for any $z \in \bbD_\delta$, any pair of subgraphs $G, G'$ of $\Z^d$ and any event $A$, if $A$ is a \emph{UD-event} for both $G$ and $G'$, then
\begin{equation}\label{equ:complex_mixing}
\big|\log \frac{\phi^0_{G,p+z}[A]}{ \phi^0_{G,p}[A]} - \log \frac{\phi^0_{G',p+z}[A]}{ \phi^0_{G',p}[A]} \big| \leq C\sum_{e \in \mathsf{Supp}_A} \exp\big(-\varsigma  \mathrm{d}(e, E(G) \Delta E(G')) \big).  
\end{equation}
where $\mathrm{d}$ stands for the graph distance\footnote{The distance between two sets of edges is defined as the length of the smallest path containing at least one edge from each set of edges.} in $\bbZ^d$ and $E(G) \Delta E(G') \subset E(\Z^d)$ is the symmetric difference between the edge set of $G$ and the edge set of $G'$. 
\end{proposition}

\begin{proof}[Proof of Proposition~\ref{prop:complex_mixing_UDE}]
Fix a large constant $c>0$ and take the constants $L>0$ and $\delta > 0$ provided by Proposition \ref{prop:clust_exp_comp}.
Fix $G,G'$ and $A$ as in the statement. 
Let $N\geq 1$ be large enough so that $\mathsf{Supp}_A \subset G\cap\Lambda_N$, and $\mathsf{Supp}_A \subset G' \cap \Lambda_N$. Again, we will shorten notation by setting $G_N \coloneqq G \cap \Lambda_N$ and $G'_N \coloneqq G' \cap \Lambda_N$.
Applying \eqref{equ:expansion_ratio} and \eqref{equ:introduction_F} to both $G_N$ and $G'_N$ gives, for $z \in \D_{\delta}$,
\begin{equation}\label{equ:decomposition_mixing}
\log \frac{\phi^0_{G_N,p+z}[A]}{ \phi^0_{G_N,p}[A]} - \log \frac{\phi^0_{G'_N,p+z}[A]}{ \phi^0_{G'_N,p}[A]} 
=  F_{G_N, A}(z) - F_{G'_N, A}(z),
 \end{equation}
where $F_{H,A}$ is the function defined below \eqref{equ:introduction_F}. As in \eqref{equ:simplified_F}, the subgraph compatibility property allows to rewrite the right hand term has
\begin{multline*}
 \sum_{n \geq 1} 
 \Big\{
  \sum_{\substack{\gamma_1, \dots, \gamma_n \in \sfP_{G_N}^L  \\ \mathsf{Span}(\cup\gamma_i)\cap \mathsf{Supp}_A \neq \emptyset} }
  \varphi_n(\gamma_1, \dots, \gamma_n)
  \Big\{
  \prod_{k=1}^n \sfw_{G_N \setminus \mathsf{Supp}_A, z}(\gamma_k)
  -
  \prod_{k=1}^n \sfw_{G_N, z}(\gamma_k)
  \Big\}
  \\ 
  -
    \sum_{\substack{\gamma_1, \dots, \gamma_n \in \sfP_{G'_N}^L  \\ \mathsf{Span}(\cup\gamma_i)\cap \mathsf{Supp}_A \neq \emptyset} }
  \varphi_n(\gamma_1, \dots, \gamma_n)  
  \Big\{
  \prod_{k=1}^n \sfw_{G'_N \setminus \mathsf{Supp}_A, z}(\gamma_k)
  -
  \prod_{k=1}^n \sfw_{G'_N, z}(\gamma_k)
  \Big\} 
  \Big\}.
 \end{multline*} 

Observe that further cancellations can be obtained in the difference between the two sums over polymer families above, by using the subgraph compatibility property of the cluster weights~$\mathsf{w}_{G,z}$. 
Indeed, the only non-cancelling terms are the families $\gamma_1, \dots, \gamma_n$ that are such that
$\mathsf{Span}(\cup \gamma_i) \cap (E(G)\Delta E(G')) \neq \emptyset$.  
Define the set 
\begin{equation}
E_n \coloneqq \Big\{ (\gamma_1, \dots, \gamma_n)\in (\sfP^L_{G_N})^n, 
\begin{cases}
\mathsf{Span}(\cup \gamma_i)\cap\mathsf{Supp}_A &\neq \emptyset \\
\mathsf{Span}(\cup \gamma_i)\cap (E(G) \Delta E(G'))  &\neq \emptyset.
\end{cases}
\Big\}.
\end{equation}
Also define $E'_N$ in the same way but replacing $G_N$ by $G'_N$. 
For $\eps < \delta$, define $\tilde\sfw(\gamma) \coloneqq C_\eps\e^{-c|\gamma| }$ to be the uniform bound obtained in Proposition~\ref{prop:clust_exp_comp}.
By the previous observation and by choice of $\tilde\sfw(\gamma)$ we obtain, for $z \in \D_{\eps}$, 
\begin{equation} |F_{G_N, A}(z) - F_{G'_N, A}(z)| \leq 4\sum_{n\geq 1}\sum_{\gamma_1, \dots, \gamma_n \in E_N\cup E'_n}|\varphi_n(\gamma_1, \dots, \gamma_n)| \prod_{k=1}^n \tilde\sfw(\gamma_k). \end{equation}

Recall that when $\cup \gamma_i$ is not connected, the definition of the Ursell functions implies that $\varphi(\gamma_1,\dots,\gamma_n) = 0$. Fix an edge $e$ of $\mathsf{Supp}_A$ and consider $\gamma_1, \dots, \gamma_n \in E_n\cup E'_n$ such that $\mathsf{Span}(\cup \gamma_i)$ contains $e$ and such that $\cup \gamma_i$ is connected.
By definition of $E_n \cup E'_n$, the graph $\mathsf{Span}(\cup \gamma_i)$ also contains a vertex which is the endpoint of an edge in the symmetric difference $E(G)\Delta E(G')$.

As a consequence, it follows from the fact that $\cup \gamma_i$ is connected that
\begin{align*}
|\gamma_1| + \dots + |\gamma_n| \geq |\gamma_1\cup\cdots\cup\gamma_n| \geq \mathrm{d}(e, E(G)\Delta E(G')) / (2dL) - 2.
\end{align*}
So, summing over $e\in \mathsf{Supp}_A$, we get two constants $C >0$ and $\varsigma>0$ such that, for $z \in \D_{\eps}$,
\begin{equation}
|F_{G_N, A}(z) - F_{G'_N, A}(z)| \leq C \sum_{e \in \mathsf{Supp}_A} \exp(-\varsigma \mathrm{d}(e, E(G)\Delta E(G'))) \sfQ_N(e),
\end{equation}
where
\begin{equation}
\sfQ_N(e) = \sum\limits_{n \geq 1}\sum_{\substack {\gamma_1, \dots, \gamma_n \in \sfP^L_{G_N \cup G'_N } \\ e \in \mathsf{Span}(\cup \gamma_i)} } |\varphi_n(\gamma_1, \dots, \gamma_n)| \prod_{k=1}^n \exp(|\gamma_k|)\tilde\sfw(\gamma_k).  
\end{equation}
To conclude, it remains to argue that $\sfQ_N(e)$ is uniformly bounded in $e$ and $N$. 

Once again we use the useful anchored cluster expansion estimate~\cite[Corollary A.2]{dalimonte_analyticity_2026}, which can be applied to $Q_N$ if $c$ is sufficiently large, since $\exp(|\gamma|)\tilde\sfw(\gamma) \leq C_\eps e^{-|\gamma|(c-1)}$. This yields the boundedness of $\sfQ_N$, and the conclusion that, for a possibly different constant $C > 0$, the following holds for all $z \in \D_{\eps}$: 
\begin{equation}
|F_{G_N, A}(z) - F_{G'_N, A}(z)| \leq C\sum_{e \in \mathsf{Supp}_A} \exp(-\varsigma \mathrm{d}(e, E(G)\Delta E(G'))). 
\end{equation} 

We insert this estimate in~\eqref{equ:decomposition_mixing}. 
The conclusion follows by letting $N$ tend to infinity. 
\end{proof}

\subsection{Conclusion: regularity of the FK-Ruelle operator}

We are now able to provide a proof of Theorem~\ref{thm:complex_reg_ruelle}.
Recall that the decomposition in irreducible pieces defined in Section \ref{sec:perturbed_ruelle} is performed in a cone $ \calY_\kappa^{t,+}(0)$ associated to a direction $t \in \partial \calW_p$ with a surcharge~ $\kappa \in (0,1)$. Recall also the notation $\calU_{\delta}$ for the bidisk $\bbD_\delta \times \bbD_\delta(t) \subset \C^{d+1}$. 

\begin{proof}[Proof of Theorem~\ref{thm:complex_reg_ruelle}]
We separately prove the three items of the statement.

\medskip
\noindent{\bf Uniform analyticity of the FK-Ruelle potential.}
We recall the rewriting~\eqref{equ:markov_for_potential}:
\begin{equation}\label{equ:markov_for_potential_bis}
 \psi^s_p(\bar\frakc) = \log \phi^0_{\bar\frakc_1,p}[\mathsf{Ins}(\bar\frakc_1)] + \log \phi^0_{G_{\bar\frakc},p}[\mathsf{Out}(\bar\frakc_1)] + \langle s, X(\bar\frakc_1) \rangle,
\end{equation}
where $G_{\bar\frakc}$ was defined below~\eqref{equ:markov_for_potential}.
Extending the last summand to complex $s$ is obvious: it is actually entire in $s$. 
Concerning the complex extension in $p$, it was observed in Remark~\ref{rq:rewriting_psi_log_UDE} that both $\mathsf{Ins}(\bar\frakc_1)$ and $\mathsf{Out}(\bar\frakc_1)$ are UD-events for $\bar\frakc_1$ and $G_{\bar\frakc}$ respectively.
Thus, we may apply the first item of Theorem~\ref{thm:clust_exp_UDE} to get a constant $\delta > 0$ independent of $\bar\frakc$ such that the two logarithmic terms have analytic extensions defined on $\D_{\delta}$. 
Summing the analytic extensions obtained above yields an analytic extension $(z,s) \mapsto\psi^s_{p+z}(\bar\frakc)$ of the real FK-Ruelle potential defined on $\calU_{\delta}$. This finishes the proof of the first item of the theorem.

\medskip
\noindent{\bf Proof of \eqref{equ:complex_theta_bound_potential}.}
Let $\theta \in (0,1)$ be fixed by Theorem~\ref{thm:complex_reg_ruelle}. Consider $(z,s) \in \calU_{\delta}$ and $\bar\frakc, \bar\frakc'$ such that $\bfi(\bar\frakc, \bar\frakc') \geq 1$. We set $G \coloneqq G_{\bar\frakc}$, $G' \coloneqq G_{\bar\frakc'}$ and $A = \mathsf{Out}(\bar\frakc_1)$. Using a triangular inequality, equation~\eqref{equ:markov_for_potential_bis} and that $\bar\frakc_1 = \bar\frakc'_1$, we get
\begin{align*}
|\psi^s_{p+z}(\bar\frakc) - \psi^s_{p+z}(\bar\frakc') | 
&\leq |\psi^s_{p}(\bar\frakc) - \psi^s_{p}(\bar\frakc') | +  |\psi^s_{p+z}(\bar\frakc) - \psi^s_{p}(\bar\frakc) - (\psi^s_{p+z}(\bar\frakc') - \psi^s_{p}(\bar\frakc'))|\\
&= |\psi^s_{p}(\bar\frakc) - \psi^s_{p}(\bar\frakc') | + \big|\log \frac{\phi^0_{G,p+z}[A]}{ \phi^0_{G,p}[A]} - \log \frac{\phi^0_{G',p+z}[A]}{ \phi^0_{G',p}[A]} \big|. 
\end{align*}
Since for any real subcritical parameter $p$, the potential $\psi^s_{p}$ satisfies $|\psi^s_{p}|_\theta < \infty$ (see~\eqref{equ:real_theta_bound_potential}), there exists a constant $C' > 0$ such that:
\begin{equation}
|\psi^s_{p}(\bar\frakc) - \psi^s_{p}(\bar\frakc') | \leq  C' \theta^{\bfi(\bar\frakc, \bar\frakc')}.
\end{equation}
By Proposition~\ref{prop:complex_mixing_UDE},
\begin{equation}\label{equ:bound_lipschitz_two_terms}
\big|\log \frac{\phi^0_{G,p+z}[A]}{ \phi^0_{G,p}[A]} - \log \frac{\phi^0_{G',p+z}[A]}{ \phi^0_{G',p}[A]} \big|
\leq  
C\sum_{e \in \mathsf{Supp}_A} \exp\big(-\varsigma \mathrm{d}(e, E(G)\Delta E(G')) \big).
\end{equation}
By definition of $\Omega^t_{\kappa}$, every non-trivial irreducible cluster in the decomposition has a non-null displacement in the cone $ \calY_\kappa^{t,+}(0)$. So, there exists a small constant $\varsigma' > 0$ such that 
\begin{align}
\mathrm{d}(\mathsf{Supp}_A,E(G)\Delta E(G')) \geq \varsigma' (\bfi(\bar\frakc, \bar\frakc') - 1).
\end{align}
So, up to changing $\theta$ and $C$, and observing that $\mathsf{Supp}_A$ and $E(G) \Delta E(G')$ are contained in a backward cone and a forward cone respectively, we get,
\begin{equation}
\sum_{e \in \mathsf{Supp}_A}\exp(-\varsigma\mathrm{d}(e, E(G) \Delta E(G')) ) \leq C \theta^{\bfi(\bar\frakc, \bar\frakc')}.
\end{equation}
Combining the previous bounds yields a constant $C > 0$ independent of $z$, $s$, $\bar\frakc$ and $\bar\frakc'$ such that, for $\theta$ sufficiently close to $1$,
\begin{equation}
|\psi^s_{p+z}(\bar\frakc) - \psi^s_{p+z}(\bar\frakc') | \leq C \theta^{\bfi(\bar\frakc, \bar\frakc')},
\end{equation}
which concludes the proof of \eqref{equ:complex_theta_bound_potential}.

\medskip
\noindent{\bf Proof of~\eqref{equ:complex_summability_potential}.}
Fix $0<\eps<\delta $, where $\delta$ is fixed in Theorem~\ref{thm:clust_exp_UDE}.
Take $(z,s) \in \calU_{\eps}$ and $\bar\frakc \in \Omega^t_\kappa$. 
Observe that applying~\eqref{equ:double_sided_bound} from Theorem~\ref{thm:clust_exp_UDE} to the two logarithm terms of \eqref{equ:markov_for_potential_bis} yields the following bound  
\begin{equation}
| \psi^s_{p+z}(\bar\frakc) - \psi_p^{t}(\bar\frakc) | \leq 2c_\eps |\bar\frakc_1| + \eps \vert X(\bar\frakc_1) \vert \leq (2c_\eps + \eps)|\bar\frakc_1|,
\end{equation}
where $|\bar\frakc_1|$ is the number of edges in $\bar\frakc_1$. 
As $2c_\eps+\eps$ goes to 0 with $\eps$, we abuse notation and replace this quantity by $c_\eps$.
We obtain:
\begin{equation}\label{equ:comparison_complex_real}
\sum_{\bar\frakc'\neq \emptyset:~ \sigma(\bar\frakc') = \bar\frakc} 
\vert \e^{\psi^{s}_{p+z}(\bar\frakc')}  \vert
\leq 
\sum_{\bar\frakc'\neq \emptyset:~ \sigma(\bar\frakc') = \bar\frakc}
\e^{\psi_p^{t}(\bar\frakc') + c_\eps|\bar\frakc'_1|}.
\end{equation}

\emph{A priori}, we cannot directly use the mass gap condition given by \cite{campanino_fluctuation_2008} and recalled in~\eqref{equ:velenik_mass_gap} to bound the right-hand term, because the small perturbation that is introduced is in the \emph{volume} of the piece to-be-added $\bar\frakc'_1$ and not in the displacement in the direction given by $t$. 
However, in the case of subcritical models, we can use the exponential decay of the volume of the cluster to conclude.
We shall use a very classical argument that was for instance already noticed in the proof of \cite[Prop. 4.2]{campanino_gaussian_1991}. Let $K > 0$ to be fixed later. 
We split the right-hand side of~\eqref{equ:comparison_complex_real} according to whether $|\bar\frakc'_1| > K\langle t,X(\bar\frakc'_1) \rangle$ or not.

Let us first treat the part of the sum where $|\bar\frakc'_1| > \langle t,X(\bar\frakc'_1) \rangle$. Fix $k > 0$. By comparison of boundary conditions and translation invariance of the full-plane FK-percolation measure, we get the following bound involving the terms~$\bar\frakc'$ such that~$\vert \bar\frakc_1' \vert  = k$,
\begin{equation} 
\sum\limits_{\substack{\bar\frakc' \neq \emptyset:~ \sigma(\bar\frakc') = \bar\frakc \\ |\bar\frakc'_1| = k}} e^{\psi^{t}_p(\bar\frakc') - \langle t,X(\bar\frakc'_1) \rangle} \leq \phi^0_p[\vert\mathcal{C}_0\vert \geq k ],
\end{equation}
where $\calC_0$ denotes the cluster of 0. Summing the above over all the possible values of $\ell = \langle t, X(\bar\frakc'_1) \rangle$ and $k = \vert \bar\frakc'_1 \vert $ such that $k > K \ell$, we obtain:
\begin{equation}\label{equ:volume_vs_length}
\sum_{\substack{\bar\frakc' \neq \emptyset:~ \sigma(\bar\frakc') = \bar\frakc \\ |\bar\frakc'_1| > K\langle t, X(\bar\frakc'_1) \rangle }} \e^{\psi_p^{t}(\bar\frakc') + c_\eps|\bar\frakc'_1|} 
\leq
\sum_{ \ell \geq 1} \sum_{k > K\ell}  \phi^0_p[ |\mathcal{C}_0| \geq k] e^{c_\eps k + \ell}.
\end{equation}
Exponential tails for the volume of the cluster of 0 for any $p<p_c$ have been obtained in~\cite{duminil-copin_sharp_2019}. So, there exists $C,c > 0$ such that $\phi^0_p[ |\mathcal{C}_0| \geq k] \leq C e^{-c k}$. Thus, by choosing $\eps > 0$ sufficiently small such that $c_\eps < c$, the right-hand side of~\eqref{equ:volume_vs_length} is upper bounded by
\begin{equation}\label{equ:final_big_volume}
\sum_{ \ell \geq 1} \sum_{k > K\ell} C \e^{-ck + c_\eps k + \ell} = C (1-\e^{-(c-c_\eps)})^{-1} \sum_{ \ell \geq 1} \e^{\ell(1 -K(c-c_\eps)) } .
\end{equation}
This sum is clearly finite when $K$ is chosen large enough. We now fix such a value of $K$.

On the other hand, the part of the sum where $|\bar\frakc'_1| \leq K\langle t, X(\bar\frakc'_1) \rangle$ satisfies
\begin{align}\label{equ:final_small_volume}
\sum_{\substack{\bar\frakc' \neq \emptyset:~ \sigma(\bar\frakc') = \bar\frakc \\ |\bar\frakc'_1| \leq K \langle t, X(\bar\frakc'_1) \rangle }} 
e^{\psi_p^{t}(\bar\frakc') + c_\eps \vert \frakc'_1 \vert }
&\leq 
\sum_{\substack{\bar\frakc' \neq \emptyset:~ \sigma(\bar\frakc') = \bar\frakc \\ |\bar\frakc'_1| \leq K \langle t, X(\bar\frakc'_1) \rangle }} 
e^{\psi_p^{t}(\bar\frakc') + c_\eps K \langle t, X(\bar\frakc'_1) \rangle}
\\
&\leq 
\sum_{\bar\frakc' \neq \emptyset : \sigma(\bar\frakc') = \bar\frakc} 
e^{\psi_p^{(1 + c_\eps K)t}(\bar\frakc')}.  
\end{align}

Choosing $\eps$ so small that $\vert c_\eps K t \vert < \delta$, we may apply the mass-gap condition of~\eqref{equ:velenik_mass_gap} to conclude on the convergence of the right-hand side of~\eqref{equ:final_small_volume}. 

Gathering~\eqref{equ:comparison_complex_real} with~\eqref{equ:final_big_volume} and~\eqref{equ:final_small_volume} shows that for $\eps>0$ small enough
\begin{align*}
\sup_{(z,s) \in \calU_{\eps}}  \sup_{\bar{\frakc} \in \Omega_\kappa^t} \sum_{\bar{\frakc}' \neq \emptyset:~\sigma(\bar{\frakc}') = \bar{\frakc}} \vert \e^{\psi^{s}_{p+z}(\bar{\frakc}')} \vert < +\infty.
\end{align*}
which is exactly ~\eqref{equ:complex_summability_potential}.
\end{proof}

\sectionstarwithouttoc{Acknowledgements}

We thank Sébastien Ott for useful discussions and comments on the draft of this manuscript.
We also thank Gordon Slade and Yvan Velenik for comments on an early draft of this work.  
The research of LD was supported by the French National Research Agency (ANR), project number \texttt{ANR-23-CPJ1-0150-01}. 
The research of LG is supported by the Swiss National Science Foundation.
\vspace*{1cm}

\sectionstarwithouttoc{AI disclosure statement}

AI was not used by the authors at any stage of this work. 
\vspace*{1cm}

\newpage

\bibliographystyle{amsalpha}
\bibliography{biblio_corr_length.bib}

@book{remmert2013classical,
  title={Classical Topics in Complex Function Theory},
  author={Remmert, R. and Kay, L.D.},
  isbn={9781475729566},
  lccn={97010091},
  series={Graduate Texts in Mathematics},
  url={https://books.google.fr/books?id=7TfnBwAAQBAJ},
  year={2013},
  publisher={Springer New York}
}

@book {grimmett,
    AUTHOR = {Grimmett, Geoffrey},
     TITLE = {Percolation},
    SERIES = {Grundlehren der mathematischen Wissenschaften [Fundamental
              Principles of Mathematical Sciences]},
    VOLUME = {321},
   EDITION = {Second},
 PUBLISHER = {Springer-Verlag, Berlin},
      YEAR = {1999},
     PAGES = {xiv+444},
      ISBN = {3-540-64902-6},
   MRCLASS = {60K35 (60-02 82B43)},
  MRNUMBER = {1707339},
MRREVIEWER = {Neal\ Madras},
       DOI = {10.1007/978-3-662-03981-6},
       URL = {https://doi.org/10.1007/978-3-662-03981-6},
}

@misc{liu_crossover_2026,
      title={Crossover from subcritical to critical decay: random walk, self-avoiding walk, percolation}, 
      author={Yucheng Liu and Gordon Slade},
      year={2026},
      howpublished={arXiv preprint 2605.15545},
      eprint={2605.15545},
      archivePrefix={arXiv},
      primaryClass={math.PR},
      url={https://arxiv.org/abs/2605.15545}, 
}

@article {ott_decay_2022,
    AUTHOR = {Ott, S\'{e}bastien},
     TITLE = {Existence and properties of connections decay rate for high
              temperature percolation models},
   JOURNAL = {Electron. J. Probab.},
  FJOURNAL = {Electronic Journal of Probability},
    VOLUME = {27},
      YEAR = {2022},
     PAGES = {Paper No. 100, 19},
      ISSN = {1083-6489},
   MRCLASS = {60K35 (82B43)},
  MRNUMBER = {4460271},
       DOI = {10.1214/22-ejp822},
       URL = {https://doi.org/10.1214/22-ejp822},
}

@article {alexander_weak_1998,
    AUTHOR = {Alexander, Kenneth S.},
     TITLE = {On weak mixing in lattice models},
   JOURNAL = {Probab. Theory Related Fields},
  FJOURNAL = {Probability Theory and Related Fields},
    VOLUME = {110},
      YEAR = {1998},
    NUMBER = {4},
     PAGES = {441--471},
      ISSN = {0178-8051,1432-2064},
   MRCLASS = {60K35 (82B05)},
  MRNUMBER = {1626951},
MRREVIEWER = {Yvan\ Velenik},
       DOI = {10.1007/s004400050155},
       URL = {https://doi.org/10.1007/s004400050155},
}

@book{book_velenik,
place={Cambridge},
title={Statistical Mechanics of Lattice Systems: A Concrete Mathematical Introduction},
DOI={10.1017/9781316882603},
ISBN={978-1-107-18482-4},
publisher={Cambridge University Press},
author={Friedli, Sacha and Velenik, Yvan},
year={2017}
}

@article {bricmont_frohlich_1985,
    AUTHOR = {Bricmont, J. and Fr\"{o}hlich, J.},
     TITLE = {Statistical mechanical methods in particle structure analysis
              of lattice field theories. {II}. {S}calar and surface models},
   JOURNAL = {Comm. Math. Phys.},
  FJOURNAL = {Communications in Mathematical Physics},
    VOLUME = {98},
      YEAR = {1985},
    NUMBER = {4},
     PAGES = {553--578},
      ISSN = {0010-3616,1432-0916},
   MRCLASS = {81E25 (82A05 82A68)},
  MRNUMBER = {789871},
MRREVIEWER = {Th\'{o}rdur\ J\'{o}nsson},
       URL = {http://projecteuclid.org/euclid.cmp/1103942542},
}

@article {lebowitz_penrose,
    AUTHOR = {Penrose, O. and Lebowitz, J. L.},
     TITLE = {On the exponential decay of correlation functions},
   JOURNAL = {Comm. Math. Phys.},
  FJOURNAL = {Communications in Mathematical Physics},
    VOLUME = {39},
      YEAR = {1974},
     PAGES = {165--184},
      ISSN = {0010-3616,1432-0916},
   MRCLASS = {82.31},
  MRNUMBER = {432092},
MRREVIEWER = {M.\ Lawrence\ Glasser},
       URL = {http://projecteuclid.org/euclid.cmp/1103860196},
}

@article {frohlich_rodriguez_2012,
    AUTHOR = {Fr\"{o}hlich, J\"{u}rg and Rodriguez, Pierre-Fran\c{c}ois},
     TITLE = {Some applications of the {L}ee-{Y}ang theorem},
   JOURNAL = {J. Math. Phys.},
  FJOURNAL = {Journal of Mathematical Physics},
    VOLUME = {53},
      YEAR = {2012},
    NUMBER = {9},
     PAGES = {095218, 15},
      ISSN = {0022-2488,1089-7658},
   MRCLASS = {82B20},
  MRNUMBER = {2905800},
       DOI = {10.1063/1.4749391},
       URL = {https://doi.org/10.1063/1.4749391},
}

@article {frohlich_rodriguez_2017,
    AUTHOR = {Fr\"{o}hlich, J\"{u}rg and Rodr\'{\i}guez,
              Pierre-Fran\c{c}ois},
     TITLE = {On cluster properties of classical ferromagnets in an external
              magnetic field},
   JOURNAL = {J. Stat. Phys.},
  FJOURNAL = {Journal of Statistical Physics},
    VOLUME = {166},
      YEAR = {2017},
    NUMBER = {3-4},
     PAGES = {828--840},
      ISSN = {0022-4715,1572-9613},
   MRCLASS = {82D40 (82B20)},
  MRNUMBER = {3607592},
       DOI = {10.1007/s10955-016-1556-2},
       URL = {https://doi.org/10.1007/s10955-016-1556-2},
}

@book{convexity,
    author = {Webster, R.},
    isbn = {9780198531470},
    lccn = {93034951},
    publisher = {Oxford University Press},
    series = {Oxford science publications},
    title = {Convexity},
    url = {https://books.google.fr/books?id=SWQmykCwabwC},
    year = {1994}}

@article {borgs_covariance_1996,
    AUTHOR = {Borgs, C. and Chayes, J. T.},
     TITLE = {The covariance matrix of the {P}otts model: a random cluster
              analysis},
   JOURNAL = {J. Statist. Phys.},
  FJOURNAL = {Journal of Statistical Physics},
    VOLUME = {82},
      YEAR = {1996},
    NUMBER = {5-6},
     PAGES = {1235--1297},
      ISSN = {0022-4715,1572-9613},
   MRCLASS = {82B20 (60K35 82B27 82B43)},
  MRNUMBER = {1374925},
MRREVIEWER = {Aernout\ C. D. van Enter},
       DOI = {10.1007/BF02183383},
       URL = {https://doi.org/10.1007/BF02183383},
}

@book {baumgartel_book,
    AUTHOR = {Baumg\"{a}rtel, H.},
     TITLE = {Analytic perturbation theory for matrices and operators},
    SERIES = {Operator Theory: Advances and Applications},
    VOLUME = {15},
 PUBLISHER = {Birkh\"{a}user Verlag, Basel},
      YEAR = {1985},
     PAGES = {427},
      ISBN = {3-7643-1664-0},
   MRCLASS = {47A55},
  MRNUMBER = {878974},
}

@book {lax,
    AUTHOR = {Lax, Peter D.},
     TITLE = {Functional analysis},
    SERIES = {Pure and Applied Mathematics (New York)},
 PUBLISHER = {Wiley-Interscience [John Wiley \& Sons], New York},
      YEAR = {2002},
     PAGES = {xx+580},
      ISBN = {0-471-55604-1},
   MRCLASS = {47-01 (46-01)},
  MRNUMBER = {1892228},
MRREVIEWER = {Melvin\ Faierman},
}

@article {baumgartel_article,
    AUTHOR = {Baumg\"{a}rtel, H.},
     TITLE = {The theory of analytic perturbations for linear operators
              depending on several complex variables},
   JOURNAL = {Mat. Issled.},
  FJOURNAL = {Akademiya Nauk Moldavsko\u{\i} SSR. Institut Matematiki s
              Vychislitel\cprime nym Tsentrom. Matematicheskie
              Issledovaniya},
    NUMBER = {1(31)},
    VOLUME = {9},
      YEAR = {1974},
     PAGES = {17--39},
      ISSN = {0542-9994},
   MRCLASS = {47A55},
  MRNUMBER = {353019},
MRREVIEWER = {Alexander\ A.\ Pankov},
}

@article {parry_zeta_1990,
    AUTHOR = {Parry, William and Pollicott, Mark},
     TITLE = {Zeta functions and the periodic orbit structure of hyperbolic
              dynamics},
   JOURNAL = {Ast\'{e}risque},
  FJOURNAL = {Ast\'{e}risque},
    NUMBER = {187-188},
      YEAR = {1990},
     PAGES = {268},
      ISSN = {0303-1179,2492-5926},
   MRCLASS = {58F20 (58F11 58F15)},
  MRNUMBER = {1085356},
MRREVIEWER = {Nicola\u{\i}\ T. A. Haydn},
}

@article {de-coninck_young_1989,
    AUTHOR = {De Coninck, Jo\"{e}l and Dunlop, Fran\c{c}ois and Rivasseau,
              Vincent},
     TITLE = {On the microscopic validity of the {W}ulff construction and of
              the generalized {Y}oung equation},
   JOURNAL = {Comm. Math. Phys.},
  FJOURNAL = {Communications in Mathematical Physics},
    VOLUME = {121},
      YEAR = {1989},
    NUMBER = {3},
     PAGES = {401--419},
      ISSN = {0010-3616,1432-0916},
   MRCLASS = {82A05 (60K35 82A68)},
  MRNUMBER = {990773},
MRREVIEWER = {C.\ Eugene\ Wayne},
       URL = {http://projecteuclid.org/euclid.cmp/1104178139},
}

@article {messager_antonov_1991,
    AUTHOR = {Messager, Alain and Miracle-Sol\'{e}, Salvador and Ruiz, Jean
              and Shlosman, Senya},
     TITLE = {Interfaces in the {P}otts model. {II}. {A}ntonov's rule and
              rigidity of the order disorder interface},
   JOURNAL = {Comm. Math. Phys.},
  FJOURNAL = {Communications in Mathematical Physics},
    VOLUME = {140},
      YEAR = {1991},
    NUMBER = {2},
     PAGES = {275--290},
      ISSN = {0010-3616,1432-0916},
   MRCLASS = {82B20},
  MRNUMBER = {1124270},
MRREVIEWER = {Klaus\ Ziegler},
       URL = {http://projecteuclid.org/euclid.cmp/1104247982},
}

@article {wu_duality_1997,
    AUTHOR = {Wu, F. Y.},
     TITLE = {Duality relations for {P}otts correlation functions},
   JOURNAL = {Phys. Lett. A},
  FJOURNAL = {Physics Letters. A},
    VOLUME = {228},
      YEAR = {1997},
    NUMBER = {1-2},
     PAGES = {43--47},
      ISSN = {0375-9601,1873-2429},
   MRCLASS = {82B20},
  MRNUMBER = {1448595},
       DOI = {10.1016/S0375-9601(97)00094-7},
       URL = {https://doi.org/10.1016/S0375-9601(97)00094-7},
}

@article {sarig_thermodynamic_1999,
    AUTHOR = {Sarig, Omri M.},
     TITLE = {Thermodynamic formalism for countable {M}arkov shifts},
   JOURNAL = {Ergodic Theory Dynam. Systems},
  FJOURNAL = {Ergodic Theory and Dynamical Systems},
    VOLUME = {19},
      YEAR = {1999},
    NUMBER = {6},
     PAGES = {1565--1593},
      ISSN = {0143-3857,1469-4417},
   MRCLASS = {37A99 (37B99 37C30 37D35)},
  MRNUMBER = {1738951},
MRREVIEWER = {Irene\ Hueter},
       DOI = {10.1017/S0143385799146820},
       URL = {https://doi.org/10.1017/S0143385799146820},
}

@book {ruelle_thermodynamic_2004,
    AUTHOR = {Ruelle, David},
     TITLE = {Thermodynamic formalism},
    SERIES = {Cambridge Mathematical Library},
   EDITION = {Second},
      NOTE = {The mathematical structures of equilibrium statistical
              mechanics},
 PUBLISHER = {Cambridge University Press, Cambridge},
      YEAR = {2004},
     PAGES = {xx+174},
      ISBN = {0-521-54649-4},
   MRCLASS = {82B05 (37A60 37D35 37N20 82-02)},
  MRNUMBER = {2129258},
       DOI = {10.1017/CBO9780511617546},
       URL = {https://doi.org/10.1017/CBO9780511617546},
}

@book {kato_perturbation_1995,
    AUTHOR = {Kato, Tosio},
     TITLE = {Perturbation theory for linear operators},
    SERIES = {Classics in Mathematics},
   EDITION = {1980},
 PUBLISHER = {Springer-Verlag, Berlin},
      YEAR = {1995},
     PAGES = {xxii+619},
      ISBN = {3-540-58661-X},
   MRCLASS = {47A55 (46-00 47-00)},
  MRNUMBER = {1335452},
}

@book {baladi_positive_2000,
    AUTHOR = {Baladi, Viviane},
     TITLE = {Positive transfer operators and decay of correlations},
    SERIES = {Advanced Series in Nonlinear Dynamics},
    VOLUME = {16},
 PUBLISHER = {World Scientific Publishing Co., Inc., River Edge, NJ},
      YEAR = {2000},
     PAGES = {x+314},
      ISBN = {981-02-3328-0},
   MRCLASS = {37C30 (37-02 37A25 37D20 47B38)},
  MRNUMBER = {1793194},
MRREVIEWER = {J\'{e}r\^{o}me\ Buzzi},
       DOI = {10.1142/9789812813633},
       URL = {https://doi.org/10.1142/9789812813633},
}

@article {campanino_fluctuation_2008,
    AUTHOR = {Campanino, Massimo and Ioffe, Dmitry and Velenik, Yvan},
     TITLE = {Fluctuation theory of connectivities for subcritical random
              cluster models},
   JOURNAL = {Ann. Probab.},
  FJOURNAL = {The Annals of Probability},
    VOLUME = {36},
      YEAR = {2008},
    NUMBER = {4},
     PAGES = {1287--1321},
      ISSN = {0091-1798,2168-894X},
   MRCLASS = {60K35 (37C30 60F15 60K15 82B44)},
  MRNUMBER = {2435850},
MRREVIEWER = {Rongfeng\ Sun},
       DOI = {10.1214/07-AOP359},
       URL = {https://doi.org/10.1214/07-AOP359},
}

@article {ocaroll_analyticity_1984-1,
    AUTHOR = {O'Carroll, Michael},
     TITLE = {Analyticity properties and a convergent expansion for the
              inverse correlation length of the high-temperature
              {$d$}-dimensional {I}sing model},
   JOURNAL = {J. Statist. Phys.},
  FJOURNAL = {Journal of Statistical Physics},
    VOLUME = {34},
      YEAR = {1984},
    NUMBER = {3-4},
     PAGES = {597--608},
      ISSN = {0022-4715,1572-9613},
   MRCLASS = {82A68 (81E25 82A05)},
  MRNUMBER = {748804},
       DOI = {10.1007/BF01018560},
       URL = {https://doi.org/10.1007/BF01018560},
}

@article {ocaroll_gluebal_1984,
    AUTHOR = {O'Carroll, Michael and de Almeida Braga, Gast\~{a}o},
     TITLE = {Analyticity properties and a convergent expansion for the
              glueball mass and dispersion curve of strongly coupled
              {E}uclidean {$2+1$} lattice gauge theories},
   JOURNAL = {J. Math. Phys.},
  FJOURNAL = {Journal of Mathematical Physics},
    VOLUME = {25},
      YEAR = {1984},
    NUMBER = {9},
     PAGES = {2741--2743},
      ISSN = {0022-2488,1089-7658},
   MRCLASS = {81E25},
  MRNUMBER = {756575},
       DOI = {10.1063/1.526481},
       URL = {https://doi.org/10.1063/1.526481},
}

@article {aoun_two-point_2022,
    AUTHOR = {Aoun, Yacine and Ott, S\'{e}bastien and Velenik, Yvan},
     TITLE = {On the two-point function of the {P}otts model in the
              saturation regime},
   JOURNAL = {Comm. Math. Phys.},
  FJOURNAL = {Communications in Mathematical Physics},
    VOLUME = {399},
      YEAR = {2023},
    NUMBER = {2},
     PAGES = {1103--1138},
      ISSN = {0010-3616,1432-0916},
   MRCLASS = {82B20 (60K35 82B27)},
  MRNUMBER = {4576765},
MRREVIEWER = {Olivier\ Couronn\'{e}},
       DOI = {10.1007/s00220-022-04574-9},
       URL = {https://doi.org/10.1007/s00220-022-04574-9},
}

@article {aoun_ornstein-zernike_2024,
    AUTHOR = {Aoun, Yacine and Ott, S\'{e}bastien and Velenik, Yvan},
     TITLE = {Ornstein-{Z}ernike behavior for {I}sing models with
              infinite-range interactions},
   JOURNAL = {Ann. Inst. Henri Poincar\'{e} Probab. Stat.},
  FJOURNAL = {Annales de l'Institut Henri Poincar\'{e} Probabilit\'{e}s et
              Statistiques},
    VOLUME = {60},
      YEAR = {2024},
    NUMBER = {1},
     PAGES = {167--207},
      ISSN = {0246-0203,1778-7017},
   MRCLASS = {60K35 (82B20 82D60)},
  MRNUMBER = {4718378},
MRREVIEWER = {Olivier\ Couronn\'{e}},
       DOI = {10.1214/22-aihp1345},
       URL = {https://doi.org/10.1214/22-aihp1345},
}

@article {aoun_non-analyticity_2021,
    AUTHOR = {Aoun, Yacine and Ioffe, Dmitry and Ott, S\'{e}bastien and
              Velenik, Yvan},
     TITLE = {Non-analyticity of the correlation length in systems with
              exponentially decaying interactions},
   JOURNAL = {Comm. Math. Phys.},
  FJOURNAL = {Communications in Mathematical Physics},
    VOLUME = {386},
      YEAR = {2021},
    NUMBER = {1},
     PAGES = {433--467},
      ISSN = {0010-3616,1432-0916},
   MRCLASS = {82B20},
  MRNUMBER = {4287191},
MRREVIEWER = {Simone\ Del Vecchio},
       DOI = {10.1007/s00220-021-04038-6},
       URL = {https://doi.org/10.1007/s00220-021-04038-6},
}

@incollection {ott_asymptotics_2023,
    AUTHOR = {Ott, Sebastien and Velenik, Yvan},
     TITLE = {Asymptotics of correlations in the {I}sing model: a brief
              survey},
 BOOKTITLE = {Topics in statistical mechanics},
    SERIES = {Panor. Synth\`eses},
    VOLUME = {59},
     PAGES = {157--184},
 PUBLISHER = {Soc. Math. France, Paris},
      YEAR = {2023},
      ISBN = {978-2-85629-970-8},
   MRCLASS = {82B20},
  MRNUMBER = {4619313},
}

@article {paes-leme_ornstein-zernike_1978,
    AUTHOR = {Paes-Leme, Paulo Jorge},
     TITLE = {Ornstein-{Z}ernike and analyticity properties for classical
              lattice spin systems},
   JOURNAL = {Ann. Physics},
  FJOURNAL = {Annals of Physics},
    VOLUME = {115},
      YEAR = {1978},
    NUMBER = {2},
     PAGES = {367--387},
      ISSN = {0003-4916,1096-035X},
   MRCLASS = {81C99 (00A89 82A67)},
  MRNUMBER = {513891},
       DOI = {10.1016/0003-4916(78)90160-4},
       URL = {https://doi.org/10.1016/0003-4916(78)90160-4},
}

@article {ocaroll_analyticity_1984-2,
    AUTHOR = {O'Carroll, Michael and Dantas Barbosa, Wilson},
     TITLE = {Analyticity properties and a convergent expansion for the
              inverse correlation length of the low-temperature
              {$d$}-dimensional {I}sing model},
   JOURNAL = {J. Statist. Phys.},
  FJOURNAL = {Journal of Statistical Physics},
    VOLUME = {34},
      YEAR = {1984},
    NUMBER = {3-4},
     PAGES = {609--614},
      ISSN = {0022-4715,1572-9613},
   MRCLASS = {82A68 (82A05)},
  MRNUMBER = {748805},
MRREVIEWER = {J.\ Fernando\ Perez},
       DOI = {10.1007/BF01018561},
       URL = {https://doi.org/10.1007/BF01018561},
}

@article {campanino_gaussian_1991,
    AUTHOR = {Campanino, Massimo and Chayes, J. T. and Chayes, L.},
     TITLE = {Gaussian fluctuations of connectivities in the subcritical
              regime of percolation},
   JOURNAL = {Probab. Theory Related Fields},
  FJOURNAL = {Probability Theory and Related Fields},
    VOLUME = {88},
      YEAR = {1991},
    NUMBER = {3},
     PAGES = {269--341},
      ISSN = {0178-8051,1432-2064},
   MRCLASS = {60K35 (82B43)},
  MRNUMBER = {1100895},
MRREVIEWER = {Neal\ Madras},
       DOI = {10.1007/BF01418864},
       URL = {https://doi.org/10.1007/BF01418864},
}

@inproceedings{chayes_random_1984,
    author = "Chayes, J. T. and Chayes, L.",
    title = "{Percolation and random media: lecture notes for les {H}ouches, summer 1984}",
    booktitle = "{Summer School in Theoretical Physics, Session XLIII: Critical Phenomena, Random Systems, Gauge Theories}",
    reportNumber = "HUTMP B184",
    month = "8",
    year = "1984"
}

@article {abraham_random_1984,
    AUTHOR = {Abraham, D. B. and Chayes, J. T. and Chayes, L.},
     TITLE = {Random surface correlation functions},
   JOURNAL = {Comm. Math. Phys.},
  FJOURNAL = {Communications in Mathematical Physics},
    VOLUME = {96},
      YEAR = {1984},
    NUMBER = {4},
     PAGES = {439--471},
      ISSN = {0010-3616,1432-0916},
   MRCLASS = {82A05 (60K35 82A43)},
  MRNUMBER = {775041},
MRREVIEWER = {Th\'{o}rdur\ J\'{o}nsson},
       URL = {http://projecteuclid.org/euclid.cmp/1103941909},
}

@book {madras_self-avoiding_1993,
    AUTHOR = {Madras, Neal and Slade, Gordon},
     TITLE = {The self-avoiding walk},
    SERIES = {Probability and its Applications},
 PUBLISHER = {Birkh\"{a}user Boston, Inc., Boston, MA},
      YEAR = {1993},
     PAGES = {xiv+425},
      ISBN = {0-8176-3589-0},
   MRCLASS = {82-02 (60K35 82B41 82C22 82D60)},
  MRNUMBER = {1197356},
MRREVIEWER = {Cheng\ Xun\ Wu},
}

@article {chayes_ornstein-zernike_1986,
    AUTHOR = {Chayes, J. T. and Chayes, L.},
     TITLE = {Ornstein-{Z}ernike behavior for self-avoiding walks at all
              noncritical temperatures},
   JOURNAL = {Comm. Math. Phys.},
  FJOURNAL = {Communications in Mathematical Physics},
    VOLUME = {105},
      YEAR = {1986},
    NUMBER = {2},
     PAGES = {221--238},
      ISSN = {0010-3616,1432-0916},
   MRCLASS = {82A05 (60J15 82A25 82A68)},
  MRNUMBER = {849206},
MRREVIEWER = {Anton\ Bovier},
       URL = {http://projecteuclid.org/euclid.cmp/1104115332},
}

@article {duminil-copin_sharp_2019,
    AUTHOR = {Duminil-Copin, Hugo and Raoufi, Aran and Tassion, Vincent},
     TITLE = {Sharp phase transition for the random-cluster and {P}otts
              models via decision trees},
   JOURNAL = {Ann. of Math. (2)},
  FJOURNAL = {Annals of Mathematics. Second Series},
    VOLUME = {189},
      YEAR = {2019},
    NUMBER = {1},
     PAGES = {75--99},
      ISSN = {0003-486X,1939-8980},
   MRCLASS = {60K35},
  MRNUMBER = {3898174},
MRREVIEWER = {Stanislav\ Volkov},
       DOI = {10.4007/annals.2019.189.1.2},
       URL = {https://doi.org/10.4007/annals.2019.189.1.2},
}

@article {campanino_ornstein-zernike_2003,
    AUTHOR = {Campanino, Massimo and Ioffe, Dmitry and Velenik, Yvan},
     TITLE = {Ornstein-{Z}ernike theory for finite range {I}sing models
              above {$T_c$}},
   JOURNAL = {Probab. Theory Related Fields},
  FJOURNAL = {Probability Theory and Related Fields},
    VOLUME = {125},
      YEAR = {2003},
    NUMBER = {3},
     PAGES = {305--349},
      ISSN = {0178-8051,1432-2064},
   MRCLASS = {82B20 (60F15 60K35 82B28)},
  MRNUMBER = {1964456},
       DOI = {10.1007/s00440-002-0229-z},
       URL = {https://doi.org/10.1007/s00440-002-0229-z},
}

@book{grimmett_RCM,
    address = {Berlin, Heidelberg},
    author = {Grimmett, Geoffrey},
    booktitle = {Probability on Discrete Structures},
    doi = {10.1007/978-3-662-09444-0_2},
    isbn = {978-3-662-09444-0},
    publisher = {Springer Berlin Heidelberg},
    title = {The Random-Cluster Model},
    url = {https://doi.org/10.1007/978-3-662-09444-0_2},
    year = {2004},
    }

@misc{messikh_surface_2006,
      title={The surface tension near criticality of the {2D}-{I}sing model}, 
      author={R. J. Messikh},
      year={2006},
      howpublished={arXiv preprint, arXiv:061036},
      eprint={math/0610636},
      archivePrefix={arXiv},
      primaryClass={math.PR},
      url={https://arxiv.org/abs/math/0610636}, 
}

@misc{dalimonte_near-critical_2025,
      title={Near-critical {O}rnstein--{Z}ernike theory for the planar random-cluster model}, 
      author={Lucas D'Alimonte and Ioan Manolescu},
      year={2025},
      howpublished={arXiv preprint, arXiv:2510.13648},
      eprint={2510.13648},
      archivePrefix={arXiv},
      primaryClass={math.PR},
      url={https://arxiv.org/abs/2510.13648}, 
}

@misc{manolescu_wulff_2026,
      title={The {W}ulff crystal of self-dual {FK}-percolation becomes round when approaching criticality}, 
      author={Ioan Manolescu and Maran Mohanarangan},
      year={2026},
      howpublished={arXiv preprint, arXiv:2603.16318},
      eprint={2603.16318},
      archivePrefix={arXiv},
      primaryClass={math.PR},
      url={https://arxiv.org/abs/2603.16318}, 
}

@article {duminil-copin_wulff_2013,
    AUTHOR = {Duminil-Copin, Hugo},
     TITLE = {Limit of the {W}ulff crystal when approaching critically for
              site percolation on the triangular lattice},
   JOURNAL = {Electron. Commun. Probab.},
  FJOURNAL = {Electronic Communications in Probability},
    VOLUME = {18},
      YEAR = {2013},
     PAGES = {no. 93, 9},
      ISSN = {1083-589X},
   MRCLASS = {60K35 (82B20 82B26)},
  MRNUMBER = {3151749},
MRREVIEWER = {N.\ N.\ Ganikhodjaev},
       DOI = {10.1214/ECP.v18-3163},
       URL = {https://doi.org/10.1214/ECP.v18-3163},
}

@misc{dalimonte_analyticity_2026,
      title={Uniform analyticity of local observables in {FK}-percolation and analyticity of the {I}sing spontaneous magnetisation}, 
      author={Lucas D'Alimonte and Lo{ï}c Gassmann},
      howpublished={arXiv preprint, arXiv:2604.18558},
      year={2026},
      eprint={2604.18558},
      archivePrefix={arXiv},
      primaryClass={math.PR},
      url={https://arxiv.org/abs/2604.18558}, 
}

@misc{ott_weak_2020-1,
      title={Weak mixing and analyticity in Random Cluster and low temperature {I}sing models}, 
      author={Sébastien Ott},
      howpublished={arXiv preprint, arXiv:2003.05879},
      year={2020},
      eprint={2003.05879},
      archivePrefix={arXiv},
      primaryClass={math-ph},
      url={https://arxiv.org/abs/2003.05879}, 
}

@article {ott_potts_2018,
    AUTHOR = {Ott, S\'{e}bastien and Velenik, Yvan},
     TITLE = {Potts models with a defect line},
   JOURNAL = {Comm. Math. Phys.},
  FJOURNAL = {Communications in Mathematical Physics},
    VOLUME = {362},
      YEAR = {2018},
    NUMBER = {1},
     PAGES = {55--106},
      ISSN = {0010-3616,1432-0916},
   MRCLASS = {82D40 (82B20)},
  MRNUMBER = {3833604},
       DOI = {10.1007/s00220-018-3197-6},
       URL = {https://doi.org/10.1007/s00220-018-3197-6},
}

@InProceedings{HDC_lecture_notes,
author="Duminil-Copin, Hugo",
editor="Barlow, Martin T.
and Slade, Gordon",
title="Lectures on the {I}sing and {P}otts Models on the Hypercubic Lattice",
booktitle="Random Graphs, Phase Transitions, and the Gaussian Free Field",
year="2020",
publisher="Springer International Publishing",
address="Cham",
pages="35--161",
isbn="978-3-030-32011-9"
}

@article{2D_Ising_mass,
  title = {Theory of {T}oeplitz Determinants and the Spin Correlations of the Two-Dimensional {I}sing Model. {I}},
  author = {Wu, Tai Tsun},
  journal = {Phys. Rev.},
  volume = {149},
  issue = {1},
  pages = {380--401},
  numpages = {0},
  year = {1966},
  month = {Sep},
  publisher = {American Physical Society},
  doi = {10.1103/PhysRev.149.380},
  url = {https://link.aps.org/doi/10.1103/PhysRev.149.380}
}

@article {Onsager,
    AUTHOR = {Onsager, Lars},
     TITLE = {Crystal statistics. {I}. {A} two-dimensional model with an
              order-disorder transition},
   JOURNAL = {Phys. Rev. (2)},
  FJOURNAL = {Physical Review. Series II},
    VOLUME = {65},
      YEAR = {1944},
     PAGES = {117--149},
      ISSN = {0031-899X,1536-6065},
   MRCLASS = {82.1X},
  MRNUMBER = {10315},
MRREVIEWER = {J.\ S.\ Frame},
}

@article {duminil-copin_discontinuity_2021,
    AUTHOR = {Duminil-Copin, Hugo and Gagnebin, Maxime and Harel, Matan and
              Manolescu, Ioan and Tassion, Vincent},
     TITLE = {Discontinuity of the phase transition for the planar
              random-cluster and {P}otts models with {$q>4$}},
   JOURNAL = {Ann. Sci. \'{E}c. Norm. Sup\'{e}r. (4)},
  FJOURNAL = {Annales Scientifiques de l'\'{E}cole Normale Sup\'{e}rieure.
              Quatri\`eme S\'{e}rie},
    VOLUME = {54},
      YEAR = {2021},
    NUMBER = {6},
     PAGES = {1363--1413},
      ISSN = {0012-9593,1873-2151},
   MRCLASS = {82B20 (82B26)},
  MRNUMBER = {4371394},
       DOI = {10.24033/asens.2485},
       URL = {https://doi.org/10.24033/asens.2485},
}

@article {DCM_scaling_2022,
    AUTHOR = {Duminil-Copin, Hugo and Manolescu, Ioan},
     TITLE = {Planar random-cluster model: scaling relations},
   JOURNAL = {Forum Math. Pi},
  FJOURNAL = {Forum of Mathematics. Pi},
    VOLUME = {10},
      YEAR = {2022},
     PAGES = {Paper No. e23, 83},
      ISSN = {2050-5086},
   MRCLASS = {60K35},
  MRNUMBER = {4515274},
       DOI = {10.1017/fmp.2022.16},
       URL = {https://doi.org/10.1017/fmp.2022.16},
}

@article {dalimonte_entropic_2024,
    AUTHOR = {D'Alimonte, Lucas},
     TITLE = {Entropic repulsion and scaling limit for a finite number of
              non-intersecting subcritical {FK} interfaces},
   JOURNAL = {Electron. J. Probab.},
  FJOURNAL = {Electronic Journal of Probability},
    VOLUME = {29},
      YEAR = {2024},
     PAGES = {Paper No. 68, 53},
      ISSN = {1083-6489},
   MRCLASS = {60K35 (82B20 82B24 82B41 82B43)},
  MRNUMBER = {4741492},
       DOI = {10.1214/24-ejp1127},
       URL = {https://doi.org/10.1214/24-ejp1127},
}

@article {duminil-copin_exponential_2020,
    AUTHOR = {Duminil-Copin, Hugo and Goswami, Subhajit and Raoufi, Aran},
     TITLE = {Exponential decay of truncated correlations for the {I}sing
              model in any dimension for all but the critical temperature},
   JOURNAL = {Comm. Math. Phys.},
  FJOURNAL = {Communications in Mathematical Physics},
    VOLUME = {374},
      YEAR = {2020},
    NUMBER = {2},
     PAGES = {891--921},
      ISSN = {0010-3616,1432-0916},
   MRCLASS = {82B20 (60K35)},
  MRNUMBER = {4072233},
MRREVIEWER = {Jianping\ Jiang},
       DOI = {10.1007/s00220-019-03633-y},
       URL = {https://doi.org/10.1007/s00220-019-03633-y},
}

@article {harel_spinka,
    AUTHOR = {Harel, Matan and Spinka, Yinon},
     TITLE = {Finitary codings for the random-cluster model and other
              infinite-range monotone models},
   JOURNAL = {Electron. J. Probab.},
  FJOURNAL = {Electronic Journal of Probability},
    VOLUME = {27},
      YEAR = {2022},
     PAGES = {Paper No. 51, 32},
      ISSN = {1083-6489},
   MRCLASS = {60K35 (60G60)},
  MRNUMBER = {4416675},
MRREVIEWER = {Bruno\ Schapira},
       DOI = {10.1214/22-ejp778},
       URL = {https://doi.org/10.1214/22-ejp778},
}

@article {Ott_Ising,
    AUTHOR = {Ott, S\'{e}bastien},
     TITLE = {Weak mixing and analyticity of the pressure in the {I}sing
              model},
   JOURNAL = {Comm. Math. Phys.},
  FJOURNAL = {Communications in Mathematical Physics},
    VOLUME = {377},
      YEAR = {2020},
    NUMBER = {1},
     PAGES = {675--696},
      ISSN = {0010-3616,1432-0916},
   MRCLASS = {82B20 (82D45)},
  MRNUMBER = {4107941},
MRREVIEWER = {Eric\ O.\ Endo},
       DOI = {10.1007/s00220-019-03606-1},
       URL = {https://doi.org/10.1007/s00220-019-03606-1},
}

\clearpage
\clearpage

\setcounter{section}{0}
\renewcommand{\thesection}{\Alph{section}}
\renewcommand{\theHsection}{appendix.\Alph{section}}

\setcounter{counterEnvMain}{0}
\setcounter{counterEnvDefault}{0}

\renewcommand{\thecounterEnvMain}{\thesection.\arabic{counterEnvMain}}
\renewcommand{\thecounterEnvDefault}{\thesection.\arabic{counterEnvDefault}}

\providecommand{\theHcounterEnvMain}{}
\providecommand{\theHcounterEnvDefault}{}

\renewcommand{\theHcounterEnvMain}{appendix.\Alph{section}.\arabic{counterEnvMain}}
\renewcommand{\theHcounterEnvDefault}{appendix.\Alph{section}.\arabic{counterEnvDefault}}

\section{Appendix}
\label{sec:appendix}

\subsection{Real continuity of the mass in a general direction}

We start by proving the continuity of $\frakm$ and $\fraky$, needed to establish equation \eqref{thm:corr_length_eigenvalue_Ruelle}.  
An elementary proof of the continuity of the rate of decay of the probability of connection to distance $n$ can be found in~\cite[Theorem 6.14]{grimmett}. However, the argument does not adapt directly to the case of the directional mass $\frakm$.
We thus present an easy argument using Theorem \ref{thm:clust_exp_UDE}.

\begin{lemma}\label{lem:mass_continuous}
The functions $\frakm(p,x)$ and $\fraky(p,x)$ are continuous on $(0,p_c) \times \bbR^d$.
\end{lemma}
\begin{proof}
The continuity of $\fraky$ follows directly from the continuity of $\frakm$ and the definition of~$\fraky$. 
The continuity of $\frakm$ in the space variable $x$ is a classical, as $\frakm$ is an equivalent norm on $\R^d$.
 Also notice that for all $p$, $\frakm(p,0)=0$. 
 So, it only remains to prove the continuity of $\frakm(\cdot, x)$ for a fixed vector $x \in \bbR^d \setminus \{0\}$.
 Let $p\in (0,p_c)$, and take $\eps > 0$ such that $p+\eps < p_c$.
For any $n \in \Z_{\geq 1}$ and for $K \in \Z_{\geq 1}$ to be specified latter, we write, 
\begin{equation}
\phi^0_{p+\eps}[ 0 \longleftrightarrow \lfloor nx \rfloor ] = \sum_{\substack {C \ni 0, \lfloor nx\rfloor\\ C \text{ is connected}  \\ |C| <Kn}} \phi^0_{p+\eps}[\calC_0 = C] + \phi_{p+\eps}[0 \longleftrightarrow \lfloor nx \rfloor, | \calC_0 | \geq Kn ]. 
\end{equation}
A rough upper bound on the exponential decay of the volume of the cluster of 0 in the subcritical phase implies that one may choose $K \geq 0$, depending only on $p$ and $x$, such that for all $\eps > 0$ sufficiently small
\begin{equation}
\phi^0_{p+\eps}[0 \longleftrightarrow \lfloor nx \rfloor, | \calC_0 | \geq Kn ] \leq \exp( -2\frakm(p, x) n).
\end{equation}
On the other hand, since $\{\calC_0 = C\}$ is a UD-event, Theorem \ref{thm:clust_exp_UDE} implies the existence\footnote{One can alternatively use~\cite[Theorem 1.2]{dalimonte_analyticity_2026} instead of Theorem~\ref{thm:clust_exp_UDE} to get the same result without using the notion of UD-event.} of a constant $c_\eps > 0$ tending to 0 when $\eps$ tends to 0 such that
\begin{align*}
\sum_{\substack {C \ni 0, \lfloor nx\rfloor\\ C \text{ is connected}  \\ |C| <Kn}} \phi^0_{p+\eps}[\calC_0 = C] 
&\leq \exp(+c_\eps K n) \sum_{\substack {C \ni 0, \lfloor nx\rfloor\\ C \text{ is connected}  \\ |C| <Kn}} \phi^0_{p}[\calC_0 = C] \\
&\leq \exp(+c_\eps K n) \phi_p[0 \leftrightarrow \lfloor nx \rfloor].
\end{align*}
Since the mass $\frakm$ is clearly decreasing in $p$, we obtain
\begin{equation}
0 \leq \frakm(p, x) - \frakm(p+\eps, x) \leq K c_\eps,
\end{equation}
which concludes on the continuity of $\frakm(\cdot, x)$ on $(0, p_c)$. 
\end{proof}
\subsection{A uniform mixing estimate}
\label{sec:appendix_mixing}

The proof presented in this paper makes several uses of mixing estimates for subcritical percolation. Either directly or by invoking results of~\cite{campanino_ornstein-zernike_2003} (especially for the results invoked in Remark \ref{remq:real_mass_gap}). The following lemma, is a generalised version of the classical ratio mixing estimate. Combining this lemma with exponential decay of connection probabilities in FK-percolation is enough to obtain all the mixing estimates used in the present paper and in the papers on which it relies.

Similar statements can be found in the literature \cite{aoun_two-point_2022, DCM_scaling_2022}, and follow from the seminal work of Alexander~\cite{alexander_weak_1998}. But for the purpose of this article, we state a ratio weak mixing for $\phi_{G,p}$ that is uniform over $G$. The proof presented below is a direct adaptation of \cite[Corollary 3.2]{aoun_two-point_2022}. The point here is simply that removing edges from $\bbZ^d$ by restricting the model to a subgraph $G$ may only decrease the probability of connections. And since mixing is controlled by these probabilities, restricting to subgraph do not alter the mixing bounds. 

\begin{lemma}[Ratio weak mixing in subgraphs]\label{lem:ratio_mixing_uniform_G}
Take $p < p_c$, a subgraph $G$ of $\bbZ^d$ and two events $A$ and $B$ with positive probability and whose supports are finite and disjoint. Let $S$ and $S'$ denote the support of $A$ and $B$, respectively. Then, the following bound holds whenever the right-hand side is less than $1$
\begin{equation}
\Big\vert \frac{\phi^0_{G,p}[A \,|\, B]}{\phi^0_{G,p}[A]} - 1 \Big\vert \leq 4 \phi^1_{\bbZ^d \setminus \{S \cup S'\}}[\partial S \leftrightarrow \partial S'].
\end{equation}
\end{lemma}

\begin{proof}
We essentially follow the proof of \cite[Corollary 3.2]{aoun_two-point_2022} the main difference being that we work on the subgraph $G$ instead of $\bbZ^d$. We fix $p < p_c$ and omit the parameter $p$ from the FK measures to shorten notation. Let $\eps(S,S') \coloneqq \phi^1_{\bbZ^d \setminus \{ S \cup  S'\}}[\partial S \leftrightarrow \partial S']$ and suppose that $4 \eps(S,S') \leq 1$.
It is actually sufficient to prove that, for any $\eta \in \{0,1\}^S$ and $\eta' \in \{0,1\}^{S'}$,
\begin{equation}\label{eq:ratio_mixing_configuration}
\Big\vert \frac{\phi^0_G[\om_{|S} = \eta \,|\, \om_{|S'} = \eta']}{\phi^0_G[\om_{|S} = \eta]} - 1 \Big\vert \leq 4 \eps(S,S').
\end{equation}
Indeed, summing this bound over the possible configurations on $S$ and $S'$ realising $A$ and~$B$ respectively yields the required bound. The bound above will follow from the following.

\medskip

\textbf{Claim:} For any $\eta \in \{0,1\}^S$ and any $\rho, \sigma \in \{0,1\}^{S'}$ such that\footnote{$\rho \geq \sigma$ means that $\forall s \in S',\, \rho_s \geq \sigma_s$.} $\rho \geq \sigma$, 
\begin{align*}
(1-\eps(S,S')) \leq \frac{\phi^0_G[\om_{|S} = \eta \,|\, \om_{|S'} = \rho]}{\phi^0_G[\om_{|S} = \eta \,|\, \om_{|S'} = \sigma]} \leq (1-\eps(S,S'))^{-1}.
\end{align*}
Let us first prove this claim. By the FKG-lattice property of FK-percolation, the measure $\phi^0_G[\cdot\,|\, \om_{|S'} = \rho]$ stochastically dominates the measure $\phi^0_G[\cdot\,|\, \om_{|S'} = \sigma]$ and it is classical that one may find a coupling $\bbP$ between $\om^+ \sim \phi^0_G[\cdot\,|\, \om_{|S'} = \rho]$ and $\omega^- \sim \phi^0_G[\cdot\,|\, \om_{|S'} = \sigma]$ such that $\om^+ \geq \om^-$ almost surely. The coupling techniques described in \cite[Section 2.3]{DCM_scaling_2022} also show that $\bbP$ can be chosen\footnote{When $S$ and $S'$ are not connected in $\om^+$, one can explore the connected component of $S'$ in $\om^+$ to obtain a separating dual surface between $S'$ and $S$. Since $\om^+ \geq \om$, the same dual surface is also present in $\om$. Once this surface is revealed in $\om^+$ and $\om^-$, the rest of the exploration of $\om^+$ and $\om^-$ can be carried out while requiring that the configurations $\om^+$ and $\om^-$ are equal in the volume delimited by the surface.}
in such a way that $\bbP$-a.s., if $\partial S$ and $\partial S'$ are not connected in $\om^+$, then the two configurations $\om^+$ and $\om^-$ coincide on $S$.

By taking such a coupling $\bbP$, we get 
\begin{align*}
\phi^0_G[\om_{|S} = \eta \,|\, \om_{|S'} = \rho] - \phi^0_G[\om_{|S} = \eta \,|\, \om_{|S'} = \sigma]
&\leq \bbP[\omega^+_{|S} = \eta,\, \omega^-_{|S} \neq \eta]\\
&\leq \bbP[\omega^+_{|S} = \eta,\, \partial S \overset{\om^+}{\longleftrightarrow} \partial S']\\
&\leq \bbP[\partial S \overset{\om^+}{\longleftrightarrow} \partial S' \,|\, \omega^+_{|S} = \eta] \phi_G^0[\omega_{|S} = \eta \,|\, \omega_{|S'} = \rho]\\
&\leq \eps(S,S') \phi_G^0[\omega_{|S} = \eta \,|\, \omega_{|S'} = \rho],
\end{align*}
where the last inequality follows from the monotonic properties of the model and the fact that $G$ is a subgraph of $\bbZ^d$. Dividing by $\phi_G^0[\omega_{|S} = \eta \,|\, \omega_{|S'} = \rho]$ provides the upper bound of the \textbf{claim}.

To prove the lower bound, we write $S_{\star} = \{e \in S: \, \eta_e = \star\}$ for $\star \in \{0,1\}$. Decomposing the event $\om_{|S} = \eta$ and using the FKG-lattice property yields
\begin{align}\label{equ:bound_decomposition_eta}
\frac{\phi^0_G[\om_{|S} = \eta \,|\, \om_{|S'} = \rho]}{\phi^0_G[\om_{|S} = \eta \,|\, \om_{|S'} = \sigma]}
&= \frac{\phi^0_G[\om_{|S_1} = 1 \,|\, \om_{|S'} = \rho] \phi^0_G[\om_{|S_0} = 0 \,|\, \om_{|S_1} = 1,\, \om_{|S'} = \rho] }{\phi^0_G[\om_{|S_1} = 1 \,|\, \om_{|S'} = \sigma] \phi^0_G[\om_{|S_0} = 0 \,|\, \om_{|S_1} = 1,\, \om_{|S'} = \sigma]} \notag\\
&\geq \frac{\phi^0_G[\om_{|S_0} = 0 \,|\, \om_{|S_1} = 1,\, \om_{|S'} = \rho] }{ \phi^0_G[\om_{|S_0} = 0 \,|\, \om_{|S_1} = 1,\, \om_{|S'} = \sigma]}.
\end{align}
For any configuration $\om$ on $G$, we denote by $\Omega(\om)$ the connected component of $S'$ in $\om$. Observing what are the possible realisations of $\Omega(\om)$ when $\partial S \overset{\om}{\centernot\longleftrightarrow} \partial S'$, we get
\begin{align*}
\phi^0_G[\om_{|S_0} = 0 \,|\, \om_{|S_1} = 1,\, \om_{|S'} = \rho]
&\geq \phi^0_G[\om_{|S_0} = 0,\, \partial S \overset{\om}{\centernot\longleftrightarrow} \partial S' \,|\, \om_{|S_1} = 1,\, \om_{|S'} = \rho]\\
&\geq \sum\limits_{\Omega} \phi^0_{G \setminus \Omega}[\om_{|S_0} = 0 \,|\, \omega_{S_1} = 1] \phi_G[\Omega(\omega) = \Omega \,|\, \om_{|S_1} = 1,\, \om_{|S'} = \rho],
\end{align*}
where the sum runs over the subgraphs $\Omega$ of $G$ such that $S' \subset \Omega$ and $\Omega \cap S = \emptyset$. Note that, for any possible choice of $\Omega$, the monotonic properties of FK-percolation imply that 
\begin{align*}
\phi^0_{G \setminus \Omega}[\om_{|S_0} = 0 \,|\, \omega_{S_1} = 1] 
\geq \phi^0_G[\om_{|S_0} = 0 \,|\, \om_{|S_1} = 1,\, \om_{|S'} = \sigma].
\end{align*}
Thus,
\begin{align*}
\frac{\phi^0_G[\om_{|S_0} = 0 \,|\, \om_{|S_1} = 1,\, \om_{|S'} = \rho]}{\phi^0_G[\om_{|S_0} = 0 \,|\, \om_{|S_1} = 1,\, \om_{|S'} = \sigma]}
&\geq \phi^0_G[ \partial S \overset{\om}{\centernot\longleftrightarrow} \partial S' \,|\, \om_{|S_1} = 1,\, \om_{|S'} = \rho]\\
&\geq 1 - \eps(S,S').
\end{align*}
Inserting this in \eqref{equ:bound_decomposition_eta}, proves the lower bound of the \textbf{claim}.

\medskip

We can now prove the bound \eqref{eq:ratio_mixing_configuration}. Take $\eta$ and $\eta'$ as in \eqref{eq:ratio_mixing_configuration}. Let us rewrite the ratio
\begin{align*}
\frac{\phi^0_G[\om_{|S} = \eta]}{\phi^0_G[\om_{|S} = \eta \,|\, \om_{|S'} = \eta']}
= \sum\limits_{\tau \in \{0,1\}^{S'}} \phi^0_G[\omega_{|S'} = \tau] \frac{\Phi(\eta|\tau)}{\Phi(\eta|0)} \frac{\Phi(\eta|0)}{\Phi(\eta|\eta')},
\end{align*}
where we used the shorthand $\Phi(\eta|\star) = \phi^0_G(\om_{|S} = \eta \,|\, \om_{S'} = \star)$. Applying the \textbf{claim} to the two fractions of $\Phi$-terms on the right-hand side gives
\begin{align*}
(1 - \eps(S,S'))^{2} \leq \frac{\phi^0_G[\om_{|S} = \eta]}{\phi^0_G[\om_{|S} = \eta \,|\, \om_{|S'} = \eta']} \leq (1 - \eps(S,S'))^{-2}.
\end{align*}
Since $4\eps(S,S') \leq 1$, the upper bound above is actually smaller than $1 + 4 \eps(S,S')$. The bound \eqref{eq:ratio_mixing_configuration} follows directly and this concludes the proof.
\end{proof}

\subsection{Analyticity and polar duality}
\label{sec:legendre}

In order to transmit the analyticity of the support function $\frakg$ to its dual $\frakm$, we use the following lemma.

\begin{lemma} \label{lem:dual_analytic}
Let $I$ be an interval of $\R$ and let $(\calK_p)_{p \in I}$ be a family of convex bodies of $\R^d$ containing 0 in their interior such that for every $p \in I$, $\partial\calK_p$ is strictly convex, analytic and has strictly positive Gaussian curvature. Call $\fraks(p, \cdot)$ and $\frakg(p, \cdot)$ the support function and gauge function of $\calK_p$, respectively. 
Then the following are equivalent:
\begin{itemize}
\item $(p, x) \in I\times\R^{d}\setminus\{0\} \mapsto \fraks(p,x)$ is jointly analytic 
\item $(p, x) \in I\times\R^{d}\setminus\{0\} \mapsto \frakg(p,x)$ is jointly analytic 
\end{itemize}
\end{lemma}

\begin{proof}
This is a mere consequence of the implicit functions theorem. Let us assume that $\frakg$ is analytic and prove that $\fraks$ is also analytic.
The converse implication will follow directly by the fact that polar duality is an involution in the case of strictly convex bodies.

By definition,  
\begin{equation}\label{equ:def_support}
\fraks(p,x) = \max_{y \in \partial \calK_p} \langle x,y \rangle.
\end{equation}
Call $y(p, x)$ the maximizer in the last display, which is unique by strict convexity. 
Classical duality considerations imply that the maximum is realized when $x$ is normal to $\partial \calK_p$ at~$y(p,x)$.
The normal vector of  $\partial \calK_p$ at $y$ is $\nabla_y \frakg(p,y)$, which implies that $y(p, x)$ solves the following system (of unknown $(y,\lambda) \in \R^d \times \R$):
\begin{equation}
\begin{cases}
&\nabla_y \frakg(p, y) = \lambda x \\ &\frakg(p,y) = 1.
\end{cases}
\end{equation}

Define a function $F$ as follows:
\begin{equation}
F(p,x, y, \lambda) \coloneqq (\nabla_y \frakg(p, y) - \lambda x, \frakg(p,y) - 1).
\end{equation}
By assumption, $F$ is analytic, and we want to prove that the first coordinate of the solution $(y(p,x), \lambda(p,x)))$ to the equation $F(p, x, y(p,x), \lambda(p,x)) = (0,0)$ is analytic. 
To apply the analytic version of the implicit function theorem, we compute the Jacobian
\begin{equation}
D_{(y,\lambda)}F =
 \begin{pmatrix}
D^2_{yy} \frakg(p,y) & -x \\
\nabla_y\frakg(p,y)^{T} & 0 
\end{pmatrix} .
\end{equation}
As the minimizer satisfies $\nabla_y \frakg(p, y) = \lambda x$, the Jacobian at the minimizer reduces to 
\begin{equation}
D_{(y,\lambda)}F (y(p,x), \lambda(p,x)))=
 \begin{pmatrix}
D^2_{yy} \frakg(p,y) & -x \\
\lambda x^T & 0 
\end{pmatrix} .
\end{equation}

It remains to show that this matrix is invertible around $(p,x)$, when $x \neq 0$. 
This is formally a consequence of the positivity of the curvature of $\calK_p$ at the point $y \in \partial \calK_p$; we prove it in the ``pedestrian'' way by showing that its kernel is reduced to $\{0\}$.

Consider a pair $(u, \alpha) \in \R^d \times \R$ such that $D_{(y,\lambda)}F \begin{pmatrix}
u\\
\alpha
\end{pmatrix} 
 = 0$.
This yields the two equations
\[ 
\begin{cases}
D^2_{yy} \frakg(p,y) u - \alpha x = 0 \\
\langle x,u \rangle = 0.
\end{cases}
 \]
 
Multiplying the first line by $u ^T$ yields the equation $u^T D^2_{yy} \frakg(p,y) u - \alpha \langle u, x\rangle  = 0$, but $\langle u, x\rangle  = 0$ by the second line. 
Hence, we obtain  $u^T D^2_{yy} \frakg(p,y) u = 0$.
Remember that $\nabla_y \frakg(p, y) = \lambda x$. 
Thus, that $\langle x,u \rangle = 0$ implies that $u$ is tangent to $\calK_p$ at the point $y$. 
By the positive curvature condition, the quadratic form $u \mapsto u^T D^2_{yy} \frakg(p,y) u $ is positive definite on the set of tangent vectors to $\calK_p$. 

We obtain that $u = 0$, and that $\alpha = 0$.
Hence, the Jacobian of $F$ is invertible. 
By the implicit function theorem, we obtain that the solution $(y(p,x), \lambda(p,x)))$ of the equation $F(p, x, y(p,x), \lambda(p,x))=0$ is jointly analytic

Finally, inserting this input in the formula~\eqref{equ:def_support} implies that $\fraks$ is jointly analytic, which concludes the proof of the lemma.  
\end{proof}

\subsection{A sequence of local events with complex probabilities having accumulating zeroes}
\label{sec:counter_example}

Here, we give an example of a sequence of events $(A_n)_{n \geq 0}$ satisfying the following two properties:
\begin{enumerate}
\item For any $n \geq 0$, the event $A_n$ is local and $\phi_p[A_n] \neq 0$.
\item There exist $q \geq 1$, a given $p < p_c(q) $  and a sequence $z_n$ of complex numbers tending to 0 when $n$ tends to infinity such that, when $n$ is sufficiently large,
\begin{equation}
\phi_{p+z_n}[A_n] = 0,
\end{equation}
where $\phi_{p+z_n}[A_n]$ denotes the complex extension of $\phi_p[A_n]$ constructed in~\cite{dalimonte_analyticity_2026}.
\end{enumerate}

In particular, this shows that Theorem~\ref{thm:clust_exp_UDE} cannot be extended to general local events. This also means that the cluster expansion of~\cite{ott_ising,ott_weak_2020-1} cannot be done uniformly on the local event under consideration. Indeed, if it were the case, the cluster expansion theorem would ensure that the complex probabilities of any local event cannot vanish in a uniform neighbourhood of the subcritical parameter $p$, which contradicts the above-mentioned construction. This provides an a priori explanation for the need of the additional argument of~\cite{dalimonte_analyticity_2026} as compared with~\cite{ott_ising,ott_weak_2020-1}.

\vspace{15mm}
{\noindent \bf Construction.}\\

For sake of simplicity\footnote{It should be clear that the construction can be extended beyond this particular setting.}, we choose $d = 2$, $p = 1/2$ and take $q > 1$. Observe that $p < p_c(q)$ since $p_c(q) = \tfrac{\sqrt q}{1+\sqrt q}$. Let $n < N$ be two integers, and define the following sets of edges (we refer to Figure~\ref{fig:UME_event_B} for illustration). 
\begin{itemize}
\item For $k \in \{1, n\}$, $a_k \coloneqq \{ ( k, 0 ), (k,1) \}$. Those are the red edges of Figure~\ref{fig:UME_event_B}.
\item $\calB_n$ is the set of edges bordering the rectangle surrounding the red edges. Those are the black edges of Figure~\ref{fig:UME_event_B}.
\item Finally, $\calC_n$ is the set of edges sharing an endpoint with the edges of $\calB_n$ and pointing towards the outside. They are depicted in blue in Figure~\ref{fig:UME_event_B}
\end{itemize}

Now, for $1 \leq k \leq n$, denote by $B_{k,n}$ the event that 
\begin{itemize}
	\item The edges $a_1, \dots, a_k$ are open and $a_{k+1}, \dots, a_n$ are closed.
	\item The edges of $\calB_n$ are open.
	\item The edges of $\calC_n$ are closed. 
\end{itemize}
Finally, set $A_n \coloneqq B_{1,n} \cup \cdots \cup B_{n,n} $.

Then, the event $A_n$ enjoys the locality and decoupling properties of UD-events (see Def.~\ref{def:UDE}), but not the uniqueness property. 
As a consequence, it follows that when $N \geq n$, there exists a function $\frakf_{N,n}$, well-defined on a neighbourhood of 0, such that one can write 
\begin{equation} 
Z_{\Lambda_N,p+z} \phi_{\Lambda_N,p+z}[A_n] = \frakf_{N,n}(z) \sum_{k=1}^{n} a_{z}^k,
\end{equation}
where $a_{z} = (p+ z)(1 - (p+z))^{-1}$.

\begin{figure}
\centering
\includegraphics[width = 0.9\textwidth]{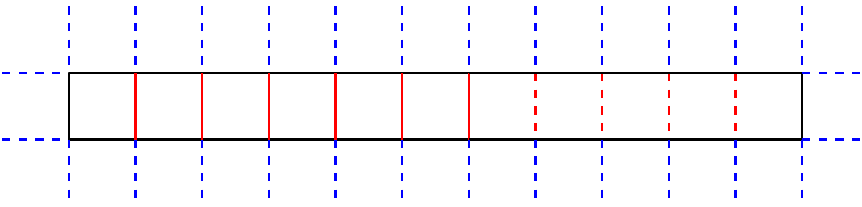}
	\caption{
	The event $B_{n,k}$ occurs if the edges are in the state prescribed by the picture. Plain edges are open whereas dashed edges are closed. The rectangle delimited by the black open edges contains $n$ edges shown in red. The $k$ first red edges are open and the $n-k$ remaining edges are closed. The closed blue edges surrounding the rectangle ``shield'' it so that $B_{n,k}$ lies in the class of UM-events.	 }
	\label{fig:UME_event_B}	
\end{figure}

Now, using a classical trigonometric formula, we observe that the choice $p=1/2$ 
and $z_n \coloneqq \frac{\mathrm{i}}{2}\tan(\frac\pi n)$ 
yields $\alpha_{z_n} = \e^{\frac{2\mathrm{i}\pi}{n}}$.

This in turns implies that $\phi_{ p+z_n}[A_n] = 0$.
It is straightforward that $z_n$ tends to 0 when $n$ tends to infinity, which concludes the announced construction of the counterexample. 

\vspace{5mm}

\end{document}